\documentclass[12pt]{amsart} 
\usepackage{amssymb} 
\usepackage{amsfonts} 
\usepackage{amsmath} 
\usepackage{amsthm} 
\usepackage{array} 
\usepackage{geometry} 
\usepackage{mathrsfs} 
\usepackage{pifont} 
\usepackage[all,ps,cmtip]{xy} 

\theoremstyle{plain} 
\newtheorem{thm}{Theorem}[subsection] 
\newtheorem{cor}[thm]{Corollary} 
\newtheorem{lem}[thm]{Lemma} 
\newtheorem{prop}[thm]{Proposition} 

\theoremstyle{definition} 
\newtheorem{defn}[thm]{Definition} 
\newtheorem{eg}[thm]{Example} 
\newtheorem{obs}[thm]{Observation} 
\newtheorem{rem}[thm]{Remark} 
\newtheorem{constr}[thm]{Construction}
\newtheorem{con}[thm]{Conventions}
\theoremstyle{remark} 
\numberwithin{equation}{subsection}
\newcommand{\N}{\mathbb{N}} 
\newcommand{\Z}{\mathbb{Z}}     
\newcommand{\Q}{\mathbb{Q}} 
\newcommand{\C}{\mathbb{C}} 
\newcommand{\G}{\mathbb{G}}
\newcommand{\BL}{\textbf{L}}
\newcommand{\sA}{\mathcal{A}}
\newcommand{\sB}{\mathcal{B}}

\newcommand{\sD}{\mathcal{D}}

\newcommand{\sI}{\mathcal{I}}
\newcommand{\sK}{\mathcal{K}}
\newcommand{\sL}{\mathcal{L}}
\newcommand{\sM}{\mathcal{M}}
\newcommand{\sN}{\mathcal{N}}
\newcommand{\sO}{\mathcal{O}}

\newcommand{\sS}{\mathcal{S}}

\newcommand{\sX}{\mathcal{X}}

\newcommand{\fM}{\mathfrak{M}}
\newcommand{\fB}{\mathfrak{B}}
\newcommand{\fC}{\mathfrak{C}}
\newcommand{\CharM}{\overline{\mathcal{M}}}
\newcommand{\Sch}{\mathfrak{Sch}}
\newcommand{\LogSch}{\mathfrak{LogSch}}
\newcommand{\Spec}{\operatorname{Spec}}

\newcommand{\lra}{\longrightarrow}
\begin{document}
\title{Stable logarithmic maps to Deligne--Faltings pairs I}

\author[Qile Chen]{Qile Chen}

\thanks{Research partially supported by funds from NSF award DMS-0901278.}

\address{Qile Chen, Department of Mathematics, Box 1917, Brown University,
Providence, RI, 02912, U.S.A} \email{q.chen@math.brown.edu}

\maketitle \setcounter{tocdepth}{1} \tableofcontents
%
%

%
\section{Introduction}

\subsection{Background on relative Gromov-Witten theory}
Gromov-Witten theory relative to a smooth divisor was introduced during the past decade, for the purpose of proving a degeneration formula, in the symplectic setting by A.M. Li and Y. Ruan \cite{LR}, and about the same time by E.N. Ionel and T. Parker \cite{IP1}. On the algebraic side, this was worked out by Jun Li \cite{Jun1,Jun2}. This approach uses the idea of expanded degenerations, which was introduced by Ziv Ran \cite{Ra}. A related idea of admissible covers was introduced even earlier by Harris and Mumford \cite{HM}. 

Recently, the idea of expanded degeneration was systematically studied using orbifold techniques by D. Abramovich and B. Fantechi \cite{AF}, and an elegant proof of degeneration formula was given there. 

On the other hand, the idea of admissible covers was revisited by Mochizuki \cite{Mochizuki} using logarithmic geometry. Along the similar idea, B. Kim defined the logarithmic stable maps \cite{Kim}, by putting certain log structures on Jun Li's predeformable maps. Then he proved that the stack parameterizing such maps is a proper DM stack, and has an explicit perfect obstruction theory using the work of M.Olsson \cite{LogStack,LogCot}. A degeneration formula under Kim's setting is proved in \cite{LogDeg}. 

Another approach using logarithmic structures without expansions was first proposed by Bernd Siebert in 2001 \cite{Siebert}. The goal here
is also to obtain the degeneration formula, but in a more general situation, such as simple normal crossings divisors. However, the program has been on hold for a while, since Mark Gross and Bernd Siebert were working on other projects in mirror symmetry. Only recently they have taken up the unfinished project of Siebert jointly. In particular, they succeeded in finding a definition of basic log maps, a crucial ingredient for a good moduli theory of stable log maps to a fixed target with Zariski log structures \cite{GS2}. Their definition builds on insights from tropical geometry, obtained by probing the stack of log maps using the standard log point, and is compatible with the minimality introduced in this paper.


A different approach using exploded manifolds to studying holomorphic curves was recently introduced by Brett Parker in \cite{Parker1}, \cite{Parker3}, and \cite{Parker2}. It also aimed at defining and computing relative and degenerated Gromov-Witten theories in general situation. The theory of exploded manifolds uses a generalized version of tropical curves, and is closely related to logarithmic geometries --- the explosion functors, which is central to this theory, can be phrased in terms of certain kind of base change in log geometry, see \cite[Section 5]{Parker1}. It was pointed out by Mark Gross that this approach is parallel, and possibly equivalent to the logarithmic approach.

\subsection{The approach and main results of this paper}
The goal of this paper is to develop the relative Gromov-Witten theory along the logarithmic approach proposed by Bernd Siebert. However, we use somewhat different methods. Instead of using tropical geometry, and probing the stack using standard log point, we associated to each log map a marked graph as in Section \ref{ss:AdGraph}, which allows us to define the right base log structure. We now describe our methods as follows.

The target we will consider in this paper is a projective variety $X$ equipped with a rank-$1$ Deligne-Faltings log structure $\sM_{X}$ on $X$ as in Definition \ref{defn:Target}, which comes from a line bundle $L$ on $X$, with a morphism of sheaves $s:L\to \sO_{X}$. In particular, if $L$ is the ideal sheaf of a smooth divisor $D\subset X$, and $s$ is the natural inclusion, then this will recover the relative case. See Section \ref{ss:DFLog} for more details on DF log structures. Denote by $X^{log}=(X,\sM_{X})$ the corresponding log scheme. Instead of considering usual stable maps to the expansions of $X$, we investigate the usual stable maps to $X$, with fixed log structure $\sM_{X}$ on the target, and suitable log structures on the source curves. 

In the subsequent paper \cite{AC}, we will consider target $X$ with generalized Deligne-Faltings log structure $\sM_{X}$, namely where there exists a global map $P\to \CharM_{X}$ from a toric monoid $P$, which locally lifts to a chart. In particular, this covers many interesting cases, such as a variety with a simple normal crossings divisor, or a simple normal crossings degeneration of a variety with simple normal crossings singularities. It does not cover the case of a normal crossings divisor which is not simple. We hope one can also cover this using the descent argument along this approach. 

A key point of this paper is the observation made in Section \ref{ss:LogMapChar}, which describes the log map on the level of characteristic monoids. This leads us to the notions of marked graphs \ref{defn:Orientation} and minimality \ref{defn:Minimal}. Such minimality condition can be explained as the ``minimal requirements'' which a log map needs to satisfy. Then our minimal stable log maps are defined to be usual stable maps with the minimal log structures. Denote by $\sK_{\Gamma}(X^{log})$ the category fibered over the category of schemes, which for any scheme $T$ associates the groupoid of minimal stable log maps over $T$ with numerical data $\Gamma$. We refer to Section \ref{ss:LogStableMap} for the precise meaning of $\sK_{\Gamma}(X^{log})$. The main result of this paper is the following:

\begin{thm}\label{thm:Main}
The fibered category $\sK_{\Gamma}(X^{log})$ is a proper Deligne-Mumford stack. Furthermore, the natural map $\sK_{\Gamma}(X^{log})\to \sK_{g,n}(X,\beta)$ by removing the log structures from minimal stable log maps is representable and finite.
\end{thm}

\begin{rem}
In fact, the stack $\sK_{\Gamma}(X^{log})$ carries a universal minimal log structure, which will be denoted by $\sM_{\sK_{\Gamma}(X^{log})}$. Thus the pair $(\sK_{\Gamma}(X^{log}),\sM_{\sK_{\Gamma}(X^{log})})$ can be viewed as a log stack in the sense of Olsson, see Section \ref{s:LogStack}. By applying the standard technique in \cite{Behrend-Fantechi}, and replacing the usual cotangent complex by logarithmic cotangent complex \cite{LogCot}, one can produce a perfect obstruction theory of $\sK_{\Gamma}(X^{log})$ relative to $\fM^{pre}_{g,n}$, the stack of log prestable curves defined in Section \ref{ss:LogCurveStack}. We will discuss this perfect obstruction theory and the corresponding virtual fundamental class in another paper.
\end{rem}

Up to now, we only introduce $\sK_{\Gamma}(X^{log})$ as category fibered over $\Sch$, the category of schemes. Denote by $\LogSch^{fs}$ the category of fine and saturated log schemes as introduced in Section \ref{ss:BacisLog}. The following result exhibits another important aspect of our construction:

\begin{thm}\label{thm:Main2}
The pair $(\sK_{\Gamma}(X^{log}),\sM_{\sK_{\Gamma}(X^{log})})$ defines a category fibered over $\LogSch^{fs}$, which for any fs log scheme $(S,\sM_{S})$ associates the category of stable log maps over $(S,\sM_{S})$.
\end{thm}

The above categorical interpretation allows us to work systematically with fs log schemes rather than usual schemes. This point of view will be a useful tool in \cite{AC}, where we reduce the case with generalized DF-log structure on the target to the case of this paper.

\subsection{Outline of the paper}
In Section \ref{sec:StackAlg}, we fix a morphism of fs log schemes $X^{log}\to B^{log}$, define an auxiliary category $\sL\sM_{g,n}(X^{log}/B^{log})$ of all log maps with target X, fibered over schemes, and show that it is an algebraic stack in the sense of Artin. This stack is unbounded and serves mainly as a construction technique. This will be achieved by verifying Artin's criteria \cite[5.1]{Artin}. Here the deformation theory of our log maps will be given by the log cotangent complex developed in \cite{LogCot}.

Section \ref{sec:LogMap} is aimed at the construction of minimal log maps. In fact, for each log map over a geometric point with fs log structure, we can associate a marked graph, see Construction \ref{cons:DualGraphMap}. These graphs will be used to describe the minimality condition. Then we show that minimality is an open condition, see Proposition \ref{prop:MinOpen}. This implies the algebricity of the stack of minimal log maps using the results of Section \ref{sec:StackAlg}. In Section \ref{sec:finiteness} we show that there are at most finitely many minimal stable log maps over a fixed underlying stable map with a fixed marked graph. The finiteness of automorphisms of minimal stable log maps over a geometric point is proved in Proposition \ref{prop:FiniteAuto}. 

Section \ref{sec:UniverMin} is devoted to proving Theorem \ref{thm:Main2}. This will follow naturally from the universal property of minimal log maps in Proposition \ref{prop:UnivMinLog}.

In Section \ref{sec:Boundedness}, we will show that $\sK_{\Gamma}(X^{log})$ is of finite type by stratifying the stack, and bounding the stratum associated to each marked graph. Indeed, we will prove the boundedness of $\sK_{\Gamma}(X^{log})$ relative to the stack of usual stable maps. One issue here is in constructing all maps of log structures for a given graph. We will turn this into constructing isomorphisms of corresponding line bundles.

The weak valuative criterion of $\sK_{\Gamma}(X^{log})$ is proved in Section \ref{sec:Valuative}. In fact, the universal property of minimality produces an extension of \textit{minimal} stable log map, once we can find \textit{any} extension of stable log maps, not necessarily minimal. For separateness, we first show that the marked graph is uniquely determined by the generic fiber. Then we introduce a new map $\bar{\beta}$ as in (\ref{equ:SpeDegMap}), which helps us compare any two possible extensions, see Lemma \ref{lem:ChartDif}. This leads us to the uniqueness of the extension. In the end, we give a proof of Theorem \ref{thm:Main}, and show that the stack of minimal stable log maps is representable and finite over the stack of usual stable maps.

Finally, we have two appendices collecting various results of logarithmic geometry and logarithmic curves, as we need in this paper. The notion of log pre-stable curve is introduced in Definition \ref{LogPreCurve}.

\subsection{Conventions}
Throughout this paper, we will work over an algebraically closed field of characteristic $0$, denoted by $\C$.

The word locally always means \'etale locally, and neighborhood always means \'etale neighborhood, unless otherwise specified.

Given a scheme or algebraic stack $X$, a point $p\in X$ means a closed point unless otherwise specified. We denote by $\bar{p}$ an algebraic closure of $p$.

We usually use $X^{log}$ to denote a log scheme $(X,\sM_{X})$ if no confusion could arise. The map $\exp:\sM_{X}\to \sO_{X}$ is reserved for the structure map of $\sM_{X}$. Given a section $u\in \sO_{X}$, we sometimes use $\log u$ to denote the corresponding section in $\sM_{X}$ if no confusion could arise. 

The category of schemes, fine log schemes, and fs log schemes are denoted by $\Sch$, $\LogSch$, and $\LogSch^{fs}$ respectively. See Section \ref{ss:BacisLog} for the precise definitions.

The letter $\xi$ (respectively $\xi^{log}$) is reserved for log maps over a usual scheme (respectively log scheme). Given a log map $\xi=(C\to S,\sM_{S},f)$ over $S$ as in Remark \ref{rem:LogMapIso}, we will denote by $\sM_{C}$ the corresponding log structure on $C$, if no confusion could arise.

\subsection{Acknowledgements}
I am very grateful to my advisor Dan Abramovich, who has suggested several interesting problems to me, including the main problem of this paper. He has given me continuous support and encouragement, and his suggestions have greatly influenced the shape of this paper. I am grateful to Mark Gross, Bernd Siebert, and Martin Olsson who pointed out several mistakes in the previous draft of this paper, and also gave many nice suggestions and comments. I would also like to thank Davesh Maulik, William Gillam, and Jonathan Wise for their helpful conversations. 

\section{Algebricity of the stack of log maps}\label{sec:StackAlg}
In this section, we prove that the stack $\sL\sM_{g,n}(X^{log}/B^{log})$ parameterizing log maps as in Definition \ref{defn:GeneralMapStack}, is algebraic by checking Artin's criteria \cite[5.1]{Artin}. The result in this section is only used to prove the stack of minimal stable log maps $\sK_{n,g}^{}(X^{log},\beta)$ as in Definition \ref{defn:MinStableStack} is algebraic. The reader may wish to assume the result of Theorem \ref{thm:AlgStackLogMap}, and skip to the next section.

\subsection{Log maps over $\LogSch$ and over $\Sch$}\label{ss:SetNote}

\begin{con}
In this section, we fix a separated, finite type, log flat morphism of log schemes $\pi: X^{log}\to B^{log}$. Denote by $B$ and $X$ the underlying schemes of $B^{log}$ and $X^{log}$ respectively. Let $\sM_{B}$ and $\sM_{X}$ be the log structure on $B^{log}$ and $X^{log}$ respectively. Given any $B$-scheme $S$, denote by $(X_{S},\sM_{X_{S}}^{X_{S}/S})\to (S,\sM^{X_{S}/S}_{S})$ the pull-back of $X^{log}\to B^{log}$ over $S$.
\end{con}

As an analogue of usual pre-stable maps, we introduce our log maps over log schemes as follows:

\begin{defn}\label{defn:LogMapOverLog}
A log map $\xi^{log}$ over a fine $B^{log}$-log scheme $(S,\sM_{S})$ is a commutative diagram of log schemes:
\begin{equation}\label{diag:LogMapOverLog}
\xymatrix{
(C,\sM_{C}) \ar[rr]^{f} \ar[rd] && (X_{S},\sM_{X_{S}}) \ar[ld]\\
 & (S,\sM_{S}) &
}
\end{equation}
where 
\begin{enumerate}
 \item The arrow $(C,\sM_{C})\to (S,\sM_{S})$ is given by the log curve $(C\rightarrow S,\sM_{S}^{C/S}\to\sM_{S})$ as in Definition \ref{DefLogC2}.
 \item The arrow $(X_{S},\sM_{X_{S}})\to (S,\sM_{S})$ is obtained from the following cartesian diagram of log schemes:
 \[
 \xymatrix{
 (X_{S},\sM_{X}) \ar[r] \ar[d] & X^{log} \ar[d]^{\pi}\\
 (S,\sM_{S}) \ar[r] & B^{log}.
 }
 \]
\end{enumerate}

Given a log map of fine log schemes $g:(T,\sM_{T})\to (S,\sM_{S})$, we define the pull-back $g^{*}\xi^{log}$ to be the log map $\xi^{log}_{T}$ over $(T,\sM_{T})$ obtained by pulling back (\ref{diag:LogMapOverLog}) via the log map $g$.
\end{defn}

The above definition gives a category of log maps fibered over $\LogSch_{B^{log}}$, the category of fine log schemes over $B^{log}$. This category allows pull-back via arbitrary log maps, hence changes the base log structures. In another word, it only parametrizes the ``log maps'', without the information of log structures on the base. This is the category of most interest to us. 

However, algebraic stacks are built over the category of schemes, rather than the category of log schemes. In order to have the algebraic structure, we need to introduce another fibered category over $\Sch_{B}$, the category of $B$-schemes. This leads to the following definition:

\begin{defn}\label{defn:LogMap}
A log map $\xi$ over a $B$-scheme $S$ consists of a fine log scheme $(S,\sM_{S})$ over $B^{log}$, and a log map $\xi^{log}$ over $(S,\sM_{S})$ as in Definition \ref{defn:LogMapOverLog}. Usually we denote it by 
\[\xi=(C\to S,X_{S}\to S, \sM_{S}^{X_{S}/S}\to\sM_{S}, \sM_{S}^{C/S}\to\sM_{S}, f),\]
where $\sM^{X_{S}/S}_{S}$ is the pull-back of $\sM_{B}$ via the structure map $S\to B$.

Consider another $B$-scheme $T$, and a $B$-scheme morphism $T\to S$. Then we have an induced strict arrow $g:(T,\sM_{T})\to (S,\sM_{S})$, where $\sM_{T}:=g^{*}(\sM_{S})$. The pull-back $\xi_{T}$ of $\xi$ via $T\to S$ is given by the log scheme $(T,\sM_{T})$, and the log map $\xi^{log}_{T}=g^{*}\xi^{log}$ over $(T,\sM_{T})$. In the rest of this section, if no confusion could arise, we will use $(C\to S,X_{S}\to S, \sM_{S}, f)$ to denote the log map $\xi$ over $S$.
\end{defn}

Since isomorphisms are central to the structure of stacks, we spell out the resulting notion of an isomorphism of log maps over schemes:

\begin{defn}\label{defn:LogMapIso}
Consider two log maps $\xi_{1}=(C_{1}\to S,X_{S}\to S, \sM_{1}, f_{1})$ and $\xi_{2}=(C_{2}\to S,X_{S}\to S, \sM_{2}, f_{2})$ over a scheme $S$. An isomorphism $\xi_{1}\to\xi_{2}$ over $S$ is given by a triple $(\rho,\theta,\gamma)$ fitting in the following commutative diagram of log schemes:
\begin{equation}\label{diag:LogMapIso}
\xymatrix{
  (C_{1},\sM_{C,1}) \ar[rr] \ar[rd] \ar[dd]_{\rho} && (X_{S},\sM_{X,1}) \ar[ld] \ar[dd]_{\gamma}\\
   & (S,\sM_{1}) \ar[dd]^(.3){\theta} \\
  (C,\sM_{C,2}) \ar'[r][rr]^{} \ar[rd] && (X_{S},\sM_{X,2}) \ar[ld] \\
  &(S,\sM_{2})& 
  }
\end{equation}
where
 \begin{enumerate}
  \item The pair $(\rho,\theta)$ is an arrow of log curves $(C_{1}\to S,\sM_{1})\to (C_{2}\to S,\sM_{2})$ as in Definition \ref{IsoLogCurve}.
  \item The arrow $\theta$ is an isomorphism of log schemes over $B^{log}$ fitting in the following commutative diagram:
  \[
  \xymatrix{
  (S,\sM_{1}) \ar[rr]^{\theta} \ar[rd] && (S,\sM_{2}) \ar[ld]\\
  &B^{log}&
  }
  \]
  \item The arrow $\gamma$ is obtained from the following cartesian diagram of log schemes:
\[
   \xymatrix{
   (X_{S},\sM_{X_{S},1}) \ar[r]^{\gamma} \ar[d] & (X_{S},\sM_{X_{S},2}) \ar[d] \\
   (S,\sM_{1}) \ar[r]^{\theta} & (S,\sM_{2}).
   }
\]
Note that the underlying map $\underline{\theta}$ and $\underline{\gamma}$ are both identities of the corresponding underlying schemes.
 \end{enumerate}
 
Denote by $\sI som_{S}(\xi_{1},\xi_{2})$ the functor over $S$, which for any $S$-scheme $T \rightarrow S$ associates the set of isomorphisms of $\xi_{T,1}$ and $\xi_{T,2}$ over $T$, where $\xi_{T,1}$ and $\xi_{T,2}$ are the pull-back of $\xi_{1}$ and $\xi_{2}$ via $T \rightarrow S$ respectively. Denote by $\sA ut_{S}(\xi)$ the functor of automorphisms of $\xi$ over $S$.
\end{defn}

\begin{defn}\label{defn:GeneralMapStack}
Denote by $\sL\sM_{g,n}(X^{log}/B^{log})$ the fibered category over $\Sch_{B}$, which for any $S\to B$ associates the groupoid of log maps $\xi$ over $S$, with the underlying pre-stable curves of genus $g$, and $n$-markings. For simplicity of notation, in this section we will use $\sL\sM$ to denote $\sL\sM_{g,n}(X^{log}/B^{log})$.
\end{defn}

\begin{rem}\label{rem:WarningOfStudy}
By Definition \ref{defn:LogMap}, we only allow pull-backs of log maps via strict log maps in $\sL\sM$, hence do not change the log structures. Thus, given a scheme $S$, the groupoid $\sL\sM(S)$ contains all possible log structures $\sM_{S}$ on $S$ with log maps over $(S,\sM_{S})$. This is a huge stack, as it parametrizes in particular all possible log structures on the base. One would like to consider a smaller stack parameterizing only log maps without the information of the base log structures. It will be shown in Section \ref{sec:UniverMin} that if we work over fs log schemes rather than the usual category of schemes for the base, then the stack we want is $\sK_{g,n}^{pre}(X^{log})$ as introduced in Section \ref{sec:LogMap}. 
\end{rem}

Denote by $\fM_{g,n}$ the algebraic stack of genus $g$, $n$-marked pre-stable curves with the canonical log structure as in Section \ref{ss:CanCurveLog}. Consider the new algebraic stack 
\[\fB=\sL og_{\fM_{g,n}\times B^{log}},\]
where the fibered product is in the log sense, and $\sL og_{\bullet}$ is the log stack introduced in Section \ref{s:LogStack}. By Theorem \ref{thm:LogStack}, the stack $\fB$ is algebraic over $B$. 

\begin{rem}\label{rem:ModuliBase}
We give the moduli interpretation of $\fB$ as follows. For any $B$-scheme $S$, an object $\zeta\in\fB(S)$ is a diagram 
\begin{equation}\label{diag:TarSou}
 \xymatrix{
 (C,\sM_{C}) \ar[rd] && (X_{S},\sM_{X_{S}}) \ar[ld]\\
  &(S,\sM) & 
 }
\end{equation}
where the left arrow is a family of genus $g$, $n$-marked log curves given by the induced map $(S,\sM_{S})\to \fM_{g,n}$, and the right arrow is given by the induced map $(S,\sM_{S})\to B^{log}$. Consider two objects $\zeta_{1}$ and $\zeta_{2}$ in $\fB(S)$. An arrow $\zeta_{1}\to \zeta_{2}$ over the scheme $S$ is a triple $(\rho,\theta,\gamma)$ given by the following diagram
\begin{equation}\label{diag:IsoTarSou}
 \xymatrix{
 (C_{1},\sM_{C,1}) \ar[rd] \ar[dd]_{\rho} && (X_{S},\sM_{X_{S},1}) \ar[ld] \ar[dd]_{\gamma}\\
   & (S,\sM_{1}) \ar[dd]_{\theta} \\
 (C,\sM_{C,2})  \ar[rd] && (X_{S},\sM_{X_{S},2}) \ar[ld] \\
  &(S,\sM_{2})& 
 }
\end{equation}
where the square on the left is an isomorphism of log curves, and the square on the right satisfies the condition in Definition \ref{defn:LogMapIso}(2) and (3).
\end{rem}
  
\begin{obs}\label{obs:RelateToBase}
Note that there is a natural morphism of fibered categories $\sL\sM\to \fB$ by removing the log map $f$ as in Definition \ref{defn:LogMap}. Note that any non-trivial isomorphism of a log map is a non-trivial isomorphism of the corresponding log source and target. This implies that $\sL\sM$ is a pre-sheaf over $\fB$.
\end{obs}

We denote by $\sK^{pre}_{g,n}(X/B)$ the stack of usual pre-stable maps with the source given by genus $g$, $n$-marked pre-stable curves, and family of targets given by $X\to B$. This is an algebraic stack over $B$. For simplicity of notations, we will denote this stack by $\sK$. 

\begin{obs}\label{obs:CompMapStack}
Note that we have a natural arrow $\sL\sM\to \sK$ by removing all log structures. Given a log map $\xi$, denote by $\underline{\xi}$ the corresponding object in $\sK$.
\end{obs}

Our main result of this section is the following:
\begin{thm}\label{thm:AlgStackLogMap}
The fibered category $\sL\sM$ is an algebraic stack.
\end{thm}
\begin{proof}
The rest of this section is devote to the proof of this theorem. The representability of the diagonal $\sL\sM\to \sL\sM\times\sL\sM$ is proved in Section \ref{ss:DiagRep}. By Observation \ref{obs:RelateToBase}, we have a natural map $\sL\sM\to \fB$ to the algebraic stack $\fB$. Thus, to produce a smooth cover for $\sL\sM$, it is enough to check Artin's criteria \cite[5.1]{Artin} relative to $\fB$. This will be done from Section \ref{ss:StackGlue} to \ref{ss:CompWithCompletion}. 
\end{proof}

\subsection{Representability of the isomorphism functors of log maps}\label{ss:DiagRep}
\begin{prop}\label{prop:RepIsoLogMap}
Consider two log maps $\xi_{1}$ and $\xi_{2}$ over a $B$-scheme $S$ as in Definition \ref{defn:LogMapIso}. The functor $\sI som_{S}(\xi_{1},\xi_{2})$ is represented by an algebraic space of finite type over $S$.
\end{prop}
\begin{proof}
Using the notations as in Definition \ref{defn:LogMapIso}, Remark \ref{rem:ModuliBase}, and Observation \ref{obs:CompMapStack}, we form the following commutative diagram:
\begin{equation}\label{diag:IsoRelToBK}
\xymatrix{
\sI som_{S}(\xi_{1},\xi_{2}) \ar[ddr]_{\phi_{1}} \ar[drr]^{\phi_{2}} \ar[dr]|-{\phi_{3}} \\
& I \ar[d] \ar[r] & \sI som_{S}(\underline{\xi_{1}},\underline{\xi_{2}}) \ar[d]_{\psi_{2}}\\
& \sI som_{S}(\zeta_{1},\zeta_{2}) \ar[r]^{\psi_{1}} & \sI som_{S}(\underline{\zeta_{1}},\underline{\zeta_{2}}),
}
\end{equation}
where the square is cartesian, and $\phi_{3}$ is given by the universal property of fiber product. Here $\zeta_{i}$ is the corresponding log source and target of $\xi_{i}$ given by the natural map in Observation \ref{obs:RelateToBase}, and $\underline{\xi_{i}}$ is the underline map of $\xi_{i}$ given by the natural map in Observation \ref{obs:CompMapStack}. The object $\underline{\zeta_{i}}$ can be obtained by removing log structures on $\zeta_{i}$, or given by the source and target of $\underline{\xi_{i}}$. 

Note that any isomorphism of $\xi_{1}$ and $\xi_{2}$ induces trivial isomorphism of the underlying structure of the target $X_{S}\to S$. Thus, the sheaf $\sI som_{S}(\underline{\zeta_{1}},\underline{\zeta_{2}})$ is the sheaf of isomorphisms of the underlying curves. Since $\sI som_{S}(\underline{\xi_{1}},\underline{\xi_{2}})$, $\sI som_{S}(\underline{\zeta_{1}},\underline{\zeta_{2}})$, and $\sI som_{S}(\zeta_{1},\zeta_{2})$ are represented by algebraic spaces of finite type over $S$, it is enough to show that $\phi_{3}$ is representable and of finite type.

Consider an $S$-scheme $U$, and an arrow $U \rightarrow I$ given by a pair $(\tau,\lambda)$, where 
\[\tau \in \sI som_{S}(\zeta_{1},\zeta_{2})(U) \mbox{\ \ and\ \ } \lambda \in \sI som_{S}(\underline{\xi_{1}},\underline{\xi_{2}})(U), \]
such that their induced elements in $\sI som_{S}(\underline{\zeta_{1}},\underline{\zeta_{2}})(U)$ coincide. Now we have a cartesian diagram :
\[
\xymatrix{
I' \ar[r] \ar[d] & \sI som_{S}(\xi_{1},\xi_{2}) \ar[d] \\
U \ar[r]_{(\tau,\lambda)} & I.
}
\]
Here $I'$ is the sheaf over $U$ which for any $V \rightarrow U$ associated a unital set $\{*\}$ if $(\tau,\lambda)_{V}$ induces an isomorphism between $\xi_{1,V}$ and $\xi_{2,V}$, and the empty set otherwise. Next we will show that $I' \rightarrow U$ is a locally closed immersion of finite type.

For simplicity, we assume $U = S$, denote by $\tau=(\rho,\theta,\gamma)$ as in Remark \ref{rem:ModuliBase}. We need to show that the commutativity of the following diagram of log schemes is represented by a locally closed immersion of finite type:
\[
\xymatrix{
(C_{1},\sM_{C_{1}}) \ar[rr]^{f_{1}} \ar[d]_{\rho} && (X,\sM_{X,1}) \ar[d]_{\gamma}\\
(C_{2},\sM_{C_{2}}) \ar[rr]^{f_{2}} && (X,\sM_{X,2}) .
}
\]
Since the map $\tau$ already gives an isomorphism of the underlying structure, we only need to consider the commutativity of 
\begin{equation}\label{diag:LogCommute}
\xymatrix{
\sM_{C_{1}}  && f_{1}^{*}\sM_{X,1} \ar[ll]_{f_{1}^{\flat}}\\
\rho^{*}\sM_{C_{2}} \ar[u]^{\rho^{\flat}}  && \rho^{*}\circ f_{2}^{*}\sM_{X,2} \ar[ll]^{\rho^{*}\circ f_{2}^{\flat}} \ar[u]_{\gamma^{\flat}}.
}
\end{equation}
Our statement follows from the following lemma. 
\end{proof}

\begin{lem}\label{lem:IsoFinIm}
The condition that (\ref{diag:LogCommute}) commutes, is represented by a quasi-compact locally closed immersion $Z\to S$.
\end{lem}
\begin{proof}
The commutativity of (\ref{diag:LogCommute}) is equivalent to the equality 
\begin{equation}\label{equ:LogCommute}
\rho^{\flat}\circ(\rho^{*}\circ f_{2}^{\flat})=f_{1}^{\flat}\circ \gamma^{\flat}.
\end{equation}
By \cite[3.6]{LogStack}, the condition that (\ref{equ:LogCommute}) holds on the level of characteristic is an open condition on $C_{1}$. Since $C_{1}\to S$ is flat and proper, by shrinking $S$, we can assume that the equality (\ref{equ:LogCommute}) on the level of characteristic holds. 

With this assumption on the characteristic, the proof in \cite[3.6]{LogStack} shows that the (\ref{equ:LogCommute}) is represented by a closed subscheme $T\subset C_{1}$ on the fiber. Note that the statement is locally on $S$. Further shrinking $S$, we can assume that the family $C\to S$ is projective. Now what we want is the maximal closed subscheme $Z\subset S$ parameterizing fibers $C_{1}\times_{S}Z\subset T$ as in \cite[Definition 3, 4]{Abramovich}. Then the lemma can be deduced from the ``essential free'' case of \cite[VIII Theorem 6.5]{SGA3}. See \cite[Theorem 6(3)]{Abramovich} for the reduction argument. 
\end{proof}

Next, we check the Artin's criteria \cite[5.1]{Artin}.
\subsection{$\sL\sM$ is a stack in the \'etale topology}\label{ss:StackGlue}
By \cite[1.1]{Artin}, or \cite[Definition 3.1]{LMB}, we need to prove the following:
\begin{enumerate}
 \item the isomorphism functor is a sheaf in the \'etale topology;
 \item any \'etale descent datum for an object of $\sL\sM$ is effective.
\end{enumerate}
Since the isomorphism functor is shown to be representable, it is a sheaf in the \'etale topology. For the second condition, let $\{S_{i}\rightarrow S\}_{i}$ be an \'etale covering of $S$, and $\xi_{i} \in \sL\sM(S_{i})$ for each $i$. Assume that we have isomorphisms $\phi_{ij}:\xi_{i}|_{S_{i}\times_{S}S_{j}}\rightarrow \xi_{j}|_{S_{i}\times_{S}S_{j}}$ for each pair $(i,j)$, which satisfy the cocycle condition. 

For any $i$, let $\zeta_{i}$ be the corresponding log curve and target as in Remark \ref{rem:ModuliBase} for $\xi_{i}$. Since such $\zeta_{i}$ is parametrized by the algebraic stack $\fB$, we can glue them together to obtain $\zeta$ over $S$, whose restriction to each $S_{i}$ is $\zeta_{i}$. Then \'etale locally we have a log map from $\zeta$ given by $\xi_{i}$. Since log maps are defined in terms of homomorphisms of \'etale sheaves, they can be glued from \'etale local data. Therefore we can glue $\xi_{i}$ to obtain the log map $\xi$ with the source curve given by $\zeta$.

\subsection{$\sL\sM$ is limit preserving}\label{ss:LocalFiniteType}
Consider 
\[R=\lim_{\rightarrow}R_{i},\] 
where $\{R_{i}\}$ is a filtering inductive system of neotherian rings. Denote by $S=\Spec R$ and $S_{i}=\Spec R_{i}$. By \cite[Section 1]{Artin}, we need to show that the following map of groupoids is an equivalence of categories:
\[\lim_{\leftarrow}\sL\sM(S_{i})\rightarrow \sL\sM(S)\] 

Consider a log map $\xi=(C\rightarrow S, X_{S}\to S,\sM_{S},f)$ in $\sL\sM(S)$. Since the stack $\fB$ is locally of finite type, we have the family $\zeta=(C\rightarrow S,X_{S}\to S, \sM_{S})$ coming from $\zeta_{i}=(C_{i}\rightarrow S_{i},X_{S_{i}}\to S_{i},\sM_{S_{i}})$ over $S_{i}$ for some $i$. Also notice that we have an induced map $S\to \sK$ given by the underlying map. Since $\sK$ is locally of finite type, the underlying map $\underline{f}$ comes from $\underline{f}_{i'}$ over some $S_{i'}$. We choose an index $i_{0}$ such that $i_{0}>i$ and $i_{0}>i'$.

It remains to consider the map of log structures $f^{\flat}:f^{*}\sM_{X}\to \sM_{C}$. We first introduce two stacks $\sL^{\Delta}$ and $\sL^{\Lambda}$ as in \cite[section 2]{LogCot}. 

\begin{rem}
Consider a scheme $U$ over $\Z$. Objects in $\sL^{\Delta}(U)$ are commutative diagrams of log structures on $U$ of the following form
\begin{equation}\label{diag:DeltaLogStr}
\xymatrix{
& \sM_{1} \ar[ld] \ar[rd] &\\
\sM_{2} \ar[rr]& & \sM_{3}.
}
\end{equation}
Objects in $\sL^{\Lambda}$ are diagrams of log structures on $U$ of the following form
\begin{equation}\label{diag:LambdaLogStr}
\xymatrix{
& \sM_{1} \ar[ld] \ar[rd] &\\
\sM_{2} & & \sM_{3}.
}
\end{equation}
It was shown in \cite[2.4]{LogCot} that those two stacks $\sL^{\Delta}$ and $\sL^{\Lambda}$ are algebraic stacks locally of finite type. Note that there is a natural morphism $\sL^{\Delta}\to\sL^{\Lambda}$ by dropping the bottom arrow in (\ref{diag:DeltaLogStr}) to obtain (\ref{diag:LambdaLogStr}).
\end{rem}

\begin{obs}\label{obs:LogStrMap}
Consider $\zeta=(\pi_{C}:C\to S, X_{S}\to S,\sM_{S})$ the family of log sources and targets constructed above. There is a natural diagram of log structures on $C$ as follows
\begin{equation}\label{diag:ComLogOnCurve}
\xymatrix{
& \pi^{*}_{C}\sM_{S} \ar[ld] \ar[rd] &\\
f^{*}\sM_{X} & & \sM_{C}.
}
\end{equation}
This induces a natural map $C\to \sL^{\Lambda}$. Consider the fiber product $\sL^{\Delta}\times_{\sL^{\Lambda}}C$. This gives an algebraic stack parameterizing the bottom arrows $f^{\flat}$ that fits in the above commutative diagram.
\end{obs}

The map $f^{\flat}$ is equivalent to a map $C\to \sL^{\Delta}\times_{\sL^{\Lambda}}C$. Note that the algebraic stack $\sL^{\Delta}\times_{\sL^{\Lambda}}C$ is locally of finite presentation. By \cite[Proposition 4.18(i)]{LMB}, we have the map $f^{\flat}$ coming from some $f^{\flat}_{i_{1}}$ over $S_{i_{1}}$ for some $i_{1}>i_{0}$. This map is compatible with all the log structures coming from base and target. Indeed, consider the composition 
\[p_{j}:C_{j}\to \sL^{\Delta}\times_{\sL^{\Lambda}}C_{j} \to C_{j}.\]
Applying \cite[Proposition 4.18(i)]{LMB} again, we see that the identity $p=id_{C}:C\to C$ comes from $p_{j}$ for some $i_{2}>i_{1}$. Thus, the map $f_{i_{2}}$ also compatible with the underlying map $\underline{f}$. This proves the essential surjectivity.

The full faithfulness follows from \cite[Proposition 4.15(i)]{LMB} and the fact that the diagonal $\sL\sM\rightarrow \sL\sM\times_{B}\sL\sM$ is representable and locally of finite type.

\subsection{Deformations and obstructions}\label{DefObs}
By \cite[Definition 5.1]{Artin}, it remains to find a smooth cover of $\sL\sM$. As in Observation \ref{obs:RelateToBase}, we have a representable map of stack $\sL\sM\to \fB$. Since $\fB$ is an algebraic stack, it would be enough to produce a smooth cover for $\sL\sM_{U}:=\sL\sM\times_{\fB} U$, where $U\rightarrow \fB$ is an arbitrary smooth map. This can be done by checking Artin's criteria \cite[5.2]{Artin} for $\sL\sM_{U}$ relative to $U$. First we consider the deformations and obstructions.

Let $A_{0}$ be a reduced neotherian ring over $U$, and $A'\rightarrow A \rightarrow A_{0}$ be an infinitesimal extension of $A_{0}$, where $A'\rightarrow A$ is surjective whose kernel $I$ is a finite $A_{0}-$module, hence is a square-zero ideal. Denote by $S=\Spec A$ and $S'=\Spec A'$. Consider a log map $\xi_{A}=(C\rightarrow S,X_{S}\to S, \sM_{S},f)\in \sL\sM_{U}$. Let $\xi_{0}=(C_{0}\rightarrow S_{0},X_{S_{0}}\to S_{0},\sM_{S_{0}},f_{0})$ be the restriction of $\xi_{A}$ over $A_{0}$. Since we are over $U$, the log sources and targets $(C\rightarrow S, X_{S}\to S, \sM_{S})$ comes from the structure morphism $S\to U$. Note that we have another family of log sources and targets $(C'\rightarrow S', X_{S'}\to S', \sM_{S'})$, which are also from the structure map $S'\rightarrow U$. To obtain a deformation of $\xi_{A}$ over $S'$, it is equivalent to producing a dotted arrow $f'$ that fits in the following commutative diagram:
\begin{equation}\label{diag:DeformMap}
\xymatrix{
(C,\sM_{C}) \ar[rr]^{k} \ar[dr]^{f} \ar@/_1pc/[dddr] && (C',\sM_{C'}) \ar@/_1pc/[dddr] \ar@{.>}[dr]^{f'} &\\
&(X_{S},\sM_{X_{S}}) \ar[dd] \ar[rr]^{j} && (X_{S'},\sM_{X_{S'}}) \ar[dd]\\
&&&\\
&(S,\sM_{S}) \ar[rr]^{i} && (S',\sM_{S'})
}
\end{equation}
Note that the front and back squares in (\ref{diag:DeformMap}) are cartesian squares. Let $\BL_{X_{S}/S}^{log}$ be the logarithmic cotangent complex of the log map $(X_{S},\sM_{X_{S}})\to(S,\sM_{S})$ as in \cite{LogCot}. By \cite[5.9]{LogCot}, we have the following results:
\begin{enumerate}
 \item there is a canonical class $o\in Ext^{1}(f^{*}\BL_{X_{S}/S}^{log},I\otimes_{A_{0}}\sO_{C_{0}})$, whose vanishing is necessary and sufficient for the existence of a morphism $f'$ fitting in the above diagram.
 \item if $o=0$, then the set of such maps $f'$ is a torsor under $Ext^{0}(f^{*}\BL_{X_{S}/S}^{log},I\otimes_{A_{0}}\sO_{C_{0}})$.
\end{enumerate}
Since the family of targets $X^{log}\to B^{log}$ is log flat, by \cite[1.1(iv)]{LogCot} we have 
\[Ext^{1}(f^{*}\BL_{X_{S}/S}^{log},I\otimes_{A_{0}}\sO_{C_{0}})\cong Ext^{1}(f^{*}_{0}\BL_{X_{S_{0}}/S_{0}}^{log},I\otimes_{A_{0}}\sO_{C_{0}}).\] 
Thus we define $\sD_{\xi_{A}}(I)=Ext^{0}(f^{*}\BL_{X_{S}/S}^{log},I\otimes_{A_{0}}\sO_{C_{0}})$ and $\sO_{\xi_{0}}(I)=Ext^{1}(f^{*}_{0}\BL_{X_{S_{0}}/S_{0}}^{log},I\otimes_{A_{0}}\sO_{C_{0}})$ to be the modules of deformations and obstructions respectively. Note that the log cotangent complex $\BL_{X_{S}/S}^{log}$ is bounded above with coherent cohomologies. The conditions of deformation and obstruction modules in \cite[5.2(4)]{Artin} follows from the standard property of cohomology, see for example \cite[5.3.4]{AV}.

\subsection{Schlessinger's conditions}\label{ss:SchCond}
By \cite[5.2(2)]{Artin}, we need to verify Schlessinger's conditions (S1) and (S2) as in \cite[section 2]{Artin}. The condition (S2) follows from the cohomological description of the module of deformation $\sD$. Next we check the condition (S1$'$) \cite[2.3]{Artin}, which is a stronger version of (S1).

Indeed, consider an infinitesimal extension $A'\rightarrow A\rightarrow A_{0}$ as in Section \ref{DefObs}, and a $U$-algebra homomorphism $B\rightarrow A$ such that the composition $B\rightarrow A_{0}$ is surjective. Consider $\xi_{A}\in \sL\sM_{U}(A)$. For any surjection $R\rightarrow A$, denote by $\sL\sM_{\xi_{A}}(R)$ the category of log maps over $\Spec R$ whose restriction to $\Spec A$ is $\xi_{A}$. Then we need to show that \[\sL\sM_{\xi_{A}}(A'\times_{A}B)\rightarrow \sL\sM_{\xi_{A}}(A')\times\sL\sM_{\xi_{A}}(B)\]
is an equivalence of categories.

First, consider the essential surjectivity. Consider objects $\xi_{A'}\in\sL\sM_{\xi_{A}}(A')$ and $\xi_{B}\in\sL\sM_{\xi_{A}}(B)$. Denote by $\xi_{A'}=(\zeta_{A'},f_{A'})$ and $\xi_{B}=(\zeta_{B},f_{B})$, where $\zeta_{A'}$ and $\zeta_{B}$ are the corresponding log sources and targets as in Remark \ref{rem:ModuliBase}. Note that the two families $\zeta_{A'}$ and $\zeta_{B}$ correspond to maps $\Spec A'\rightarrow U$ and $\Spec B\rightarrow U$, which induce the same map $\Spec A\rightarrow U$ by restricting to $\Spec A$. Then we can glue them together to obtain a map $\Spec B\times_{A}A' \rightarrow U$. This induces a family $\zeta_{B\times_{A}A'}$ over $\Spec B\times_{A}A'$, whose restrictions to $\Spec A'$ and $\Spec B$ are $\zeta_{A'}$ and $\zeta_{B}$ respectively. Since the stack $\sK$ parameterizing the underlying maps is algebraic, the same argument as above produces a gluing $\underline{f}_{A'\times_{A}B}$ of $\underline{f}_{A'}$ and $\underline{f}_{B}$.

It remains to produce a compatible morphism of log structures $f^{\flat}_{A'\times_{A}B}$. Next we choose an affine open cover $V_{B\times_{A}A'}=\bigcup_{i}V_{i}$ of the log source curve in $\zeta_{B\times_{A}A'}$, its restrictions to $A'$ and $B$ give the affine open covers $V_{B}$ and $V_{A}$ for curves of $\zeta_{A'}$ and $\zeta_{B}$ respectively. Consider the stack 
\[\sL^{\Delta}\times_{\sL^{\Lambda}}C_{A'} \]
induced by $\zeta_{A'}$ and the map $\underline{f}_{A}$ as in Observation \ref{obs:LogStrMap}. Similarly, we have stack
\[\sL^{\Delta}\times_{\sL^{\Lambda}}C_{B}\]
induced by $\zeta_{B}$ and $\underline{f}_{B}$. They can be glued to give $\sL^{\Delta}\times_{\sL^{\Lambda}}C_{A'\times_{A}B}$ which corresponds to $\zeta_{A'\times_{A}B}$. Consider the maps $V_{A'}\rightarrow \sL^{\Delta}\times_{\sL^{\Lambda}}C_{A'}$ and $V_{B}\rightarrow \sL^{\Delta}\times_{\sL^{\Lambda}}C_{B}$ induced by $f_{A'}$ and $f_{B}$ respectively. Note that these maps can be glued together and descend to a map 
\[C_{A'\times_{A}B} \to \sL^{\Delta}\times_{\sL^{\Lambda}}C_{A'\times_{A}B}.\]
This induce a map of log structures 
\[f^{\flat}_{A'\times_{A}B}: \underline{f}_{A'\times_{A}B}^{*}\sM_{X_{A'\times_{A}B}} \to \sM_{C_{A'\times_{A}B}}.\] 
By construction $f^{\flat}_{A'\times_{A}B}$ is compatible with $\zeta_{A'}\times_{A}B$ and the underlying map $\underline{f}_{A'\times_{A}B}$. 

The full faithfulness follows from the representability of isomorphism functor of log maps.

\subsection{Compatibility with formal completions}\label{ss:CompWithCompletion}
Let $\hat{A}$ be a complete local ring, and $m$ be the maximal ideal of $\hat{A}$. Denote by $A_{n}=\hat{A}/m^{n}$, $S=\Spec \hat{A}$, and $S_{n}=\Spec A_{n}$. Since we work over a fixed chart $U\to \sL\sM$, it is enough to consider a family of log maps $\{\xi_{n}=(C_{n}\rightarrow S_{n}, X_{S_{n}}\to S_{n},\sM_{S},f_{n})\}_{n}$ such that $\xi_{n}\in\sL\sM_{U}(S_{n})$, and $\xi_{n}|_{S_{k}}=\xi_{k}$ for any $n\geq k$. According to \cite[5.2(3)]{Artin}, we need to show that there exists an element $\xi\in\sL\sM_{U}(S)$, such that $\xi|_{S_{n}}=\xi_{n}$ for any $n$.

Denote by $\zeta_{n}=(C_{n}\to S_{n},X_{S_{n}}\to S_{n}, \sM_{S_{n}})$ the family of log sources and targets of $\xi_{n}$. For each $n$, there is a map $S_{n}\to U$ induced by $\zeta_{n}$, such that they fit in the following commutative diagrams for any $k\leq n$:
\[
\xymatrix{
S_{n} \ar[r] \ar[d] & U\\
S_{k} \ar[ru].
}
\]
Thus the above diagram induces a map $S\rightarrow U$, whose restriction to $S_{n}$ is the map given by $\zeta_{n}$ as above. By pulling back the universal family over $U$, we obtain a family of log sources and targets $\zeta=(C\rightarrow S, X_{S}\to S,\sM_{S})$. Note that $\zeta|_{S_{n}}=\zeta_{n}$ for any $n$.

Denote by $\underline{\xi}_{n}$ the usual pre-stable map over $S_{n}$. Consider the family of compatible underlying maps $\{\underline{\xi}_{n}\}$. By \cite[5.4.1]{EGAIII(1)}, there exists a unique underlying map (up to a unique isomorphism) $\underline{f}:C\rightarrow X_{S}$ such that $\underline{f}|_{S_{n}}=\underline{f}_{n}$.

Now to construct $\xi$, we need to construct a log map $f:(C,\sM_{C})\rightarrow (X_{S},\sM_{X_{S}})$, which is compatible with the underlying map $\underline{f}$ and $f_{n}$ for all $n$. By definition of log maps, this is equivalent to constructing a map of log structures $f^{\flat}: \underline{f}^{*}\sM_{X_{S}} \rightarrow \sM_{C}$, which is compatible with $f^{\flat}_{n}$ and the log structure $\sM_{S}$ on the base. For simplicity, denote by $\sM=\underline{f}^{*}\sM_{X_{S}}$.

Choose an affine \'etale cover of $C$, such that over each affine chart $V\to C$ the log structures $\sM|_{V}$, $\sM_{C}|_{V}$, and $\pi_{1}^{*}\sM_{S}$ can be obtained by taking the log structure associated to $\Gamma(\sM,V)\to \sO_{V}$, $\Gamma(\sM_{C},V)\to\sO_{V}$, and $\pi_{1}^{*}\sM_{S}\to \sO_{V}$ respectively. Since the charts of fine log structures always exist \'etale locally, the above choice of cover exists. We first construct $f^{\flat}$ on such chart $V$.

Denote by $V_{n}=V\times_{S}S_{n}$. Then the canonical map $V_{n}\to C_{n}$ gives an affine \'etale chart. Consider the compatible families of monoids $\{\Gamma(\sM_{n},V_{n})\}_{n}$ and $\{\Gamma(\sM_{C_{n}},V_{n})\}_{n}$. For simplicity, let $\sN$ to be one of the monoids $\Gamma(\sM,V)$, $\Gamma(\sM_{C},V)$, or $\Gamma(\pi_{1}^{*}\sM_{S},V)$, and $\sN_{n}$ to be the one of the corresponding reductions $\Gamma(\sM_{n},V_{n})$, $\Gamma(\sM_{C_{n}},V_{n})$, or $\Gamma(\pi_{1}^{*}\sM_{S_{n}},V_{n})$. Denote by $q_{n}: \sN\to \sN_{n}$ the restriction map, and by $p_{n}: \varprojlim\sN_{n}\to \sN_{n}$ the canonical map. Let $p_{nk}: \sN_{n}\to \sN_{k}$ be the restriction map for all $k\leq n$. Assume that $V_{n}=\Spec R_{n}$, and $R=\varprojlim R_{n}$. Thus, we write $V=\Spec R$. 

Note that inverse limit exists in the category of monoids, and their formation commutes with the forgetful functor to the category of sets (\cite[Chapter I, 1.1]{Ogus}). Furthermore, the inverse limit of a family of integral monoids is again integral (\cite[Chapter I, 1.2]{Ogus}). Consider an element $e\in \sN$. This induces a family of compatible elements $\{q_{n}(e)\}_{n}\in \varprojlim \sN_{n}$. In this way, we obtain a canonical map of integral monoids:
\[p: \sN \to \varprojlim \sN_{n}.\]

\begin{lem}\label{lem:compare-section}
\begin{enumerate}
  \item Consider an element $e\in \sN_{n}$. Then $p_{nk}(e)\in R_{k}^{*}$ for some $k\leq n$ if and only if $e\in R^{*}_{n}$. Furthermore, the map $p_{nk}$ induces a natural isomorphism $\bar{p}_{nk}:\sN_{n}/R_{n}^{*}\to \sN_{k}/R_{k}^{*}$.
  \item Consider an element $e\in \sN$. Then $q_{n}(e)\in R_{n}^{*}$ for some $n$ if and only if $e\in R^{*}$. Furthermore, the map $q_{n}$ induces a natural isomorphism of monoids $\bar{q}_{n}: \sN/R^{*}\to \sN_{n}/R^{*}_{n}$.
  \item There is a natural inclusion $R^{*}\hookrightarrow \varprojlim\sN_{n}$, which fits in the following commutative diagram
  \begin{equation}\label{diag:compare-unit}
  \xymatrix{
  &R^{*} \ar[ld] \ar[rd]&\\
  \sN \ar[rr]^{p} && \varprojlim\sN_{n},
  }
  \end{equation}
  where the left side arrow is the natural inclusion of units given by the corresponding log structures.
  \item The canonical projection $p_{n}:\varprojlim\sN_{n}\to \sN_{n}$ induces an isomorphism of monoids
  \[\bar{p}_{n}: (\varprojlim\sN_{n})/R^{*}\to \sN_{n}/R^{*}_{n}.\]
  \item The canonical map $p_{}:\sN \to \varprojlim\sN_{n}$ induces an isomorphism of monoids
  \[\bar{p}: \sN/R^{*} \to (\varprojlim\sN_{n})/R^{*}.\]
\end{enumerate}
\end{lem}
\begin{proof}
The first part of Statement (1) follows from the following commutative diagram
\[
\xymatrix{
\sN_{n} \ar[rr]^{p_{nk}} \ar[d] && \sN_{k} \ar[d] \\
R^{}_{n} \ar[rr] && R_{k},
}
\]
where the two vertical maps are given by the structure morphism of the corresponding log structures. This immediately implies the existence of $\bar{p}_{nk}$. The surjectiviry of $p_{nk}$ for any $k\leq n$ implies that $\bar{p}_{nk}$ is also surjective. To see the injectivity, consider two elements $a,b\in \sN_{n}$ such that $p_{nk}(a)=p_{nk}(b)+\log u$ for some $u\in R_{k}^{*}$. Without loss of generality, we can assume that $a + c' = b + c$ in $\sN_{n}$. Thus $p_{nk}(c') = p_{nk}(c) + \log u$, which implies $p_{nk}(c'-c)\in R^{*}_{k}$, hence $c'-c\in R^{*}_{n}$. This proves the second part of Statement (1).

Statement (2) can be proved similarly as the first one.

To prove (3), consider $e\in \varprojlim\sN_{n}$, which can be represented by a family of compatible elements $\{e_{n}\in \sN_{n}\}_{n}$. Assume that $e_{n'}\in R_{n'}^{*}$ for some $n'$. Then the first statement implies that $e_{n}\in R_{n}^{*}$ for all $n$. Thus we have a unique element $e\in R^{*}\subset R$ such that $e|_{V_{n}}=e_{n}$. This induces a canonical map $R^{*}\hookrightarrow \varprojlim\sN_{n}$. Now the commutativity of (\ref{diag:compare-unit}) can be checked directly.

In fact the above argument proves that $p_{n}(e)\in R^{*}_{n}$ for some $e\in \varprojlim\sN_{n}$, if and only if $e\in R^{*}$. Thus we obtain a canonical map $\bar{p}_{n}: (\varprojlim\sN_{n})/R^{*}\to \sN_{n}/R^{*}_{n}.$ Note that we have the following commutative diagram
\begin{equation}\label{diag:compare-limit}
\xymatrix{
\sN \ar[d]_{p} \ar[rrrd]^{q_{n}} &&& \\
\varprojlim\sN_{n} \ar[rrr]^{p_{n}} &&& \sN_{n}. 
}
\end{equation}
The surjectivity of $q_{n}$ implies that $\bar{p}_{n}$ is also surjective. The injectivity of $\bar{p}_{n}$ can be proved similarly as for the first statement. This proves (4). 

Finally, note that (\ref{diag:compare-limit}) induces a commutative diagram 
\begin{equation}\label{diag:compare-limit-monoid}
\xymatrix{
\sN/R^{*} \ar[d]_{\bar{p}} \ar[rrrd]^{\bar{q}_{n}} &&& \\
(\varprojlim\sN_{n})/R^{*} \ar[rrr]^{\bar{p}_{n}} &&& \sN_{n}/R^{*}_{n}. 
}
\end{equation}
Since both $\bar{q}_{n}$ and $\bar{p}_{n}$ are isomorphisms of monoids, we conclude that $\bar{p}$ is also an isomorphism. This proves (5).
\end{proof}

\begin{prop}\label{prop:compare-section}
The map of monoids $p: \sN \to \varprojlim\sN_{n}$ is an isomorphism.
\end{prop}
\begin{proof}
By Lemma \ref{lem:compare-section}(3) and (5), we have a commutative diagram
\[
\xymatrix{
\sN \ar[rr]^{p} \ar[d] && \varprojlim\sN_{n} \ar[d]\\
\sN/R^{*} \ar[rr]^{\bar{p}} && \varprojlim\sN_{n}/R^{*}.
}
\]
Pick two sections $e,e'\in \sN$ such that $p(e)=p(e')$. Denote by $\bar{e}$ and $\bar{e}'$ the corresponding images in $\sN/R^{*}$. It follows from Lemma \ref{lem:compare-section}(5) that $\bar{e}=\bar{e}'$. Thus, we have $e=e'+\log u$ for some $u\in R^{*}$. But the assumption $p(e)=p(e')$ implies that $p(u)=1\in \varprojlim\sN_{n}$. By Lemma \ref{lem:compare-section}(3), we have $u=1$. This proves that $p$ is also injective. 

To prove the surjectivity, consider an element $a\in \varprojlim\sN_{n}$. Since $\bar{p}$ is an isomorphism, denote by $\bar{a}$ the image of $a$ in $\sN/R^{*}$. Let $a'$ be a lifting of $\bar{a}$ in $\sN$. Then there exists an element $u\in R^{*}$ such that $a=p(a')+\log u$. Thus $p(a'+\log u)=a$.
\end{proof}

Pick an element $\{e_{n}\in \Gamma(\sM_{n},V_{n})\}_{n}\in \varprojlim\Gamma(\sM_{n},V_{n})$. We obtain a compatible family $\{f_{n}(e_{n})\}_{n}\in \varprojlim\Gamma(\sM_{C_{n}},V_{n})$. Thus the compatible morphism of log structures $\{f_{n}\}$ induces a natural map of monoids 
\[\varprojlim\Gamma(\sM_{n},V_{n}) \to \varprojlim\Gamma(\sM_{C_{n}},V_{n}).\]
By Proposition \ref{prop:compare-section}, we have a natural map of monoids:
\[\Gamma(f^{\flat},V):\Gamma(\sM,V)\to \Gamma(\sM_{C},V).\]

Next we show that $\Gamma(f^{\flat},V)$ induces a map of log structures $f^{\flat}_{V}:\sM|_{V}\to \sM_{C}|_{V}$. Since the two log structures $\sM|_{V}$ and $\sM_{C}|_{V}$ can be obtained from $\Gamma(\sM,V)$ and $\Gamma(\sM_{C},V)$ respectively, it is enough to show that the following diagram is commutative:
\[
\xymatrix{
\Gamma(\sM,V) \ar[rr]^{\Gamma(f^{\flat},V)} \ar[dr]_{\exp_{1}} && \Gamma(\sM_{C},V) \ar[dl]^{\exp_{2}}\\
&R&
}
\]
where $\exp_{1}$ and $\exp_{2}$ are the structure morphism of the corresponding log structures. To see this, consider any section $s\in \Gamma(\sM,V)$. Since $\exp_{1}(s)|_{S_{n}}=\exp_{1}\circ \Gamma(f^{\flat},V)(s)|_{S_{n}}$ for any $n$, we have $\exp_{1}(s)=\exp_{1}\circ \Gamma(f^{\flat},V)(s)$. This proves the commutativity.

We claim that $f^{\flat}_{V}$ is compatible with the log structure on the base. This is equivalent to showing the commutativity of the following diagram of log structures on $V$:
\begin{equation}\label{diag:local-base-commute}
\xymatrix{
&\pi_{1}^{*}\sM_{S}|_{V} \ar[ld]_{\pi_{2}^{\flat}} \ar[rd]^{\pi_{1}^{\flat}} & \\
\sM|_{V} \ar[rr]^{f^{\flat}_{V}} &  & \sM_{C}|_{V},
}
\end{equation}
where $\pi_{1}:C\to S$ and $\pi_{2}:X_{S}\to S$ are the projections. Since $\pi_{1}^{*}\sM_{S}|_{V}$ can be obtained by taking the log structure associated to $\Gamma(\pi_{1}^{*}\sM_{S},V)\to \sO_{V}$. Hence to verify the commutativity of (\ref{diag:local-base-commute}) it is enough to show that the following diagram is commutative:
\[
\xymatrix{
&\Gamma(\pi_{1}^{*}\sM_{S},V) \ar[ld]_{\pi_{2}^{\flat}} \ar[rd]^{\pi_{1}^{\flat}} & \\
\Gamma(\sM,V) \ar[rr]^{f^{\flat}_{V}} &  & \Gamma(\sM_{C},V).
}
\]
But This follows from the definition of $f^{\flat}_{V}$, and the following commutative diagram for each $n$:
\[
\xymatrix{
  & \pi^{*}_{1,n}\sM_{S_{n}} \ar[ld] \ar[rd] & \\
f^{*}\sM_{n} \ar[rr]^{f^{\flat}_{n}} &  & \sM_{C_{n}}.
}
\]
Thus we obtain the desired map $f^{\flat}_{V}$ over each affine chart $V$. 

Finally, notice that the construction of $f^{\flat}_{V}$ is functorial. Hence, we are able to obtain a global map $f^{\flat}$ by gluing $f^{\flat}_{V}$ on each affine chart. This finishes the proof of compatibility with formal completions.

\section{Minimal logarithmic maps to rank one Deligne-Faltings log pairs}\label{sec:LogMap}

\subsection{Basic definitions and notations}

\begin{defn}\label{defn:Target}
We call the log scheme $X^{log}=(X,\sM_{X})$ a \textit{rank one Deligne-Faltings pair}, if 
\begin{enumerate}
 \item $X$ is a projective variety over $\C$; 
 \item $\sM_{X}$ is a DF log structure on $X$ as in Definition \ref{defn:DF}, with a global presentation $\N\to \CharM_{X}$.
\end{enumerate}
\end{defn}

\begin{rem}
The results in Section \ref{sec:LogMap} and \ref{sec:UniverMin} still hold if we assume $X$ to be only separated of finite type over $\C$. However, the projectivity is essential for the properness of the stack $\sK_{\Gamma}(X)$ as in Definition \ref{defn:MinStableStack}.
\end{rem}

\begin{con}\label{con:LogTarget}
In the rest of this paper, we fix a Deligne-Faltings pair $(X,\sM_{X})$ as our target of log maps, with a global presentation $\N\to \CharM_{X}$. Denote by $(L,s)$ the pair consisting of a line bundle $L$, and a morphism of sheaves $s:L\to \sO_{X}$ corresponding to $\sM_{X}$. Let $D$ be the vanishing locus of the section $s\in H^{0}(L^{\vee})$. Denote by $\delta$ the standard generator of $\N$. For convenience, locally we identify $\delta$ with its image in $\CharM_{X}$.
\end{con}

\begin{rem}\label{rem:BundleExplain}
Note that if $s=0$, then $D=X$. If $s$ is not a zero section, then $D$ is a divisor in $X$. Thus, we have $L= \sO_{X}(-D)$, with the natural inclusion $s:\sO_{X}(-D)\hookrightarrow \sO_{X}$. The section $\delta$ locally lifts to a section in $\sO_{X}$, whose vanishing locus gives the divisor $D$.
\end{rem}

\begin{rem}\label{rem:LogMapIso}
The target $X^{log}$ should be viewed as a log scheme over a point with trivial log structures. Thus, we can simplify the notations in Section \ref{ss:SetNote} as follows. A log map over a usual scheme $S$ is given by the triple $(C\to S, \sM_{S},f)$, where $(C\rightarrow S,\sM_{S})$ is a log curve in Definition \ref{DefLogC2}, and $f:(C,\sM_{C})\rightarrow (X,\sM_{X})$ is a log map. 

Consider two log maps $\xi=(C\to S,\sM_{S},f)$ and $\xi'=(C'\to S,\sM_{S}',f')$ over a scheme $S$. An arrow $\xi\to\xi'$ over $S$ is a pair $(\rho,\theta)$ as in Definition \ref{IsoLogCurve} such that the following diagram commutes:
\[
\xymatrix{
&&(X,\sM_{X}) \\
(C,\sM_{C}) \ar@/^/[urr]^{f} \ar[r]^{\rho} \ar[d]& (C',\sM_{C'}) \ar@/_/[ur]_{f'} \ar[d]&\\
(S,\sM_{S})  \ar[r]^{\theta} &(S,\sM'_{S})}
\]
where the square is a cartesian square of log schemes. This is compatible with Definition \ref{defn:LogMap} and \ref{defn:LogMapIso}. 
\end{rem}

\subsection{Log maps on the level of characteristic}\label{ss:LogMapChar}
Consider a log map $\xi=(\pi:C\to S, \sM_{S},f)$ as in Remark \ref{rem:LogMapIso}, where $S=\Spec k$ is a geometric point and $(C\to S,\sM_{S})$ is a log pre-stable curve. Pick a point $p\in C$, which sits in an irreducible component $Z$. We have a map of characteristic monoids:
\begin{equation}\label{equ:CharMapSm}
\bar{f}^{\flat}:f^{*}(\CharM_{X})_{p}\to\CharM_{C,p}.
\end{equation}

First consider the case $p$ is a smooth non-marked point. By the description in Definition \ref{DefLogC1}, we have $\bar{f}^{\flat}(\delta)=e\in \CharM_{S}$ at $p$. By \cite[3.5(i),(iii)]{LogStack}, the equality $\bar{f}^{\flat}(\delta)=e$ lifts to an \'etale neighborhood of $p$. 

\begin{defn}
We call $e$ \textit{the degeneracy of $Z$}. Note that if $p\notin D$ for some $p\in Z$, then the image $e$ vanishes in $\CharM_{S}$. A component $Z$ is called \textit{degenerate} if its degeneracy is not zero. This is equivalent to saying that $Z$ maps to $D$ via $f$.
\end{defn}

Next, we consider the case where $p$ is a marked point. Locally at $p$, we have $\sM_{C}\cong \pi^{*}\sM_{S}\oplus_{\sO^{*}_{C}}\sN$, where $\sN$ is the canonical log structure associated to the marked point $p$. Then on the level of characteristic, we have 
\begin{equation}\label{equ:MarkedPt}
\bar{f}^{\flat}(\delta)=e+ c_{p}\cdot \sigma_{p}, 
\end{equation}
where $e\in \CharM_{S}$, the element $\sigma_{p}$ is the generator of $\overline{\sN}_{p}$, and $c_{p}$ is a non-negative integer.

\begin{obs}\label{obs:DegMarkedPt}
When we generalize (\ref{equ:MarkedPt}) to nearby smooth points, any lifting of $\sigma_{p}$ in the structure sheaf becomes invertible. Thus, the element $e$ is the degeneracy of the component $Z$ containing $p$.
\end{obs}

\begin{defn}
We call $c_{p}$ \textit{the contact order of $f$ at $p$}.
\end{defn}

\begin{lem}\label{lem:ContactOpen}
Consider a log map $\xi=(C'\to S', \sM_{S'}, g)$ over a scheme $S'$, and a marking $\Sigma_{i}$ on $C'$. There is an open subset in $S'$, such that the contact order along the fixed marking $\Sigma_{i}$ is constant.
\end{lem}
\begin{proof}
Consider the relative characteristic $\CharM_{C'/S'}$. This is a locally constant sheaf along $\Sigma_{i}$, with stalks given by $\N$. Thus along $\Sigma_{i}$ there is a map of locally constant sheaves $g^{*}\CharM_{X}\to \CharM_{C'/S'}$, which locally at $p\in \Sigma_{i}$ is given by $\N\to\N$ by $1\mapsto c$, for some positive integer $c$. Note that the correspondence $1\mapsto c$ can be generalize to the nearby points of $p$. Therefore it forms an open condition on the base. 
\end{proof}

\begin{rem}
When $D$ is a divisor, the contact order of a marked point $\Sigma$ in a non-degenerate component can be identified with the local intersection multiplicity $(C\cdot D)_{\Sigma}$.
\end{rem}

Finally, let us consider the case where $p$ is a node joining two irreducible components $Z$ and $Z'$. Let $e_{p}$ be the element in $\CharM_{S}$ smoothing the node $p$. Denote by $\log x_{p}$ and $\log y_{p}$ the elements in $\CharM_{C}$ given by the local coordinates of the two components $Z$ and $Z'$ at $p$ respectively as in Section \ref{CanLog}. Then locally at $p$ we have the equation in $\CharM_{C}$:
\begin{equation}\label{equ:NodeCurve}
e_{p} = \log x_{p} + \log y_{p}.
\end{equation}
Thus, without loss of generality we can assume that 
\begin{equation}\label{equ:Node}
\bar{f}^{\flat}(\delta)= e + c_{p}\cdot \log x_{p},
\end{equation}
where $c_{p}$ is a positive integer.

\begin{defn}\label{defn:DistNode}
The integer $c_{p}$ is called \textit{the contact order of $f$ at the node $p$.} If $c_{p}\neq 0$, then $p$ is called a \textit{distinguished node.} A point $p\in C$ is called a \textit{distinguished point}, if it is a marked point or node with non-trivial contact order. Otherwise, it is called \textit{non-distinguished point}.
\end{defn}

\begin{lem}\label{lem:DegJump}
Using notations as above, the degeneracy of $Z$ is $e$, and the degeneracy of $Z'$ is $e_{}+c_{p}\cdot e_{p}$.
\end{lem}
\begin{proof}
When we generalize (\ref{equ:Node}) to a smooth point in $Z'$, the section $y$ becomes invertible. Then the statement for $Z'$ follows from the definition of degeneracy of a component. For $Z$, the proof is similar. 
\end{proof}

Lemma \ref{lem:DegJump} gives a way to put a partial order on the set of irreducible components as follows:

\begin{defn}\label{defn:DegNodeLow}
Using the notations as above, we call $Z$ the \textit{lower component} of $p$, and $Z'$ the \textit{upper component} of $p$. 
\end{defn}

\begin{lem}\label{lem:NodeContactOpen}
Consider a log map $\xi=(C'\to S', \sM_{S'}, g)$, and a connected singularity $p\subset C'$. There is an open subset in $S'$, such that over each fiber we have that either the node $p$ is smoothed out, or its contact order remains the same.
\end{lem}
\begin{proof}
The proof is similar to the one for Lemma \ref{lem:ContactOpen}. 
\end{proof}

\subsection{Marked graph}\label{ss:AdGraph}
We next introduce the marked graph, which will be used later to describe the combinatorial data associated to log maps.

\begin{defn}\label{defn:WeightGraph}
A \textit{weighted graph} $G$ is a connected graph with the following data:
\begin{enumerate}
 \item A subset $V_{n}(G)\subset V(G)$ of the set of vertices of $G$, which is called the {\em set of nondegenerate vertices}.
 \item For each edge $l\in E(G)$, we associate a non-negative integer weight $c_{l}$ called the \textit{contact order} of $l$.
\end{enumerate}
These data are subject to the only condition that if the edge $l$ is a loop, then $c_{l}=0$. 
\end{defn}

Note that the set $V_{n}(G)$ can be empty. If the contact order of an edge $l$ is zero, then $l$ is called the \textit{non-distinguished edge}, otherwise is called a \textit{distinguished edge}. Two vertices are called \textit{adjacent} if they are connected by an edge. Denote by $\underline{G}$ the underlying graph of $G$, obtained by removing all weights.

\begin{defn}\label{defn:Orientation}
Consider a weighted graph $G$ as in the above definition. An \textit{orientation} on $G$ is an orientation on the underlying graph $\underline{G}$, except that we allow some edges to be \textit{non-oriented}, i.e. an edge with two directions. Consider an edge $l$ from $v_{1}$ to $v_{2}$ under the orientation. Then $v_{1}$ is called the \textit{initial vertex of $l$}, and $v_{2}$ is called the \textit{end vertex of $l$}. We denote this by $v_{1}\leq v_{2}$. If $l$ is orientated, then we write $v_{1}<v_{2}$.

An orientation on $G$ is called \textit{compatible} if
\begin{enumerate}
 \item An edge $l\in E(G)$ is non-oriented if and only if $c_{l}=0$.
 \item If $v\in V_{n}(G)$, then for any other adjacent vertex $v'$ of $v$ we have $v\leq v'$.
\end{enumerate}
Note that if $v, v'\in V_{n}(G)$, then any edges between them is non-oriented. The graph $G$ is called a \textit{marked graph}, if it is a weighted graph with a compatible orientation.
\end{defn}

A \textit{path} is a non-repeated sequence of edges $(l_{1},l_{2},\cdots,l_{m})$, such that the end vertex of $l_{j}$ is the initial vertex of $l_{j+1}$. Such a path is called a \textit{cycle} if the initial vertex of $l_{1}$ is the end vertex of $l_{m}$. A cycle is called \textit{strict} if it contains at least one oriented edges. A vertex $v\in V(G)$ is called \textit{minimal} (respectively \textit{maximal}) if it is not the end (respectively initial) vertex of any oriented edge. Thus by condition (2) above, any vertex $v\in V_{n}(G)$ is minimal.

\begin{constr}\label{cons:AssoMonoid}
Consider a marked graph $G$ as in Definition \ref{defn:Orientation}. For each edge $l\in E(G)$ (respectively each vertex $v\in V(G)$), we introduce a variable $e_{l}$ (resp. $e_{v}$), which is called the \textit{element associated to $l$} (resp. $v$). For any $v\in V_{n}(G)$, we set
\begin{equation}\label{equ:NonDegVer}
h_{v}: \ \ \ e_{v}=0.
\end{equation} 

Consider an edge $l\in E(G)$ with initial vertex $v_{1}$ and end vertex $v_{2}$. We associate an equation 
\begin{equation}\label{equ:GraphEqu}
h_{l}: \ \ \ e_{v_{2}}=e_{v_{1}} + c_{l}\cdot e_{l}.
\end{equation}

Consider the monoid
\begin{equation}\label{equ:CoarseMonoid}
M(G)=\Big\langle\ e_{v},e_{l}\ \Big|\  v\in V_{}(G),\ l\in E(G) \ \Big\rangle\ \bigg/ \big\langle\ h_{l}, h_{v}\ \big|\  l\in E(G),\ v\in V_{n}(G) \  \big\rangle
\end{equation}
Denote by $T(G)$ the torsion part of $M(G)^{gp}$. Then we have the following composition
\[M(G)\to M(G)^{gp} \to M(G)^{gp}/T(G).\]
Denote by $N(G)$ the image of $M(G)$ in $M(G)^{gp}/T(G)$, and $\CharM(G)$ the saturation of $N(G)$ in $M(G)^{gp}/T(G)$.
\end{constr}

\begin{defn}
The monoid $\CharM(G)$ constructed above is called the \textit{associated monoid} of the marked graph $G$.
\end{defn}

Note that $N(G)$ is the image of $M(G)$ in $\CharM(G)$. By the definition of $\CharM(G)$ and Proposition \ref{prop:AdjSat}, we have the following:

\begin{lem}\label{lem:WeightInMonoid}
By viewing $N(G)$ and $\CharM(G)$ as sub-monoids of $\CharM(G)^{gp}=M(G)^{gp}/T(G)$, we have that for any $a\in \CharM(G)$ there exists $b\in N(G)$ and a positive integer $m$ such that $b=m\cdot a$.
\end{lem}

\begin{defn}\label{defn:AdGraph}
The marked graph $G$ is called \textit{admissible} if $\CharM(G)$ is a sharp monoid, and the image of $e_{l}$ in $\CharM(G)$ is non-trivial for all $l\in E(G)$.
\end{defn}

\begin{cor}\label{defn:AdmGraph}
If $G$ is admissible, then there is no strict cycle in $G$.
\end{cor}
\begin{proof}
If there is a strict cycle $(l_{1},\cdots,l_{k})$, then we have $\sum_{i=1}^{k}c_{l_{i}}e_{l_{i}}=0$. The strictness implies that at least one of the $c_{l_{i}}$ is non-zero. Thus, the monoid $\CharM(G)$ fails to be sharp, which contradicts the admissibility assumption. 
\end{proof}

Note that when $G$ is admissible, the monoid $\CharM(G)$ generates a strongly convex rational cone $C(\CharM(G))$ in the vector space $\CharM(G)^{gp}\otimes \Q$ (see \cite[Page 4]{Toric}).

\begin{lem}\label{lem:IrrEle}
Consider an irreducible element $e\in \CharM(G)$, where $G$ is admissible. Assume that $e$ lies on an extremal ray of $C(\CharM(G))$. Then at least one of the following holds:
\begin{enumerate}
 \item There is a positive integer $n$ and a minimal vertex $v$, such that $n\cdot e=e_{v}$.
 \item There is a positive integer $n$ and an edge $l$, such that $n\cdot e=e_{l}$. 
\end{enumerate}
\end{lem}
\begin{proof}
Let $n$ be the minimal positive integer such that $n\cdot e\in N(G)$. Assume that $n\cdot e=b+c$ with $b,c\in N(G)$. Note that $e$ generates an extremal ray of the strongly convex rational cone $C(\CharM(G))$.
Thus we have positive numbers $n_{1}$ and $n_{2}$ such that $b=n_{1}\cdot e$ and $c=n_{2}\cdot e$. The minimality of $n$ implies that either $b=0$ or $c=0$. Since $b$ and $c$ in $N(G)$ are elements associated to edges or vertices, the element $n\cdot e$ must satisfy one of the two possibilities above. 
\end{proof}

\subsection{Marked graphs associated to log maps}
Consider a log map $\xi=(C\to S,\sM_{S}, f)$ over a geometric point $S$, such that the log structure $\sM_{S}$ is fs.

\begin{constr}\label{cons:DualGraphMap}
We construct a weighted graph $G_{\xi}$ of $\xi$ with an orientation as in Definition \ref{defn:Orientation}:
\begin{enumerate}
 \item The underlying graph $\underline{G}_{\xi}$ is given by the dual graph of the curve $C$.
 \item The subset $V_{n}(G)$ consists of the vertices corresponding to non-degenerate components.
 \item For each edge $l\in E(G_{\xi})$, we associate a non-negative integers $c_{l}$, where $c_{l}$ is the contact order of the node $l$ as in Definition \ref{defn:DistNode}.
 \item Let $l\in E(G_{\xi})$ be a node joining two irreducible components $v_{1}, v_{2}\in V(G_{\xi})$. Then we define an orientation by putting $v_{1}\leq v_{2}$ if $v_{1}$ is the lower component, and $v_{2}$ is the upper component of $l$ as in Definition \ref{defn:DegNodeLow}.
\end{enumerate}

\end{constr}

Consider a node $l\in E(G_{\xi})$. Denote by $e_{l}'$ the element in $\CharM_{S}$ smoothing $l$, and by $e_{l}$ the element associated to $l$ in $\CharM(G_{\xi})$. Then consider an irreducible component $v\in V(G_{\xi})$. Denote by $e_{v}'$ the degeneracy of $v$ in $\xi$, and by $e_{v}\in \CharM(G_{\xi})$ the element associated to $v$.  
We define a correspondence 
\begin{equation}\label{equ:CanMapChar}
e_{l}\mapsto e_{l}' \ \ \ \mbox{and} \ \ \ e_{v}\mapsto e_{v}'
\end{equation}

\begin{prop}\label{prop:CanChar}
Assuming that $\sM_{S}$ is fs, the correspondence (\ref{equ:CanMapChar}) induces a canonical morphism of monoids 
\[\phi: \CharM(G_{\xi}) \to \CharM_{S}.\]
\end{prop}
\begin{proof}
Note that (\ref{equ:CanMapChar}) induces a map $M(G_{\xi})\to \CharM_{S}$. By Proposition \ref{prop:AdjSat}, this induces a unique map $\phi': M(G_{\xi})^{Sat}\to \CharM_{S}$. Since the monoid $\CharM_{S}$ is sharp, if $e\in M(G_{\xi})^{Sat}$ is torsion, then $\phi'(e)=0$. Thus, there is a unique map $\phi: \CharM(G_{\xi}) \to \CharM_{S}$ induced by $\phi'$.
\end{proof}

\begin{cor}\label{cor:DualGraphAdm}
The graph $G_{\xi}$ is an admissible marked graph.
\end{cor}
\begin{proof}
The compatibility of the orientation follows from Lemma \ref{lem:DegJump}. Let us consider the admissibility. First note that $e_{l}$ is non-trivial for any $l\in E(G)$, since its image in $\CharM_{S}$ is the element smoothing the node $l$, which is non-trivial. For any element $a\in \CharM(G_{\xi})$, if $a$ is invertible, then by Lemma \ref{lem:WeightInMonoid}, there exists some positive integer $m$ such that $m\cdot a =\sum_{i}d_{i}e_{i}$, where $e_{i}$ are elements associated to some edges or vertices, and $d_{i}$ are non-negative integers. Since the monoid $\CharM_{S}$ is sharp, we have $\phi(a)=\sum_{i}d_{i}\phi(e_{i})=0$ in $\CharM_{S}$. If $d_{i}\neq 0$, then $\phi(e_{i})=0$, which implies that $e_{i}$ is the element associated to a non-degenerate component. Thus we have $e_{i}=0$ in $\CharM(G_{\xi})$. This implies that $a=0$ in $\CharM(G_{\xi})$, which proves the statement. 
\end{proof}

\begin{defn}
We call $G_{\xi}$ the {\em marked graph of $\xi$}.
\end{defn}

\subsection{Minimal logarithmic maps}
We still consider a log map $\xi=(C\to S,\sM_{S}, f)$ over a geometric point $S$ with fs log structure $\sM_{S}$.

\begin{defn}\label{defn:Minimal}
The log map $\xi$ over $S$ is called \textit{minimal} if the induced canonical map $\phi$ in Proposition \ref{prop:CanChar} is an isomorphism of monoids. A family of log maps $\xi_{T}$ over a scheme $T$ is called {\em minimal} if each geometric fiber is minimal.
\end{defn}

\begin{prop}[\textbf{Openness of minimal log maps}]\label{prop:MinOpen}
Let $\xi=(C\to S,\sM_{S},f)$ be a family of log maps over a scheme $S$, and assume that $\xi_{\bar{s}}$ is minimal for some point $s\in S$. Then there exists an \'etale neighborhood of $s$ with all geometric fibers minimal.
\end{prop}
\begin{proof}
Shrinking $S$, we may assume that $S$ is connected, and we have a chart $\beta: \CharM_{S,\bar{s}} \to \sM_{S}$ by Proposition \ref{prop:ChartLogStr}. We next show that for any $\bar{t}\in S$, the fiber $\xi_{\bar{t}}$ is minimal.

Denote by 
\[K_{\bar{t}}=\{\ a\in \CharM_{S,\bar{s}}\ | \ \mbox{$\beta(a)$ is invertible at $\bar{t}$}\ \}.\]
Note that $K_{\bar{t}}$ is a submonoid of $\CharM_{S,\bar{s}}$. Consider the following composition
\[\CharM_{S,\bar{s}}\to\CharM_{S,\bar{s}}^{gp}\to\CharM_{S,\bar{s}}^{gp}/K_{\bar{t}}^{gp}.\]
By \cite[3.5]{LogStack}, we have $\CharM_{S,\bar{s}}^{gp}/K_{\bar{t}}^{gp}\cong\CharM_{S,\bar{t}}^{gp}$. The above composition induces a map $q:\CharM_{S,\bar{s}}\to \CharM_{S,\bar{t}}$, which is exactly the specialization map as in \cite[3.5(iii)]{LogStack}. We construct a new graph from the marked graph $G_{\xi_{\bar{s}}}$ as follows.
\begin{enumerate}
 \item For an edge $l\in E(G_{\xi,\bar{s}})$, if $q(e_{l})=0$, then we contract $l$, and identify the two end vertices of $l$ and the corresponding associated elements.
 \item For a vertex $v\in V(G_{\xi,\bar{s}})$, if $q(e_{v})=0$, then we put $e_{v}=0$ in $G'$. 
\end{enumerate}
Other vertices and edges in $G_{\xi_{\bar{s}}}$, and their associated elements and contact orders remain the same. We denote by $G'$ the resulting graph. 

First note that the underlying graph $\underline{G}'$ is the dual graph of $C_{\bar{t}}$, since an edge $l\in E(G_{\xi,\bar{s}})$ gets contracted if and only if the corresponding node is smoothed out over $\bar{t}$. Furthermore, the orientation on $G_{\xi_{\bar{s}}}$ induces a natural orientation on $G'$. Since all contact orders remain the same, the graph $G'$ is in fact the marked graph $G_{\xi_{\bar{t}}}$ of $\xi_{\bar{t}}$.

The construction of $G'$ gives a map of monoids:
\begin{equation}\label{}
M(G_{\xi_{\bar{s}}})\to M(G_{\xi_{\bar{t}}}) \to \CharM_{G_{\xi_{\bar{t}}}}.
\end{equation}
By the same argument in Proposition \ref{prop:CanChar}, we obtain a canonical map of monoids:
\begin{equation}\label{equ:generalize}
q':\CharM_{S,\bar{s}}\cong \CharM(G_{\xi_{\bar{s}}})\to \CharM(G_{\xi_{\bar{t}}}),
\end{equation}
which gives the following commutative diagram:
\begin{equation}\label{diag:GeneralizeChar}
\xymatrix{
&\CharM_{S,\bar{s}} \ar[ld]_{q'} \ar[rd]^{q}&\\
\CharM(G_{\xi_{\bar{t}}}) \ar[rr] && \CharM_{S,\bar{t}},
}
\end{equation}
where the bottom map is the canonical map as in Proposition \ref{prop:CanChar}. Consider the induced commutative diagram:
\begin{equation}
\xymatrix{
&\CharM_{S,\bar{s}}^{gp} \ar[ld]_{(q')^{gp}} \ar[rd]^{q^{gp}}&\\
\CharM(G_{\xi_{\bar{t}}})^{gp} \ar[rr] && \CharM_{S,\bar{t}}^{gp},
}
\end{equation}
Note that both $q^{gp}$ and $(q')^{gp}$ are surjective maps. By the construction of $q$, the group $K^{gp}_{\bar{t}}$ is the kernel of $q^{gp}$. On the other hand, the construction of $G'$ and the fact that $G'=G_{\xi_{\bar{t}}}$ implies that $K^{gp}_{\bar{t}}$ is also the kernel of $(q')^{gp}$. Since the monoids in (\ref{diag:GeneralizeChar}) are fine and saturated, the map $\CharM(G_{\xi_{\bar{t}}}) \to \CharM_{S,\bar{t}}$ is an isomorphism. This proves the statement. 
\end{proof}

\begin{defn}\label{defn:min-pre-stab-map}
Denote by $\sK_{g,n}^{pre}(X^{log})$ the stack parameterizing minimal log maps to $X^{log}$, with the fixed genus $g$, and $n$-markings.
\end{defn}

\begin{cor}
$\sK_{g,n}^{pre}(X^{log})$ is an open substack of the stack $\sL\sM_{g,n}(X^{log})$ of log maps, hence is an algebraic stack.
\end{cor}
\begin{proof}
This follows from Theorem \ref{thm:AlgStackLogMap} and Proposition \ref{prop:MinOpen}. 
\end{proof}

\subsection{Stable log maps}\label{ss:LogStableMap}

\begin{defn}\label{defn:MinStable}
A log map $\xi=(C\to S,\sM_{S},f)$ over a geometric point $S$ is called \textit{stable} if its underlying map is stable in the usual sense, and $\sM_{S}$ is fs. A family of log maps $\xi_{T}$ over a scheme $T$ is called stable if its geometric fibers are stable. A stable log map is called \textit{minimal stable} if it satisfies the minimality condition as in Definition \ref{defn:Minimal}.
\end{defn}

Similarly, we can work over log schemes rather than the usual schemes. Then we have the following:

\begin{defn}
A log map $\xi^{log}$ over a fs log scheme $(S,\sM_{S})$ as in Definition \ref{defn:LogMapOverLog} is called stable, if its underlying map is stable in the usual sense.
\end{defn}

\begin{con}\label{con:DataGamma}
In this paper, we fix the discrete data $\Gamma = (\beta,g,n,\textbf{c})$ where
 \begin{enumerate}
   \item $\beta \in H^{2}(X,\Z)$ is a curve class in $X$;
   \item $n$ and $g$ are two non-negative integers;
   \item $\textbf{c}=(c_{i})_{i=1}^{n}$ is a set of non-negative integers such that 
   \begin{equation}\label{equ:ContactOrder}
   \sum_{i=1}^{n}c_{i} = c_{1}(L_{}^{\vee})\cap\beta
   \end{equation}
where $c_{1}(L^{\vee})$ is the first chern class of the line bundle $L^{\vee}$ as in Conventions \ref{con:LogTarget}.
 \end{enumerate}
\end{con}

\begin{defn}\label{defn:GammaLogMap}
A minimal stable log map $\xi=(C\to S,\sM_{S},f)$ over a geometric point $S$ is called a \textit{$\Gamma$-minimal stable log} if
\begin{enumerate}
 \item The source curve $(C\to S,\sM_{S})$ is a log pre-stable curve of genus $g$ with $n$ marked points. 
 \item $f_{*}[C]=\beta$.
 \item For any $i\in \{1,2,\cdots, n\}$, the contact order along section $\Sigma_{i}$ is given by $c_{i}$.
\end{enumerate}
A log map $\xi'$ over an arbitrary scheme $T$ is called {\em $\Gamma$-minimal stable log} if its geometric fibers are all $\Gamma$-minimal stable log. The arrows between stable log maps are the same as the arrow of log maps in Definition \ref{defn:LogMapIso}.
\end{defn}

\begin{defn}\label{defn:MinStableStack}
Denote by $\sK_{g,n}^{}(X^{log},\beta)$ the stack parameterizing minimal stable log maps with genus $g$, $n$ marked points, and curve class $\beta$. Let $\sK_{\Gamma}(X^{log})$ be the stack parameterizing $\Gamma$-minimal stable log maps. These are substacks of $\sL\sM_{g,n}(X^{log})$ as in Theorem \ref{thm:AlgStackLogMap}.
\end{defn}

\begin{thm}\label{thm:StabLogAlg}
The stack $\sK_{g,n}^{}(X^{log},\beta)$ is an open substack of $\sK^{pre}_{n,g}(X^{log})$, hence is algebraic.
\end{thm}
\begin{proof}
Note that the stability condition is a condition on the underlying map, which is well known to be open. 
\end{proof}

\begin{rem}
Denote by $\Lambda$ the set of discrete data $\Gamma$ as in Convention \ref{con:DataGamma} with fixed $g$, $n$ and $\beta$. Note that $\Lambda$ is a finite set. By Lemma \ref{lem:ContactOpen}, we have the disjoint union of open and closed substacks
\[\sK_{g,n}^{}(X^{log},\beta)=\coprod_{\Gamma\in\Lambda}\sK_{\Gamma}(X^{log}).\]
\end{rem}

\subsection{A quasi-finiteness result}\label{sec:finiteness}

We fix a minimal stable log map $\xi_{1}=(\pi:C\to S,\sM_{1},f_{1})$ over a geometric point $S$. Denote by $G$ the marked graph of $\xi_{1}$. Choose a chart $\beta:\CharM(G)\to \sM_{S}$. Let $\sN_{1}\subset \sM_{S}$ be the sub-log structure generated by $\beta(N(G))$. Since different choice of $\beta$ only differ by invertible elements, the log structure $\sN_{1}$ does not depend on the choice of $\beta$. 

Consider the fine log scheme $(S,\sN_{1})$ with the sub-log structure $\sN_{1}\subset \sM_{1}$ induced by $N(G)$. Since the map of characteristics $\CharM_{S}^{C/S}\to \CharM(G)$ factors through $N(G)$, the structure map $\sM_{S}^{C/S}\to \sM_{S}$ factors through $\sN_{1}$. This induces a log curve $(C\to S, \sN_{1})$. Denote by $\sN_{C,1}$ the log structure of $(C\to S, \sN_{1})$ on $C$, and by $\sM_{C,1}$ the log structure of $\xi_{1}$ on $C$. Then $\sN_{C,1}$ is a sub-log structure of $\sM_{C,1}$. Again by considering the map of characteristics, we see that the log map $f_{1}^{\flat}:f_{1}^{*}\sM_{X}\to \sM_{C,1}$ factors through $\sN_{C,1}$. Then we obtain a log map $g_{1}:(C,\sN_{C,1})\to X^{log}$. Denote by $\xi'_{1}=(C\to S,\sN_{1}, g_{1})$ the log map over $S$. 

\begin{defn}
The log map $\xi'_{1}$ is called the {\em coarse log map} of $\xi_{1}$.
\end{defn}

\begin{rem}
The log structures of coarse log maps are in general not saturated. The above construction yields a natural arrow $\xi_{1}\to \xi_{1}'$. 
\end{rem}

\begin{cor}
The pair $(\xi_{1},\xi_{1}\to \xi_{1}')$ is unique up to a unique isomorphism
\end{cor}
\begin{proof}
This follows from the uniqueness of the log structure $\sN_{1}$.
\end{proof}

The following result reveals the importance of the notion of coarse log maps.

\begin{lem}\label{lem:quasi-finit}
Consider another minimal stable log map $\xi_{2}=(C\to S,\sM_{2},f_{2})$ whose underlying structure and marked graph are identical to those of $\xi_{1}$. Then there exists a canonical isomorphism of coarse log maps $\xi_{2}'\cong\xi_{1}'$.
\end{lem}
\begin{proof}
Let $\xi_{2}'=(C\to S,\sN_{2},g_{2})$ be the coarse log map of $\xi_{2}$. Denote by $\sN_{C,i}$ the log structures on $C$ corresponding to $\xi_{i}'$ for $i=1,2$. Consider the solid diagram of log structures on $C$:
\begin{equation}\label{diag:on-curve-coarse}
\xymatrix{
 && \sN_{C,1} \ar@{-->}[dd]^{\psi_{\sN}} && \\
\underline{f}^{*}\sM_{X} \ar[rru]^{g_{1}^{\flat}} \ar[rrd]_{g_{2}^{\flat}} &&&&  \sM_{C}^{C/S} \ar[llu]_{\phi_{1}} \ar[lld]^{\phi_{2}} \\
 && \sN_{C,2} &&
}
\end{equation}
where $\sN_{C,i}$ is the associated log structure on $C$ with respect to $\sN_{i}$. We will first construct the dashed arrow $\psi_{\sN}$, which makes (\ref{diag:on-curve-coarse}) commutative. 

Since the underlying structures of $\xi_{1}$ and $\xi_{2}$ are identical, it would be enough to construct $\psi_{\sN}: \pi^{*}\sN_{1}\to \pi^{*}\sN_{2}$. We fix a chart $\beta_{i}: N(G)\to \sN_{i}$. Consider a section $e\in\pi^{*}\sN_{i}$ for $i=1,2$. We want to define the image $\psi_{\sN}(e)$. For this, it is enough to assume that $e=\pi^{*}(\beta_{1}(\bar{e}))$ for some element $\bar{e}\in N(G)$ associated to a vertex or an edge.

We first assume that $\bar{e}$ is an element associated to an edge $l\in E(G)$. Then there exists a section $e_{l}\in \sM_{S}^{C/S}$, such that $\phi_{1}(e_{l})=e$. Thus, to make (\ref{diag:on-curve-coarse}) commutative one has to define
\begin{equation}\label{equ:node-def}
\psi_{\sN}(e)=\phi_{2}(e_{l})
\end{equation}

We then consider the case that $\bar{e}$ is an element associated to a vertex $v\in V(G)$. We can assume that $v$ is degenerate, otherwise $\bar{e}$ is trivial in $N(G)$. We may restrict (\ref{diag:on-curve-coarse}) to a small neighborhood of a non-distinguished point as in Definition \ref{defn:DistNode} on the component corresponding to $v$. Denote by $\delta\in \underline{f}^{*}\sM_{X}$ a local generator. Since on the level of characteristic we have $\bar{\phi}_{1}(\delta)=\bar{e}$ in $\overline{\sN}_{C,1}$, we may assume $f^{\flat}_{1}(\delta)=e$. Thus, one has to define
\begin{equation}\label{equ:vertex-def}
\psi_{\sN}(e)=f_{2}^{\flat}(\delta)
\end{equation}

It is clear that (\ref{equ:node-def}) and (\ref{equ:vertex-def}) make (\ref{diag:on-curve-coarse}) commutative. Now we need to show that the map $\psi_{\sN}$ is well-defined. The only issue here is to check the left triangle of (\ref{diag:on-curve-coarse}) at a distinguished node. 

We may assume that $p$ is a distinguished node joining two components $Z_{1}$ and $Z_{2}$ with contact order $c$. We need to check that the map $\psi_{\sN}$ defined at the generic points of the two components can be extended to $p$. Let $x_{j}$ be the local coordinate of $Z_{j}$ at $p$. Denote by $\log x_{j}$ the corresponding section in both $\sN_{C,1}$ and $\sN_{C,2}$. Then we automatically have 
\[\psi_{\sN}(\log x_{i})=\log x_{i}.\]
Without loss of generality, assume that the orientation of the node is given by $Z_{1}>Z_{2}$. Then locally at $p$, we have
\begin{equation}\label{equ:low-comp}
f_{1}^{\flat}(\delta)=e_{Z_{2}} + c\cdot \log x_{2}
\end{equation}
and
\begin{equation}\label{equ:up-comp}
f_{2}^{\flat}(\delta)=\psi_{\sN}(e_{Z_{2}}) + c\cdot \log x_{2}
\end{equation}
where $\psi_{\sN}(e_{Z_{2}})$ is defined by (\ref{equ:vertex-def}) at some smooth non-marked point of $Z_{2}$. 

On the other hand, we have a section $e_{l}\in \sM_{S}^{C/S}$ such that 
\[e_{l}= \log x_{1} + \log x_{2}.\]
We identify $e_{l}$ with the corresponding sections in $\sN_{S,i}$ and $\sN_{C,i}$ via $\phi_{i}$. Now combining this with (\ref{equ:low-comp}) and (\ref{equ:up-comp}), and generalizing to a smooth non-marked point of $Z_{1}$, we get 
\[
f_{1}^{\flat}(\delta) + c\cdot \log x_{1}=e_{Z_{2}} + c\cdot e_{l} 
\]
and
\[
f_{2}^{\flat}(\delta) + c\cdot \log x_{1}=\psi_{\sN}(e_{Z_{2}}) + c\cdot e_{l} .
\]
Since $x_{1}$ is invertible away from $p$, it make sense to write ``$- log x_{1}$'' as above. Hence 
\[\psi_{\sN}(e_{Z_{2}}+c\cdot e_{l}) = \psi_{\sN}(e_{Z_{2}}) + c\cdot e_{l}.\]
This proves that the definitions of $\psi_{\sN}$ on the two components meeting at $p$ are compatible. Therefore the map $\psi_{\sN}$ is well-defined. In particular, the above construction gives a canonical isomorphism $\pi^{*}\sN_{1}\cong \pi^{*}\sN_{2}$.

Note that $\pi: C\to S$ forms a flat cover. Since log structures can be glued under fppf topology \cite{LogStack}, the map $\psi_{\sN}$ descends to a well-defined isomorphism of log structures $\sN_{1}\to \sN_{2}$, which induces an isomorphism $\xi_{1}'\cong\xi_{2}'$. The uniqueness follows from that of $\psi_{\sN}$ in the above construction. 
\end{proof}

\begin{prop}\label{prop:quasi-finite}
There are at most finitely many minimal stable log maps over a geometric point with fixed underlying map and marked graph.
\end{prop}
\begin{proof}
Fixing a discrete data $\Gamma$, the number of possible choices of contact orders along marked points is finite. It is enough to show that the number of $\Gamma$-minimal stable log maps with fixed underlying structure and marked graph is at most finite.

Consider a $\Gamma$-minimal stable log map $\xi=(C\to S,\sM_{S},f)$ over a geometric point $S$ with the fixed underlying structure $\underline{\xi}$ and marked graph $G$. Denote by $\xi'=(C\to S,\sN_{S},g)$ the coarse log map of $\xi$ over $S$. Then we have the natural arrow $\xi\to\xi'$.

On the other hand, consider the saturation map $\sS: (S^{\sS},\sM)=(S,\sN_{S})^{\sS at}\to (S,\sN_{S})$. Denote by $\xi^{\sS}$ the stable log map over $S^{\sS}$ obtained by pulling back $\xi'$ via $\sS$. It is easy to check that $\xi^{\sS}$ is minimal. By \cite[Chapter II,2.4.5]{Ogus}, we have a canonical map $h':(S,\sM_{S})\to (S^{\sS},\sM)$, such that $h=\sS\circ h'$. This induces an arrow of minimal stable log maps $\xi\to\xi^{\sS}$. By comparing the characteristic, it is easy to see that $h'$ is a strict closed immersion.

Since the underlying map of $\sS$ is finite, the statement follows from Lemma \ref{lem:quasi-finit}.
\end{proof}

\subsection{Finiteness of automorphisms}\label{ss:FiniteAuto}
Let $\xi=(C\to S, \sM_{S}, f)$ be a minimal stable log map over a geometric point $S$. We fix a chart $\CharM_{S}\to \sM_{S}$, and identify the elements $e_{v}$ and $e_{l}$ for $v\in V(G_{\xi})$ and $l\in E(G_{\xi})$ with their images in $\sM_{S}$.
\begin{prop}\label{prop:FiniteAuto}
Notations as above, the set $\sA ut_{S}(\xi)(S)$ is finite.
\end{prop}
\begin{proof}
Note that the set of underlying automorphisms of $\underline{f}$ is finite. Fixing a underlying automorphism $(\underline{\rho},id_{S})$, it is enough to show that there are finitely many automorphisms of $\xi$ whose underlying structure are given by $(\underline{\rho},id_{S})$. For simplicity, we assume that $\underline{\rho}=id_{C}$, and other cases can be proved similarly.

Let $(\rho,\theta)$ be an automorphism with the underlying structure given by $(id_{C},id_{S})$. First we consider a node $l\in E(G_{\xi})$. Denote by $x$ and $y$ the local coordinates of $l$. We can choose $x$ and $y$ so that $e_{l}=\log x +\log y$. Note that we have 
\begin{equation}\label{equ:unsat-underlying}
\rho^{\flat}(e_{l})= \rho^{\flat}(\log x) + \rho^{\flat}(\log y)=\log \rho^{*}(x) + \log \rho^{*}(y)=\log x+\log y.
\end{equation}
Since $\underline{\rho}=id_{C}$, the element $e_{l}$ is fixed by $\rho$ for any $l$. The same argument shows that the log structure from the marked points is also fixed by $\rho$.

Now consider a minimal vertex $v\in V(G_{\xi})$. Locally on the component of $v$, we have 
\[f^{\flat}(\delta)=e_{v}+\log h,\] 
where $h$ is a local invertible section. Note that we have 
\begin{equation}\label{equ:unsat-minivert}
\rho^{\flat}(e_{v}+\log h)=\rho^{\flat}(e_{v})+\log \rho^{*}(h)=\rho^{\flat}(e_{v})+\log h.
\end{equation}
Since $\rho$ fixes the section $f^{\flat}(\delta)$, the map $\rho^{\flat}$ also fixes the element $e_{v}$. Thus, the automorphism $(\rho,\theta)$ acts trivially on all elements associated to vertices and edges of $G_{\xi}$. By Lemma \ref{lem:WeightInMonoid} and \ref{lem:IrrEle}, the number of choices of $(\rho,\theta)$ is finite. 
\end{proof}

Denote by $\xi'=(C\to S,\sN_{S},g)$ the coarse log map of $\xi$. Then any automorphism $(\rho,\theta)\in \sA ut_{S}(\xi)(S)$ induces a unique automorphism of $\xi'$. Indeed, the above proof of Proposition \ref{prop:FiniteAuto} implies the following:

\begin{cor}\label{cor:unsat-log-iso}
Consider an automorphism $(\rho,\theta)\in \sA ut_{S}(\xi)(S)$. The induced map of sub-log structures $\theta^{\flat}:\sN_{S}\to\sN_{S}$ is uniquely determined by the underlying automorphism $(\underline{\rho},id_{S})$. In particular, the automorphism of $\xi'$ induced by $(\rho,\theta)$ is uniquely determined by the underlying automorphism $(\underline{\rho},id_{S})$.
\end{cor}

Denote by $\underline{\xi}$ the usual stable maps obtained by removing log structures on $\xi$. We can strengthen the result of Proposition \ref{prop:FiniteAuto} as follows:
\begin{lem}\label{lem:auto-inject}
The map of groups $\sA ut_{S}(\xi)(S)\to \sA ut_{S}(\underline{\xi})(S)$ is an isomorphism.
\end{lem}
\begin{proof}
Consider an element $(\underline{\rho},id_{S})\in \sA ut_{S}(\underline{\xi})(S)$. It is enough to show that there is exactly one element $(\rho,\theta)\in \sA ut_{S}(\xi)(S)$, which is the pre-image of $(\underline{\rho},id_{S})$. Consider the following diagram:
\begin{equation}\label{diag:iso-under-log}
\xymatrix{
\sN_{S} \ar[d] \ar@{^{(}->}[r] & \sM_{S} \ar@{-->}[d]^{\theta^{\flat}}  \\
\sN_{S} \ar@{^{(}->}[r] & \sM_{S},
}
\end{equation}
where the left vertical arrow can be constructed similarly by (\ref{equ:unsat-underlying}) and (\ref{equ:unsat-minivert}), which is uniquely determined by the underlying map $(\underline{\rho},id_{S})$. Corollary \ref{cor:unsat-log-iso} implies that any $(\rho,\theta)$ over $(\underline{\rho},id_{S})$ induces the same map $\sN_{S}\to \sN_{S}$ as in (\ref{diag:iso-under-log}). Hence to find $(\rho,\theta)$, it is equivalent to find the dashed arrow $\theta^{\flat}$, which makes the above diagram (\ref{diag:iso-under-log}) commutative. By the adjointness of saturation and inclusion functors of log structures as in \cite[Chapter II,2.4.5]{Ogus}, we have the following commutative diagram:
\begin{equation}\label{diag:sat-iso}
\xymatrix{
(S,\sM_{S}) \ar@{^{(}->}[r]^{i} \ar[dr]^{\theta'} & (S,\sN_{S})^{\sS at} \ar[r] \ar[d]^{\exists ! \cong} \ar[r] & (S,\sN_{S}) \ar[d]^{}\\
 & (S,\sN_{S})^{\sS at} \ar[r] & (S,\sN_{S})
}
\end{equation}
where $(S,\sN_{S})^{\sS at}$ is the saturation of $(S,\sN_{S})$. 

Denote by $\xi'$ the coarse log curve of $\xi$ over $S$. Then by taking the saturation, we obtain a minimal stable log map $(\xi')^{\sS}$ over $(S,\sN_{S})^{\sS at}$. Note that the left triangle of (\ref{diag:sat-iso}) induces a commutative diagram of minimal stable log maps:
\begin{equation}
\xymatrix{
\xi \ar[r] \ar[dr] & (\xi')^{\sS} \ar[d]\\
 & (\xi')^{\sS}
}
\end{equation}
This gives a unique $\theta^{\flat}$ as in (\ref{diag:iso-under-log}), hence a unique isomorphism of $\xi$.
\end{proof}

\begin{prop}\label{prop:remove-log-rep}
The natural map $\sK_{\Gamma}(X^{log})\to \sK_{g,n}(X,\beta)$ by removing log structures from minimal stable log maps is representable.
\end{prop}
\begin{proof}
This follows from Lemma \ref{lem:auto-inject} and \cite[8.1.1]{LMB}. 
\end{proof}

\section{The stack of minimal log maps as category fibered over $\LogSch^{fs}$}\label{sec:UniverMin}
By the construction in last section, the stacks $\sK_{g,n}^{pre}(X^{log})$, $\sK_{g,n}(X^{log},\beta)$, and $\sK_{\Gamma}(X^{log})$ as open sub-stacks of $\sL\sM_{g,n}(X^{log})$, are fibered categories over $\Sch$, parameterizing minimal log maps over usual schemes with various numerical conditions. In this section, we give a different categorical explanation as categories fibered over $\LogSch^{fs}$.

\subsection{The universal property of minimal log maps}
In this subsection, we fix a log map $\xi=(C\to S,\sM_{S},f:(C,\sM_{C})\to(X,\sM_{X}))$ such that the log structure $\sM_{S}$ is fs. 

\begin{prop}\label{prop:UnivMinLog}
There exists a minimal log map over $S$
\[\xi_{min}=(C\to S,\sM_{S}^{min},f_{min}:(C,\sM_{C}^{min})\to (X,\sM_{X})),\] 
and a map of fs log schemes $\Phi:(S,\sM_{S})\to (S,\sM_{S}^{min})$, which fits in the following commutative diagram 
\begin{equation}\label{diag:UniverMinLogMap}
\xymatrix{
&&(X,\sM_{X}) \\
(C,\sM_{C}) \ar@/^/[urr]^{f} \ar[r]_{\Phi_{C}} \ar[d]& (C,\sM_{C}^{min}) \ar@/_/[ur]_{f_{min}} \ar[d]&\\
(S,\sM_{S})  \ar[r]^{\Phi} &(S,\sM^{min}_{S})}
\end{equation}
where the square is a caterian square of log schemes. Furthermore, the datum $(\Phi,\xi_{min})$ is unique up to a unique isomorphism.
\end{prop}
\begin{proof}
Note that the statement is local on $S$. Then the proposition follows from Lemmas \ref{lem:UniqMinMap}, \ref{lem:MinCharUniq}, \ref{lem:CompCurveLog}, and \ref{lem:Generalization} below. 
\end{proof}

By Construction \ref{cons:DualGraphMap}, for each geometric point $\bar{t}\in S$ we can associate a marked graph $G_{\xi_{\bar{t}}}$ to the fiber $\xi_{\bar{t}}$. It was shown in Lemma \ref{cor:DualGraphAdm} that $G_{\xi_{\bar{t}}}$ is admissible. By Proposition \ref{prop:CanChar}, we have a canonical morphism of monoids 
\begin{equation}\label{equ:CanMapUniv}
\phi_{\bar{t}}:\CharM(G_{\xi_{\bar{t}}})\to \CharM_{S,\bar{t}}
\end{equation}

\begin{lem}\label{lem:UniqMinMap}
Assume that we have a log pre-stable curve $(C\to S,\sM_{S}^{min})$ and a morphism $\Phi:(S,\sM_{S})\to(S,\sM_{S}^{min})$ such that 
\begin{enumerate}
 \item For each point $s\in S$, we have a fixed isomorphism $\CharM_{S,\bar{s}}^{min}\cong \CharM(G_{\xi_{\bar{s}}})$.
 \item The induced map $\bar{\Phi}^{\flat}_{\bar{s}}:\CharM(G_{\xi_{\bar{s}}})\cong\CharM_{S,\bar{s}}^{min}\to\CharM_{S,\bar{s}}$ on the level of characteristic is identical to the canonical map $\phi_{\bar{s}}$ as (\ref{equ:CanMapUniv}).
 \item The log pre-stable curve $(C\to S,\sM_{S})$ is the pull-back of $(C\to S,\sM_{S}^{min})$ via $\Phi$.
\end{enumerate}
Then we have a unique log map $f_{min}:(C,\sM_{C}^{min})\to (X,\sM_{X})$, which fits in diagram \ref{diag:UniverMinLogMap}. Note that $(C\to S,\sM_{S}^{min},f_{min})$ forms a minimal log map over the scheme $S$.
\end{lem}
\begin{proof}
Since all the underlying maps are fixed, it is enough to construct a map of log structures $f_{min}^{\flat}:f^{*}(\sM_{X})\to\sM_{C}^{min}$, which fits in the following commutative diagram
\[
\xymatrix{
 &f^{*}(\sM_{X}) \ar[rd]^{f^{\flat}} \ar[ld]_{f^{\flat}_{min}} & \\
\sM_{C}^{min} \ar[rr]^{\Phi_{C}^{\flat}} && \sM_{C}. 
}
\]
Consider an arbitrary closed point $p\in C$, which lies in an irreducible component corresponding to the vertex $v\in V(G_{\xi_{\bar{s}}})$. Then locally at $p$, we have
\begin{equation}\label{equ:LocalDecomp}
f^{\flat}(\delta)=e_{v,0}+\log h,
\end{equation}
where $e_{v,0}\in \sM_{S}$ near $\bar{s}$, and $h$ is a non-zero regular section locally near $p$. Note that there are two possible cases: if $p$ is a smooth non-marked point, then $h$ is a locally invertible section; if $p$ is a special point with contact order $c$, then $h=u\cdot \sigma^{c}$, where $u$ is a locally invertible section, and $\sigma$ is a local coordinate function vanishing at $p$. Note that the underlying map $\underline{\Phi}_{C}$ is an identity. Thus, to define $f^{\flat}_{min}(\delta)$ locally at $p$, it is enough to find a lifting $\tilde{e}_{v}\in \sM_{S}^{min}$ of $e_{v,0}$, such that the image of $\tilde{e}_{v}$ in $\CharM_{S}^{min}$ is the element associated to the vertex $v$.

We first consider the uniqueness. Assume that we have two lifting $\tilde{e}_{v}$ and $\tilde{e}_{v}'$, such that their images in $\CharM_{S}^{min}$ are given by the element associated to $v$. Then, we have $\tilde{e}_{v}=\log u + \tilde{e}_{v}'$ for some locally invertible function $u$. This implies that 
\[\Phi^{\flat}_{C}(\tilde{e}_{v})=\Phi^{\flat}_{C}(\log u) + \Phi_{C}^{\flat}(\tilde{e}_{v}').\]
Since $\tilde{e}_{v}$ and $\tilde{e}_{v}'$ are two lifting of $e_{v,0}$, we have $\Phi^{\flat}_{C}(\log u)=1$. Note that the underlying map $\underline{\Phi}_{C}=id_{C}$. It follows that $u=1$. This shows that the lifting is unique.

Now we consider the existence of the lifting. Denote by $\bar{e}_{v,0}$ the image of $e_{v,0}$ in the characteristic $\CharM_{S,\bar{s}}$. By (2) the map of monoids $\bar{\Phi}_{\bar{s}}^{\flat}$ is identical to $\phi_{\bar{s}}$. Then we have a unique element $\bar{e}_{v}\in \CharM_{S,\bar{s}}^{min}$, which corresponds to the element associated to $v$ in the graph $G$, and $\bar{\Phi}_{\bar{s}}^{\flat}(\bar{e}_{v})=\bar{e}_{v,0}$. Thus, locally we can lift $\bar{e}_{v}$ to an element $\tilde{e}_{v}\in \sM_{S}^{min}$ such that $\Phi_{\bar{s}}^{\flat}(\tilde{e}_{v})=e_{v}$. Then we define
\begin{equation}\label{equ:MinLocalConstr}
f_{min}^{\flat}(\delta)= \tilde{e}_{v}+\log h.
\end{equation}

The uniqueness of the lifting shows that the construction in (\ref{equ:MinLocalConstr}) can be glued globally to obtain a unique map $f^{\flat}_{min}$. We can check locally that the map of monoids $f^{\flat}_{min}$ is compatible with the structure morphisms of the corresponding log structures. This finishes the proof of the statement.
\end{proof}

In fact, in the above proof we constructed a log map $f_{min}$, which is minimal at $\bar{s}$, hence minimal in a neighborhood of $\bar{s}$, by the openness of minimality. We next construct a unique log prestable curve $(C\to S, \sM_{S}^{min})$ satisfying the three conditions in the above lemma. Note that the question is local on $S$. Pick a point $\bar{s}\in S$. Shrinking $S$, we can assume that there is a global chart $\beta:\CharM_{S,\bar{s}}\to\sM_{S}$. We have the canonical map $\phi_{\bar{s}}:\CharM(G_{\xi_{\bar{s}}})\to \CharM_{S,\bar{s}}$. Consider the pre-log structure given by the following composition:
\[\CharM(G_{\xi_{\bar{s}}})\stackrel{\phi_{\bar{s}}}{\longrightarrow}\CharM_{S,\bar{s}}\stackrel{\beta}{\longrightarrow}\sM_{S}\stackrel{\exp}{\longrightarrow}\sO_{S}.\]
Denote by $\sM^{min}_{S}$ the log structure associated to the above pre-log structure. Thus, the construction above gives a global chart $\beta_{min}:\CharM(G_{\xi_{\bar{s}}})\to \sM_{S}^{min}$ and a natural map $\Phi^{\flat}:\sM_{S}^{min}\to \sM_{S}$.

Note that the construction of $\sM_{S}^{min}$ depends on the choice of the chart $\beta$. Assume that we have another log structure $\sM_{1}^{min}$  and a map $\Phi_{1}^{\flat}:\sM_{1}^{min}\to \sM_{S}$ over $S$, which comes from another chart $\beta_{1}:\CharM_{S,\bar{s}}\to \sM_{S}$. Then we have:

\begin{lem}\label{lem:MinCharUniq}
There is a unique isomorphism of log structures $\sM_{1}^{min}\to \sM_{S}^{min}$ fitting in the following commutative diagram:
\[
\xymatrix{
\sM_{1}^{min} \ar[rr] \ar[rd]_{\Phi_{1}^{\flat}} && \sM_{S}^{min} \ar[ld]^{\Phi^{\flat}} \\
 &\sM_{S} &
}
\]
\end{lem}
\begin{proof}
Consider an irreducible element $a\in \CharM(G)$. Then the construction of $\sM_{1}^{min}$ and $\sM_{S}^{min}$ implies that 
\[\Phi_{1}^{\flat}\circ \beta_{1}(a) + \log u=\Phi^{\flat}\circ\beta_{min}(a),\]
where $u$ is a unique invertible section. We define
\[\beta_{1}(a)\mapsto \beta_{min}(a)+\log u^{-1}.\]
This induces a unique map $\sM_{1}^{min}\to \sM_{S}^{min}$, which satisfies the condition of the lemma. 
\end{proof}

\begin{lem}\label{lem:CompCurveLog}
There exists a unique dashed arrow which fits in the following commutative diagram:
\begin{equation}\label{diag:CurveArrow}
\xymatrix{
\sM_{S}^{min} \ar[rr]^{\Phi^{\flat}} && \sM_{S} \\
 & \sM_{S}^{C/S} \ar@{-->}[lu]^{\phi_{min}} \ar[ru]_{\phi} &.
}
\end{equation}
where $\phi$ is the structure arrow defining the log pre-stable curve $(C\to S,\sM_{S})$, see Definition \ref{DefLogC2}.
\end{lem}
\begin{proof}
Further shrinking $S$, we can choose a global chart $\N^{m}\to \sM_{S}^{C/S}$. Let $e$ be a generator of $\N^{m}$ which corresponds to an edge $l\in V(G_{\xi_{\bar{s}}})$. For convenience, we will identify $e$ with its image in $\sM_{S}^{C/S}$. Consider $\phi(e)\in \sM_{S}$, and its image $\bar{\phi}(e)\in \CharM_{S}$. Now on the level of characteristic, there is a unique element $\bar{e}'\in \CharM_{S}^{min}$, which corresponds to the element associated to $l$, such that $\bar{\Phi}^{\flat}(\bar{e}')=\bar{\phi}(e)$. A similar argument as in the proof of Lemma \ref{lem:UniqMinMap} shows that there is a unique section $e'\in \sM_{S}^{min}$ such that $\Phi^{\flat}(e')=\phi(e)$. Then we can define $\phi_{min}(e)=e'$ for every generator $e$. This gives the map $\phi_{min}:\sM_{S}^{min}\to \sM_{S}$. 

Note that our construction depends on a fixed chart $\N^{m}\to \sM_{S}^{C/S}$. However, a similar argument as in the proof of Lemma \ref{lem:MinCharUniq} shows that different choice of the global chart will induces the same map $\phi_{min}$. This finishes the proof. 
\end{proof}

Note that the arrow $\phi_{min}$ induces a log pre-stable curve $(C\to S,\sM_{S}^{min})$. Denote by $\sM_{C}^{min}$ the corresponding log structure on $C$ associated to the log curve. By Lemma \ref{lem:NodeContactOpen}, we can further shrink $S$, and assume that the contact order of the nodes on each geometric fiber is given by the weights of the edges of $G_{\xi_{\bar{s}}}$. Now we have: 
\begin{lem}\label{lem:Generalization}
The log pre-stable curve $(C\to S,\sM_{S}^{min})$ and the log map $(S,\sM_{S})^{min}\to(S,\sM_{S})$ induced by $\Phi^{\flat}$ satisfy the condition (1), (2) and (3) in Lemma \ref{lem:UniqMinMap}.
\end{lem}
\begin{proof}
Note that (3) follows from the commutativity of (\ref{diag:CurveArrow}) and Definition \ref{DefLogC2}. For (1) and (2), we can repeat the argument in Lemma \ref{lem:UniqMinMap}. Indeed, the same construction there yields a log map $(C,\sM_{C}^{min})\to (X,\sM_{X})$, which is minimal at $\bar{s}$. Now the openness of minimality implies that all points in a neighborhood of $\bar{s}$ are minimal. Therefore, the property (1) and (2) in Lemma \ref{lem:UniqMinMap} follows. 
\end{proof}

\subsection{Proof of Theorem \ref{thm:Main2}}\label{ss:ChangeBaseCat}
By the definition of log stack in Section \ref{s:LogStack}, the stack $\sK_{\Gamma}(X^{log})$ carries a natural log structure $\sM_{\sK_{\Gamma}(X^{log})}$ as follows. For any $g: S\to \sK_{\Gamma}(X^{log})$, the log structure $g^{*}\sM_{\sK_{\Gamma}(X^{log})}$ is exactly the log structure on $S$ given by the minimal log map $\xi$ over $S$ induced by $g$. Now we have a universal diagram of log stacks:
\[
\xymatrix{
(\fC_{\Gamma},\sM_{\fC_{\Gamma}}) \ar[r] \ar[d] & X^{log} \\
(\sK_{\Gamma}(X^{log}),\sM_{\sK_{\Gamma}(X^{log})})
}
\]
where the pair $(\fC_{\Gamma},\sM_{\fC_{\Gamma}})$ is the universal curve over $\sK_{\Gamma}(X^{log})$ with universal log structures $\sM_{\fC_{\Gamma}}$. This diagram gives a stable log map $\xi_{\sK_{\Gamma}(X^{log})}$ over $(\sK_{\Gamma}(X^{log}),\sM_{\sK_{\Gamma}(X^{log})})$.

Now consider a stable log map $\xi^{log}$ over $(S,\sM_{S})$. Then the tuple $(\xi^{log},S,\sM_{S})$ gives a stable log map over $S$. The universal property of minimal log map implies that there is a unique minimal log map $\xi_{min}=(\xi^{log}_{min},S,\sM_{S}^{min})$ over $S$, and a map of log schemes $g:(S,\sM_{S})\to(S,\sM_{S}^{min})$, such that $\xi^{log}=g^{*}\xi^{log}_{min}$ as in Definition \ref{defn:LogMapOverLog}. This induces a unique log map $f: (S,\sM_{S})\to (\sK_{\Gamma}(X^{log}),\sM_{\sK_{\Gamma}(X^{log})})$ such that $\xi^{log}=f^{*}\xi_{\sK_{\Gamma}(X^{log})}$. 

Theorem \ref{thm:Main2} follows. 

\begin{rem}\label{rem:ModuliOverTor}
Using the same argument as above, we can shows that the two stacks $\sK_{n,g}(X^{log},\beta)$ and $\sK_{g,n}^{pre}(X^{log})$ with their universal minimal log structures can be viewed as categories fibered over $\LogSch^{fs}$, parameterizing log maps over fs log schemes with corresponding numerical conditions.
\end{rem}

\begin{rem}
If the log structure $\sM_{X}$ on the target $X$ is trivial, the stack $\sK_{g,n}^{}(X^{log},\beta)$ is isomorphic to the stack $\sK_{g,n}(X,\beta)$ of usual stable maps with the minimal log structure coming from the canonical log structure of its universal curve.
\end{rem}

\section{The boundedness theorem for minimal stable log maps}\label{sec:Boundedness}

\subsection{Statement of the boundedness theorem}
In this section, we fix the target $X^{log}=(X,\sM_{X})$ as in Convention \ref{con:LogTarget}. The main result of this section is the following:

\begin{thm}\label{Boundedness}
There exists a scheme $T$ of finite type, and a map $g:T\to \sK_{\Gamma}(X^{log})$, which exhausts all geometric point of $\sK_{\Gamma}(X^{log})$. Namely, for any point $\xi\in \sK_{\Gamma}(X^{log})(\C)$, there exists a lifting $\Spec\C\to T$, such that its composition with $g$ gives $\xi$.
\end{thm}
\begin{proof}
The proof of this theorem will occupy the whole section. Indeed, we will prove that the map $\sK_{\Gamma}(X^{log})\to \sK_{g,n}(X,\beta)$, obtained by removing all log structures, is of finite type. In Section \ref{ss:FinDisData}, we will bound the choices of marked graph by stratifying $\sK_{g,n}(X,\beta)$. In Section \ref{ss:PfBd}, we will construct a family of minimal stable log maps, which exhausts all possible minimal log structures with fixed underlying map and marked graph. This will be achieved by considering isomorphisms of corresponding line bundles. The result from Section \ref{ss:RemoveGenericDeg} will be used in the above argument.
\end{proof}

\subsection{Isomorphisms of line bundles induced by stable log maps}\label{ss:RemoveGenericDeg}
Consider a stable log map (not necessarily minimal) $\xi=(C\to S,\sM_{S},f)$ over a scheme $S$. In this subsection, we put the following assumption 
\begin{equation}\label{assum:GenericDeg}
 \mbox{The characteristic $\CharM_{S}$ is a constant sheaf of monoids on $S$}.
\end{equation}

\begin{lem}\label{lem:ConstantDegGraph}
With the assumptions as above, the marked graphs of all geometric fibers of $\xi$ are isomorphic.
\end{lem}
\begin{proof}
Note that the elements smoothing the distinguished nodes are in $\CharM_{S}$. Then the statement follows from the assumption (\ref{assum:GenericDeg}). 
\end{proof}

Given the stable log map $\xi$ over $S$ as above, let us consider the following commutative diagram:
\[
\xymatrix{
 & f^{*}(\sM_{X}) \ar[d]^{p_{1}} \ar[rr]^{f^{\flat}} && \sM_{C} \ar[d]^{p_{2}} \\
\N \ar[r] & f^{*}(\CharM_{X}) \ar[rr]^{\bar{f}^{\flat}} && \CharM_{C}. 
}
\]
The composition of the bottom arrow $\N\to \CharM_{C}$ locally lifts to a chart of a sub-log structure of $\sM_{C}$. Denote by $\sM$ the resulting sub-log structure. Note that this is also a DF-log structure. The map $f^{\flat}$ induces an isomorphism of log structures $f^{*}(\sM_{X})\to \sM$. By the argument in Section \ref{ss:DFLog}, this gives an isomorphism of the corresponding line bundles and sections. Next, we will describe this isomorphism on each irreducible component of $C$.

Pick a point $\bar{s}\in S$. Shrinking $S$, we may choose a lifting of global chart $\beta:\CharM_{S,\bar{s}}\to \sM_{S}$. Consider the induced map $\hat{\beta}:\CharM_{S}\to \sM_{C}$. Denote by $\widehat{\sM}_{C}=\sM_{C}^{gp}/(\CharM_{S})^{gp}$ the quotient given by the map $\hat{\beta}$. Consider the following commutative diagram:
\begin{equation}\label{diag:QuotDeg}
\xymatrix{
0 \ar[r] & (\CharM_{S})^{gp} \ar[r]^{\hat{\beta}^{gp}} & \sM_{C}^{gp} \ar[r] & \widehat{\sM}_{C}\\
 & & f^{*}(\sM_{X}) \ar[u] \ar[ur]_{\hat{f}^{\flat}}& 
}
\end{equation}
where the map $\hat{f}^{\flat}$ is given by the composition $f^{*}(\sM_{X})\to \sM_{C} \to \widehat{\sM}_{C}$.

\begin{rem}\label{rem:DepOnChart}
Note that the morphism $\hat{f}^{\flat}$ depends on the choice of a lifting $\beta:\CharM_{S,\bar{s}}\to \sM_{S}$. This will be important when we discuss the valuative criterion.
\end{rem}

\begin{con}\label{con:LabDistPoint}
Consider the irreducible component $C_{v}$ of $C$ corresponding to a vertex $v\in G_{\xi}$. Note that $C_{v}$ is connected. Denote by $\{p_{l}\}_{l\in \Lambda^{low}_{v}}$ the set of splitting nodes, joining $v$ with $v'$ for some $v'\leq_{}v$. Let $\{p_{l}\}_{l\in \Lambda^{up}_{v}}$ be the set consisting of the following special points in $C_{v}$:
\begin{enumerate}
 \item the set of splitting nodes, joining $v$ with $v''$ for some $v\leq v''$;
 \item the marked points with non-trivial contact orders.
\end{enumerate}
Denote by $c_{l}$ the contact order at $p_{l}$ for $l\in \Lambda^{low}_{v}\cup \Lambda^{up}_{v}$. Consider the line bundle
\[L_{v}=\prod_{l\in \Lambda^{low}_{v}}\sO_{C_{v}}(c_{l}\cdot p_{l})\otimes\prod_{l\in \Lambda^{up}_{v}}\sO_{C_{v}}(-c_{l}\cdot p_{l}).\]
Note that the line bundle $L_{v}$ only depends on the graph $G_{\xi}$.
\end{con}

\begin{prop}\label{prop:LineBdIso}
Assume that the element associated to a vertex $v\in G_{\xi}$ is not zero. Then the map $\hat{f}^{\flat}$ induces a natural isomorphism of line bundles 
\[\hat{f}^{\flat}_{v}:f^{*}(L)_{v}\to L_{v}.\]
\end{prop}

\begin{proof}
We first construct $\hat{f}^{\flat}_{v}$ locally. There are three cases.

\textbf{Case 1:} Consider a closed point $p$ of $p_{l}$ for $l\in \Lambda^{up}_{v}$. Locally at $p$ we have 
\[f^{\flat}(\delta)= e_{v} + c_{l} \log \sigma_{l},\]
where $\sigma_{l}$ is the local coordinate of $p$ in $C_{v}$ defining the marking $p_{l}$, and $e_{v}$ is contained in the image of $\hat{\beta}$. Thus, we have $\hat{f}^{\flat}(\delta)= c_{l}\log \sigma_{l}.$ Then locally near $p$ we define 
\begin{equation}\label{equ:SplitUpNode}
\hat{f}^{\flat}_{v}(\delta)= \sigma_{l}^{c_{l}},
\end{equation}
Note that $\sigma_{l}^{c_{l}}$ is the local section of $L_{v}$ at $p$.

\textbf{Case 2:} Consider a closed point $p$ of the splitting node $p_{l}$ for $l\in \Lambda^{low}_{v}$. Assume that $p_{l}$ joining vertices $v'$ and $v$ such that $v'\leq_{}v$. Locally at $p$ we have 
\begin{equation}\label{equ:LowNodeLog}
f^{\flat}(\delta)= e_{v'} + c_{l} \log \sigma_{l}',
\end{equation}
where $e_{v'}$ is in the image of $\hat{\beta}$. By a nice choice of coordinates we have
\begin{equation}\label{equ:LowNodeCur}
c_{l}\cdot e_{l}= c_{l}\log \sigma_{l} + c_{l}\log \sigma_{l}' \mbox{,\ \ \ in } \sM_{C}
\end{equation}
where $\sigma_{l}'$ is the local coordinate of $p_{l}$ in $C_{v}'$, and $e_{l}$ is the element smoothing node, and contained in the image of $\hat{\beta}$. Then we have
\[1= c_{l}\log \sigma_{l} + c_{l}\log \sigma_{l}' \mbox{,\ \ \ in } \hat{\sM}_{C}.\]
This induces
\[\hat{f}^{\flat}(\delta)= c_{l} \log \sigma_{l} = 1 - c_{l}\log \sigma_{l}.\]
Then locally at the node $p$ we define
\begin{equation}\label{equ:SplitLowNode}
\hat{f}^{\flat}_{v}(\delta)= (\frac{1}{\sigma_{l}})^{c_{l}}.
\end{equation}
Note that this is a local generator of $L_{v}$ at $p$.

\textbf{Case 3:} Locally at a point $p$ which is not contained in one of the $p_{l}$ for $l\in \Lambda^{up}_{v}\cup\Lambda^{low}_{v}$, we have
\[f^{\flat}(\delta)= e_{v}+ \log h,\]
where $h$ is an invertible function at $p$ and $e_{v}$ is contained in the image of $\hat{\beta}$. Then the map $\hat{f}^{\flat}(\delta)= \log h$ induces 
\begin{equation}\label{equ:SplitGenPt}
\hat{f}^{\flat}_{v}(\delta_{\lambda})=h.
\end{equation}

Note that the local construction of $\hat{f}^{\flat}_{v}$ is uniquely determined by $\hat{f}^{\flat}$, which is a map of sheaves of monoids. Thus these local definitions can be glued to obtain a global map. We also notice that $\delta$ lifts to a the local generator of $L_{}$. Therefore we construct an isomorphism of line bundles $\hat{f}^{\flat}_{v}$ as required. 
\end{proof}

\begin{rem}
The local calculation shows that the isomorphism $\hat{f}^{\flat}_{v}$ in Proposition \ref{prop:LineBdIso} depends on the choice of the chart $\beta$. 
\end{rem}

\subsection{Finiteness of the discrete data}\label{ss:FinDisData}
\begin{prop}\label{prop:FinDisData}
The following set is finite:
\[\{G\ | \ G \mbox{\ is the marked graph of some\ } \xi\in \sK_{\Gamma}(X^{log})(\C) \}.\]
\end{prop}
\begin{proof}
\textbf{Step 1: Bounding the choices of underlying dual graph.} 
Denote by $\sK_{g,n}(X,\beta)$ the Kontsevich moduli space of stable maps, with $n$-marked points, genus $g$, and curve class $\beta$ in $X$. Note that we have a morphism 
\[\sK_{\Gamma}(X^{log})\to \sK_{g,n}(X,\beta),\] 
by removing all log structures. Let $U\to \sK_{g,n}(X,\beta)$ be an affine \'etale chart. Consider the following cartesian diagram without log structures:
\[
\xymatrix{
\sK_{U} \ar[r] \ar[d] & \sK_{\Gamma}(X^{log}) \ar[d] \\
U \ar[r] & \sK_{g,n}(X,\beta).
}
\]
Since the stack $\sK_{g,n}(X,\beta)$ is of finite type, it is enough to prove that the set of dual graphs corresponding to the geometric point of $\sK_{U}$ is finite. Denote by $C_{U}\to U$ the universal curve and $\underline{f}_{U}: C_{U}\to X$ the universal map over $U$.

Since $U$ is of finite type, it is covered by finitely many strata, where the family of curves over each stratum have a fixed dual graph. We put the reduced scheme structure on each stratum. Let $S$ be the stratum corresponding to a graph $G$. Denote by $\underline{f}: C\to X$ the universal map over $S$.

\textbf{Step 2: Bounding the choices of distinguished nodes and orientations.} 
Since $G$ is a finite graph, the number of choices of distinguished nodes is finite. We first fix a choice of distinguished nodes on $G$. So we fix an orientation on $G$ such that:
\begin{enumerate}
 \item If $C_{v}$ does not degenerate to $D$, then $v\in V_{n}(G)$ 
 \item The non-oriented edges are in one-to-one correspondence to the non-distinguished nodes.
 \item No cycles contain distinguished edges.
\end{enumerate}

\textbf{Step 3: Bounding the choices of contact orders.}
Since we fixed the orientation and distinguished edges on $G$, we can use the notations $\{p_{l}\}_{l\in \Lambda_{v}^{low}}$ and $\{p_{l}\}_{l\in \Lambda_{v}^{up}}$ for the two sets of distinguished points on the subcurve $C_{v}$ as in Conventions \ref{con:LabDistPoint}. Denote by $c_{l}$ the possible contact order at the distinguished point $p_{l}$. Since the dual graph of the underlying curve is fixed, the multi-degree of $\underline{f}^{*}(L)$ on $C_{v}$ is fixed for any $v\in V(G)$. By Proposition \ref{prop:LineBdIso}, we have 
\begin{equation}\label{equ:CompDeg}
\deg \underline{f}^{*}(L)|_{C_{v}} = \sum_{l\in \Lambda_{v}^{low}}c_{l}-\sum_{l'\in \Lambda_{v}^{up}}c_{l'}.
\end{equation}
First, consider a maximal vertex $v\in V(G)$. Then the set $\{p_{l}\}_{l\in \Lambda_{v}^{up}}$ is given by the discrete data $\Gamma$. Since the contact orders are all positive, the choices of $c_{l}$ for $l\in \Lambda_{v}^{low}$ is finite by (\ref{equ:CompDeg}).

Consider an arbitrary vertex $v'\in V(G)$. We assume that for any adjacent vertex $v$ of $v'$ such that $v'\leq_{}v_{}$, the number of choices of the contact orders along the splitting nodes joining $v'$ and $v$ is finite. Then by taking into account all contact orders from adjacent vertices and those from marked points of $C$, a similar argument shows that the possible choices of contact order $c_{l}$ for $l\in \Lambda_{v'}^{low}$ are also finite in number. Since $G^{}$ is a finite graph, this proves that the choice of contact orders on $G$ is finite.

This finishes the proof of the proposition. 
\end{proof}

\subsection{Proof of Theorem \ref{Boundedness}}\label{ss:PfBd}
Consider the family of usual stable maps $\underline{f}:C\to X$ over $S$ as in Step 1 of the above proof. Fix a possible marked graph $G_{0}$ with $\underline{G}_{0}=G$ the dual graph of $C$. We use the notations as in Step 3 of the above proof, and assume that (\ref{equ:CompDeg}) holds for any $v\in V(G_{0})$. Since the stack $\sK_{g,n}(X,\beta)$ is of finite type, to prove Theorem \ref{Boundedness}, it is enough to prove the following:

\begin{prop}\label{prop:ExhaustStrata}
Notations and assumptions as above, there exists a scheme $T$ of finite type over $S$, and a family of minimal stable log maps $\xi$ over $T$, which satisfies the following conditions: for any minimal stable log map $\xi'$ over $\bar{s}$, with marked graph given by $G_{0}$, and underlying map $\underline{\xi}'$ given by the pull-back of $\underline{f}$ via $\bar{s}\to S$, there exists a lifting $\bar{s}\to T$, such that $\xi'$ is isomorphic to the pull-back $\xi_{\bar{s}}$.
\end{prop}
\begin{proof}
By shrinking $S$, we can assume that $S$ is affine, and the canonical log structure $\sM_{S}^{C/S}$ on $S$ coming from the family $C\to S$ has a global chart $\N^{n}\to\CharM_{S,\bar{s}}^{C/S}$ for some geometric point $\bar{s}\in S$. Consider the pre-log structure $\CharM(G_{0})\to \sO_{S}$, given by $e\mapsto 0$ for any non-trivial element $e\in \CharM(G_{0})$. Denote by $\sM_{S}$ the new log structure associated to the pre-log structure. Note that there is a map $\N^{n}\to \CharM(G_{0})$ given by the corresponding nodes. This induces a map $\sM_{S}^{C/S}\to \sM_{S}$, hence a log pre-stable curve $\zeta=(C\to S,\sM_{S})$ over $S$. Note that any minimal log map $\xi'$ over $\bar{s}\in S$ as in the statement has the source log curve isomorphic to $\zeta_{\bar{s}}$. 

Denote by $\sM_{C}$ the log structure on $C$ corresponding to the log pre-stable curve $\zeta$. Note that over $C$ we have another log structure $\underline{f}^{*}(\sM_{X})$. Since the dual graph $G_{0}$ is fixed, we have a morphism of sheaves of monoids on $C$:
\[\bar{f}^{\flat}: \underline{f}^{*}(\CharM_{X})\to \CharM_{C},\]
which is locally described as in Section \ref{ss:LogMapChar}. To define a log map $f:(C,\sM_{C})\to X^{log}$, it is enough to define a map of log structures $f^{\flat}:\underline{f}^{*}(\sM_{X})\to \sM_{C}$ fitting in the following commutative diagram:
\begin{equation}\label{diag:LiftCharMap}
\xymatrix{
 &&&& \sO_{C}\\
 & \underline{f}^{*}(\sM_{X}) \ar[d]^{p_{1}} \ar@/^/[rrru] \ar@{-->}[rr]^{f^{\flat}} && \sM_{C} \ar[d]^{p_{2}} \ar@/_/[ru]& \\
\N \ar[r] & \underline{f}^{*}(\CharM_{X}) \ar[rr]^{\bar{f}^{\flat}} && \CharM_{C}&,
}
\end{equation}
where the two vertical arrows are the canonical projection, and the arrow $\N \to \underline{f}^{*}(\CharM_{X})$ is the pull-back of the global presentation. Note that the arrow $\bar{f}^{\flat}$ is injective. Denote by $\delta_{X}$ and $\delta_{C}$ the image of $\delta$ in $\underline{f}^{*}(\CharM_{X})$ and $\CharM_{C}$ respectively. The inverse images $p_{1}^{-1}(\delta_{X})$ and $p_{2}^{-1}(\delta_{C})$ form two $\sO_{C}^{*}$-torsors. Denote by $\sI som_{C}(p_{1}^{-1}(\delta_{X}), p_{2}^{-1}(\delta_{C}))$ the presheaf of isomorphisms of the two torsors over $C$. To find a dashed arrow as in (\ref{diag:LiftCharMap}) is equivalent to have a global section of $\sI som_{C}(p_{1}^{-1}(\delta_{X}), p_{2}^{-1}(\delta_{C}))$. Note that the torsor $p_{1}^{-1}(\delta_{X})$ corresponds to the line bundle $\underline{f}^{*}L$. Denote by $L_{C}$ the corresponding line bundle of $p_{2}^{-1}(\delta_{C})$. Then we have
\[\sI som_{C}(p_{1}^{-1}(\delta_{X}), p_{2}^{-1}(\delta_{C}))\cong \sI som_{C}(\underline{f}^{*}L, L_{C})\cong \sI som_{C}(\underline{f}^{*}L\otimes L_{C}^{-1},\sO_{C}).\]
Denote by $I$ the above presheaf. It is well-known that line bundles are parametrized by the algebraic stack $\sB\G_{m}$. Thus, $I$ represented by a $\G_{m}$-torsor, which is a separated algebraic space of finite type. Let $\pi:C\to S$ be the projection. By \cite[Theorem 1.5]{HomStack}, there is an algebraic space $\pi_{*}I$ locally of finite type over $S$, which for any $Y\to S$ associates the groupoid of isomorphisms $(\underline{f}^{*}L\otimes L_{C}^{-1})_{Y}\to\sO_{C_{Y}}$. We have the following lemma for the boundedness of $\pi_{*}I$.
\end{proof}

\begin{lem}
The algebraic space $\pi_{*}I$ is of finite type over $S$. 
\end{lem}
\begin{proof}
By our assumption on $G_{0}$, the two line bundles $L_{C}$ and $f^{*}L$ have the same degree when restrict to each irreducible component over $\bar{s}\in S$. Since $S$ is affine, by \cite[Proposition 1]{FP}, there is a unique closed subscheme $T\subset S$ which represents the condition that the line bundle $f^{*}L\otimes L_{C_{T}}^{-1}$ is trivial. The same proof shows that over that locus the line bundle is pulled back from the base. Its sheaf of trivializations is again represented by a $\G_{m}$-torsor $U \to \pi_*I$, which is of finite type. 
\end{proof}

By pulling back via $\pi_{*}I\to S$, we have a family of log pre-stable curves $\zeta_{\pi_{*}I}=(C_{I}\to \pi_{*}I,\sM_{\pi_{*}I})$, a usual stable map $\underline{f}_{\pi_{*}I}:C_{I}\to X$, and a morphism of sheaves of monoids $f^{\flat}_{\pi_{*}I}:\underline{f}_{\pi_{*}I}^{*}\sM_{X}\to \sM_{C_{I}}$, where $\sM_{C_{I}}$ is the log structure on $C_{I}$ given by the log curve $\zeta_{\pi_{*}I}$. 

\begin{lem}
The set of points $t\in \pi_{*}I$, whose fiber $f^{\flat}_{\pi_{*}I,\bar{t}}$ gives a morphism of log structures, forms a closed subset of $\pi_{*}I$.
\end{lem}
\begin{proof}
The condition that $f^{\flat}_{\pi_{*}I}$ is a morphism of log structures is equivalent to having the following commutative diagram:
\begin{equation}\label{diag:BdLogMap}
\xymatrix{
\underline{f}_{\pi_{*}I}^{*}\sM_{X} \ar[rr]^{f^{\flat}_{\pi_{*}I}} \ar[dr]_{\exp_{X}} && \sM_{C_{I}} \ar[ld]^{\exp_{C}}\\
 & \sO_{C_{I}} &
},
\end{equation}
where the two arrows $\exp_{X}$ and $\exp_{C}$ are the structure maps of the corresponding log structures $\underline{f}_{\pi_{*}I}^{*}\sM_{X}$ and $\sM_{C_{I}}$. Locally on $C_{I}$, we choose a generator $\delta\in \underline{f}_{\pi_{*}I}^{*}\sM_{X}$, then the commutativity of the diagram is equivalent to the following equality of sections of $\sO_{C_{I}}$:
\[\exp_{X}(\delta)=\exp_{C}\circ f^{\flat}_{\pi_{*}I}(\delta),\]
which is a closed condition. Let $V\subset C_{I}$ be the closed sub-scheme representing the commutativity of (\ref{diag:BdLogMap}) over $C_{I}$, and $V^{c}$ the complement of $V$ in $C_{I}$. Denote by $W$ the image of $V^{c}$ in $\pi_{*}I$ via the projection $C_{I}\to \pi_{*}I$. Since the family of curves is flat, the image $W$ is open in $\pi_{*}I$. Thus, the complement $W^{c}$ of $W$ is closed in $\pi_{*}I$. This proves the lemma. 
\end{proof}

We take $T=W^{c}$ as in the above proof with the reduced scheme structure. Then $T$ is a closed subspace of $\pi_{*}I$. By pulling back families from $\pi_{*}I$, we obtain a family of minimal stable log maps $\xi$ over $T$. According to our construction, this family $\xi$ over $T$ satisfies the lifting property in Proposition \ref{prop:ExhaustStrata}.

Theorem \ref{Boundedness} follows from the above arguments. 

\section{The weak valuative criterion for minimal stable log maps}\label{sec:Valuative}

\subsection{Statement of the weak valuative criterion}
Let $R$ be a discrete valuation ring, and $K$ the fraction field of $R$. Denote by $\pi$ the uniformizer of $R$, and write $S=\Spec R$. Let $s$ and $\eta$ be the closed and generic point of $S$ respectively. If $R'$ is another discrete valuation ring, we will write $\pi'$ for its uniformizer. Denote by $s'$ and $\eta'$ the closed and generic point of $S'=\Spec R'$ respectively. 

\begin{thm}\label{thm:Val}
With the notations above, consider a minimal stable log map $\xi_{\eta}$ over $\eta$. Possibly after an base change given by an injection $R \hookrightarrow R'$ of DVR, which induce a finite extension of fraction fields, we have an extension of minimal stable log maps given by the following cartesian diagram:
\[
 \xymatrix{
 \xi_{\eta'} \ar[r] \ar[d] & \xi_{S'} \ar[d] \\
 \eta' \ar[r] & S',
 }
\]
where $\xi_{\eta'}$ is the pull-back of $\xi_{\eta}$ via $\eta'\rightarrow \eta$, and $\xi_{S'}$ is a minimal stable log map over $S'$. Furthermore, the extension $\xi_{S'}$ is unique up to a unique isomorphism and its formation commutes with further injections of discrete valuation rings.
\end{thm}
\begin{proof}
We first assume that $\xi_{\eta}$ is a minimal log map over $\eta$, which is not necessarily stable. Possibly after base change, we fix an extension of the underlying pre-stable map $\underline{f}:C\to S$, such that its restriction to the generic fiber is given by the pull-back of $\underline{\xi}_{\eta}$. Denote by $\underline{\xi}$ the extended underlying map. Here for simplicity, we still use $S$ to denote the new base. The existence of compatible minimal log structures on $\underline{\xi}$ will be proved in Section \ref{ss:ExistExt}. This will be achieved by constructing an extension of certain simplified log maps, and using the universal property of minimal log maps. The uniqueness of the extended minimal log structure on $\underline{\xi}$ will be proved in Section \ref{ss:UniqueLimit}. 

In case of stable maps, the extended underlying map $\underline{\xi}$ is unique up to a unique isomorphism. Hence the theorem will be proved by the above argument. 
\end{proof}

\begin{rem}
By Observation \ref{obs:CompMapStack}, there is a map $\sK^{pre}_{g,n}(X^{log},\beta)\to \sK_{g,n}^{pre}(X,\beta)$, where $\sK_{g,n}^{pre}(X,\beta)$ is the stack of usual pre-stable maps, and $\sK^{pre}_{g,n}(X^{log},\beta)$ is as in Definition \ref{defn:min-pre-stab-map}. Then our proof implies that this map of stacks satisfies the weak valuative criterion.
\end{rem}

\subsection{Local analysis of the extended underlying map}
Let $\xi_{\eta}=(C_{\eta}\to \eta,\sM_{\eta},f_{\eta})$. We first consider the case $\xi_{\eta}$ is a log map (not necessarily minimal and stable). We still use $\underline{f}:C\to X$ to denote the extended underlying map over $S$. Possibly after a base change, we fix a chart $\beta_{\eta}:\CharM_{\eta}\to \sM_{\eta}$. Denote by $G_{\eta}$ the marked graph of $\xi_{\eta}$. If a node of $C$ is smoothed out over $\eta$, then we call it a \textit{special node}, otherwise we call it a \textit{generic node}. 

Consider a point $p\in C_{\bar{s}}$, and choose an \'etale neighborhood $U$ of $p$. Write $U_{\eta}:=U\times_{S}\eta$. Shrinking $U$, we have
\begin{equation}\label{equ:OverGenPoint}
f^{\flat}_{\eta}(\delta)= \beta_{\eta}(e_{v}) + \log u_{p}, \mbox{  over } U_{\eta}
\end{equation}
where $e_{v}\in \CharM_{\eta}$ is the degeneracy of some vertex $v\in G_{\eta}$, and $u_{p}\in \sO_{U_{\eta}}$. 

Assume that $p$ is not a generic node. Shrinking $U$ further, we can assume that $U$ is connected, and does not contain a generic node. We also assume that $U$ does not contain points on other components which are not containing $p$. Note that in this case $U$ is normal, and $u_{p}$ extends to a rational function on $U$. Denote by $\nu_{\pi}$ the evaluation of the divisor given by the uniformizer $\pi$. Let $n_{p}=\nu_{\pi}(u_{p})$, then we have the following result.

\begin{lem}\label{lem:LocalUnderlying}
With the above assumption, the point $p$ satisfies one of the following possibilities:
\begin{enumerate}
 \item If $p$ is a non-distinguished point, then there is a neighborhood of $p$, which contains only non-distinguished points over $\eta$, and we have $u_{p}=\pi^{n_{p}}\cdot h_{p}$, where $h_{p}\in \sO_{U}^{*}$.
 \item If $p$ is a marked point with contact order $c$ over $\eta$, then $u_{p}=\pi^{n_{p}}\cdot x^{c}\cdot h_{p}$, where $h_{p}\in \sO_{U}^{*}$, and the section containing $p$ is given by the vanishing of $x\in \sO_{U}$.
 \item If $p$ is a special node, then $u_{p}=\pi^{n_{p}}\cdot x^{c}\cdot h_{p}$, where $h_{p}\in \sO_{U}^{*}$, the section $x\in \sO_{U}$ is a local coordinate of one component at $p$, and $c$ is a non-negative integer.
\end{enumerate}
Note that if in (3) we have $c=0$, then this is compatible with the case described in (1).
\end{lem} 
\begin{proof}
Since $n_{p}=\nu_{\pi}(u_{p})$, and $u_{p}$ is well-defined over the generic fiber, we have $u_{p}=\pi^{n_{p}}\cdot h_{p}'$ for some $h_{p}'\in \sO_{U}$. Since $p$ is non-distinguished, there is a neighborhood of $p$, which contains only non-distinguished points over $\eta$. It follows that $h_{p}'\in \sO_{U_{\eta}}^{*}$. Since $\nu_{\pi}(h_{p}')=0$, then $h_{p}'\in \sO_{U}^{*}$. This proves (1).

For (2), we have $\nu_{x}(h_{p}')=c$, where $\nu_{x}$ is the evaluation map given by the divisor corresponding to the vanishing of $x$. Then we have $h_{p}'=x^{c}\cdot h_{p}$, such that the restriction $h_{p}|_{U_{\eta}}$ is invertible. The same argument as for (1) shows that $h_{p}\in \sO_{U}^{*}$.

Consider the case where $p$ is a special node. Let $x$ and $y$ be local coordinates of the two components meeting at $p$. Choosing the coordinates 
appropriately, we may assume that $x\cdot y=\pi^{n}$ for some positive integer $n$. Without loss of generality, we can assume that $\nu_{y}(h_{p}')=0$ and $\nu_{x}(h_{p}')=c$ for some non-negative integer $c$. Thus, as in (2), we have $h_{p}'=x^{c}\cdot h_{p}$ for some $h_{p}\in \sO_{U}^{*}$. This proves (3). 
\end{proof}

\begin{obs}\label{obs:SpecialDeg}
For a smooth point $p$, the integer $n_{p}$ and the rational section $u_{p}$ in (\ref{equ:OverGenPoint}) depend on the choice of the chart $\beta_{\eta}$. We call the integer $n_{p}$ \textit{the special degeneracy of $p$ with respect to the chart $\beta_{\eta}$}. Let $Z$ be the irreducible component of the fiber containing $p$. Then it is not hard to see that generic points on $Z$ also have $n_{p}$ as the special degeneracy under $\beta_{\eta}$. Thus, we call $n_{p}$ \textit{the special degeneracy of $Z$ under $\beta_{\eta}$}.
\end{obs}

\begin{rem}\label{rem:SpecialOrder}
Consider a node $p$ joining two irreducible components $Z_{1}$ and $Z_{2}$ over $\bar{s}$. First we assume that $p$ is a special node. Let $x$ and $y$ be local coordinates on $Z_{1}$ and $Z_{2}$ at $p$ respectively, such that $x\cdot y=\pi^{n}$. By Lemma \ref{lem:LocalUnderlying}(3), we can assume that $u_{p}=\pi^{n_{p}}\cdot x^{c}\cdot h_{p}$. Thus, we can check that the special degeneracy of $Z_{1}$ is $n_{p}$, and the special degeneracy of $Z_{2}$ is $n_{p}+c\cdot n$. In this case, we write $Z_{1}\leq Z_{2}$. Note that if $c=0$, we have both $Z_{1}\leq Z_{2}$, and $Z_{2}\leq Z_{1}$.

Consider the case where $p$ is a generic node. We take the normalization of $C$ along all the generic node. Then we obtain a set of usual pre-stable curves $\{C_{v}\}_{v\in V(G_{\eta})}$ over $S$. If $Z_{1}\subset C_{v_{1}}$ and $Z_{2}\subset C_{v_{2}}$, and $v_{1}\leq v_{2}$, then we define $Z_{1}\leq Z_{2}$. We thus define an orientation on the dual graph $G$ of the curve $C_{\bar{s}}$. 
\end{rem} 

The following result, which gives a way of comparing sections in the base log structure, is crucial in the proof of the uniqueness of the extension.

\begin{lem}\label{lem:SpecialDeg}
Using the notations as above, consider another chart $\beta_{\eta}':\CharM_{\eta}\to \sM_{\eta}$, and a generic point $p\in C_{\bar{s}}$ lies in the component corresponding to $v\in V(G_{\eta})$. The two special degeneracy of $p$ given by $\beta_{\eta}$ and $\beta_{\eta}'$ are the same, if and only if $\beta_{\eta}'(e_{v})=\log u + \beta_{\eta}(e_{v})$ for some $u\in R^{*}$.
\end{lem}
\begin{proof}
The `` if '' part is obvious. Consider the other direction. As in (\ref{equ:OverGenPoint}), locally at a non-marked point $p$ we have
\[f_{\eta}^{\flat}(\delta)=\beta_{\eta}'(e_{v})+\log u_{p}'=\beta_{\eta}(e_{v})+\log u_{p}.\]
Then the assumption implies that $u\cdot u_{p}'= u_{p}$, for some $u\in R^{*}$. Thus, we have
\[\beta_{\eta}'(e_{v})=\beta_{\eta}(e_{v})+\log u,\]
which proves the statement.
\end{proof}

\begin{lem}\label{lem:UnderTangency}
With the notations as above, the integer $c$ as in (2) and (3) of Lemma \ref{lem:LocalUnderlying} does not depend on the choice of chart $\beta_{\eta}$. Therefore the orientation on $G$ defined in Remark \ref{rem:SpecialOrder} does not depend on the choice of the chart $\beta_{\eta}$.
\end{lem}
\begin{proof}
Consider another chart $\beta_{\eta}:\CharM_{\eta}\to \sM_{\eta}$, and 
\[f^{\flat}(\delta)=\beta'_{\eta}(e_{v})+\log u_{p}',\]
for some $u_{p}'\in \sO_{U_{\eta}}^{*}$. Then we have 
\[\beta_{\eta}'(e_{v})=\beta_{\eta}(e_{v}) + \log a,\]
for some element $a\in K$. Comparing with (\ref{equ:OverGenPoint}), we have 
\[u_{p}= a\cdot u_{p}'.\]
Since $c$ is given by the evaluation $\nu_{x}$, this implies the statement of the lemma. 
\end{proof}

We take the normalization of $C$ along all generic nodes. For each $v\in V(G_{\eta})$, denote by $C_{v}$ the corresponding connected component. Now we consider the case where $p$ is a generic node. Let $l\in E(G)$ be the edge corresponding to the generic node $p$, and assume that $l$ connects two vertices $v_{1}$ and $v_{2}$. Denote by $p_{\eta}$ the corresponding node over the generic point. Again we have the section $u_{p}$ over $U_{\eta}$ as in  (\ref{equ:OverGenPoint}). By shrinking $U$, we can choose two regular sections $x$ and $y$ on $U_{\eta}$, which correspond to the coordinates of the two components meeting at $p_{\eta}$ in $C_{\eta}$. Choosing the coordinates appropriately, we may assume
\begin{equation}\label{equ:GenNodeDVR}
\beta_{\eta}(e_{l})=\log x + \log y, \mbox{\ \ in\ \ }U_{\eta}.
\end{equation}
Without loss of generality, we can assume that $u_{p}=x^{c}$, where $c$ is the contact order at $p_{\eta}$. Then $u_{p}$ vanishes along the component with coordinate $y$.

By taking the normalization of $U$ along the generic node given by $l$, we obtain two sub-schemes $U_{1}$ and $U_{2}$ of $U$. By shrinking $U$, we can assume that $U_{i}\subset C_{v_{i}}$ for $i=1,2$. We still use $x$ and $y$ to denote restriction of $x$ and $y$ to $U_{1}$ and $U_{2}$ respectively, and $p_{i}$ the pre-image of $p$ in $U_{i}$ for $i=1,2$. Then $x$ and $y$ can be viewed as a rational function on $U_{1}$ and $U_{2}$ respectively. Let $\Sigma_{1}$ and $\Sigma_{2}$ be the two sections in $U_{1}$ and $U_{2}$ respectively, coming from the splitting node $l$. Let $\sigma_{i}$ be the regular functions on $U_{i}$, whose vanishing gives the section $\Sigma_{i}$ for $i=1,2$.

\begin{lem}\label{lem:UnderGenNode}
With the notations as above, locally at $p_{i}$, we have
\begin{enumerate}
 \item $x=\pi^{n_{1}}\cdot \sigma_{1}\cdot h_{1}$, where $n_{1}=\nu_{\pi}(x)$, and $h_{1}\in \sO_{U_{1}}^{*}$.
 \item $y=\pi^{n_{2}}\cdot \sigma_{2}\cdot h_{2}$, where $n_{2}=\nu_{\pi}(y)$, and $h_{2}\in \sO_{U_{2}}^{*}$.
\end{enumerate}
\end{lem}
\begin{proof}
The proof of this is similar to that for Lemma \ref{lem:LocalUnderlying}. 
\end{proof}

\begin{rem}\label{rem:UnderGenNode}
Note that $\sigma_{1}$ and $\sigma_{2}$ form the coordinates at $p$. Choosing coordinates appropriately, we may assume that $h_{1} =h_{2} =1$ in Lemma \ref{lem:UnderGenNode}. Thus, we have $u_{p}=\pi^{c\cdot n_{1}}\cdot \sigma_{1}^{c}$.
\end{rem}

\subsection{Existence of the extension}\label{ss:ExistExt}
Now we consider the minimal log map $\xi_{\eta}$ and the extended underlying map $\underline{\xi}$. Denote by $C(\CharM_{\eta})$ the convex rational polyhedral cone of $\CharM_{\eta}$ in $\CharM_{\eta}^{gp}\otimes_{\Z}\Q$. Since $\CharM_{\eta}$ is sharp, the cone $C(\CharM_{\eta})$ is strongly convex. 

\begin{lem}
There is a lattice point $\tilde{v}\in \CharM_{\eta}^{gp}$ such that $(u,\tilde{v}) >0$ for any non-zero element $u\in C(\CharM_{\eta})$, where $(\cdot,\cdot)$ is the standard pairing in the Euclidean space $\CharM_{\eta}^{gp}\otimes_{\Z}\Q$. 
\end{lem}
\begin{proof}
This follows from \cite[Section 1.2(iv)]{Toric}. 
\end{proof}

We fix a lattice point $\tilde{v}$ satisfies the condition in the above lemma. The set 
\[\{(u,\tilde{v})\ |\ u\in C(\CharM_{\eta})\}\subset \Q\] 
forms a monoid, whose saturation is the rank one free monoid $\N$. Thus, we have a map of saturated monoids $l_{\tilde{v}}:\CharM_{\eta}\to \N$. Consider the log structure $\sM_{\eta}'$, associated to the pre-log structure $\N\to K, \mbox{\ \ }e\mapsto 0$ over $\eta$. We fix a chart $\beta_{\eta}:\CharM_{\eta}\to \sM_{\eta}$ and $\beta_{\eta}':\CharM_{\eta}'\cong \N\to \sM_{\eta}'$. Then we have a morphism of log structures $\sM_{\eta}\to \sM_{\eta}'$ given by 
\[\beta_{\eta}(e)\mapsto \beta_{\eta}'\circ l_{\tilde{v}}(e).\]
Denote by $\xi_{\eta}'=(C\to S,\sM_{\eta}',f'_{\eta})$ the log map obtained by pulling back $\xi_{\eta}$ via the map $(\eta,\sM_{\eta})\to(\eta,\sM_{\eta}')$. By Proposition \ref{prop:UnivMinLog}, it is enough to construct a log map (not necessarily minimal) $\xi'$, such that its generic fiber is given by $\xi_{\eta}'$ as above. 

\begin{lem}\label{lem:PositiveSpeDeg}
Using the notations as above, there exists a chart $\beta_{\eta}':\CharM_{\eta}'\to\sM_{\eta}'$, such that no components of $C$ over $\bar{s}$ have negative special degeneracy under $\beta_{\eta}'$ as in Observation \ref{obs:SpecialDeg}.
\end{lem}
\begin{proof}
We fix an arbitrary chart $\beta_{\eta}'$ as above. Consider an irreducible component $Z$ over the closed point $\bar{s}$. Let $p\in Z$ be a smooth non-distinguished point $p\in Z$. Consider the nearby points of $p$ over $\eta$. By Lemma \ref{lem:LocalUnderlying}, we have
\begin{equation}\label{equ:ChangeChart}
(f_{\eta}')^{\flat}(\delta_{})=\beta_{\eta}'(e)+\log \pi^{n}\cdot u,
\end{equation}
where $u$ is a locally invertible section near $p$, and $e\in \CharM_{\eta}'$. If $n\geq 0$, then there is nothing to prove. Consider the case $n< 0$. Since the number of irreducible components over $\bar{s}$ is finite, we can assume that $n$ is minimal among the special degeneracy of all irreducible components of the closed fiber under $\beta_{\eta}'$. Consider the new chart given by
\begin{equation}\label{equ:DefnNewChart}
\beta_{\eta}'': \CharM_{\eta}'\to\sM_{\eta}, \mbox{\ \ }e\mapsto \beta_{\eta}'(e)-n\cdot\log\pi .
\end{equation}
It is not hard to check that (\ref{equ:ChangeChart}) becomes
\[(f'_{\eta})^{\flat}(\delta_{})=\beta_{\eta}''(e)+\log u.\]
Since $n$ is minimal, by applying (\ref{equ:ChangeChart}) and (\ref{equ:DefnNewChart}) to other components, it follows that no irreducible component of $C$ over $\bar{s}$ has negative special degeneracy under $\beta_{\eta}''$. 
\end{proof}

We fix a chart $\beta_{\eta}':\CharM_{\eta}'\to\sM_{\eta}$, which satisfies the condition in Lemma \ref{lem:PositiveSpeDeg}. Consider the log structure $\sM_{S}'$ associated to the following pre-log structure on $S$:
\[\N^{2}\to R, \mbox{\ \ }e_{\eta}\mapsto 0, \mbox{\ and\ }e_{s}\mapsto \pi,\]
where $e_{\eta}$ and $e_{s}$ form the basis of $\N^{2}$. Now we identify $\sM_{S,\eta}'$ with $\sM_{\eta}'$, and the element $e_{\eta}$ corresponds to the chart $\beta_{\eta}':\CharM_{\eta}'\to\sM_{\eta}$.

\begin{lem}
With the notations as above, there is a 
 morphism of log structures $\sM_{S}^{C/S}\to \sM_{S}'$, whose restriction to the generic point $\eta$ is identical to the morphism of log structures $\sM_{\eta}^{C_{\eta}/\eta}\to \sM_{\eta}'$ given by $\xi_{\eta}'$.
\end{lem}
\begin{proof}
Possibly after a base change, we can choose a global chart $\beta^{C/S}:\CharM_{S,\bar{s}}^{C/S}\to \sM_{S}^{C/S}$. Denote by $G$ the dual graph of $C_{\bar{s}}$, and $\{e_{l}\}_{l\in E(G)}$ the set of generators of $\CharM_{S}$, such that $\beta^{C/S}(e_{l})$ is an element in $\sM_{S}^{C/S}$ smoothing the node corresponding to $l$ in the closed fiber. Assume that $l$ is smoothed out over $\eta$, then $\exp\circ\beta^{C/S}(e_{l})=\pi^{n}\cdot h$, where $n$ is an positive integer, and $h$ is an invertible element in $R$. Thus, we define 
\[e_{l}\mapsto n\cdot e_{s}+ \log h.\] 
If the node corresponding to $l$ persists over $\eta$, then we have 
\[e_{l}\mapsto n_{\eta}\cdot e_{\eta}+\log \pi^{n_{s}}+\log h, \mbox{\ \ over\ } \eta \]
where $n_{\eta}$ and $n_{s}$ are two integers, and $h$ is an invertible element in $R$. Note that $n_{\eta}$ is positive, and we may assume that $n_{s}$ is also positive by choosing a sufficiently large $n$ in (\ref{equ:DefnNewChart}). Thus, we define 
\[e_{l}\mapsto n_{\eta}\cdot e_{\eta}+ n_{s}\cdot e_{s}+\log h.\]
This induces a map $\sM_{S}^{C/S}\to \sM_{S}'$, whose restriction to the generic point coincides with $\sM_{\eta}^{C_{\eta}/\eta}\to \sM_{\eta}'$. 
\end{proof}

Note that the map $\sM_{S}^{C/S}\to \sM_{S}'$ in the above lemma gives a log pre-stable curves $(C\to S,\sM_{S}')$, whose restriction to $\eta$ is given by the log-prestable curve $(C_{\eta}\to \eta,\sM_{\eta}')$ of $\xi_{\eta}'$.  

\begin{prop}\label{prop:LimitExist}
There is a 
 log map $\xi'$ over $(S,\sM_{S}')$ with the log curve $(C\to S,\sM_{S}')$ and underlying map $\underline{\xi}$, whose restriction to $\eta$ is identical to $\xi_{\eta}'$.
\end{prop}
\begin{proof}
It is enough to define the morphism of log structures $f^{\flat}:\underline{f}^{*}\sM_{X}\to\sM_{C}'$, where $\sM_{C}'$ is the log structure on $C$ corresponding to the log curve $(C\to S,\sM_{S}')$. Pick a point $p\in C$ over $\bar{s}$, and an \'etale neighborhood $U$ of $p$. By shrinking $U$, we can assume that over the generic point, we have 
\[f^{\flat}_{\eta}(\delta_{})=n\cdot e_{\eta}+\log u_{p}, \mbox{\ \ in\ } U_{\eta},\]
where $u_{p}\in \sO_{U_{\eta}}$. 

We first assume that $p$ is not a generic node. By Lemma \ref{lem:LocalUnderlying}, further shrinking $U$ if necessary, the section $u_{p}$ extend to $U$ of the following form:
\[u_{p}=\pi^{n_{1}}\cdot h',\]
where $n_{1}=\nu_{\pi}(u_{p})$, and $h'\in \sO_{U}$. Note that the Lemma \ref{lem:PositiveSpeDeg} implies that the integer $n_{1}$ is non-negative. Thus, the only possible way to define $f^{\flat}$ near $p$ is given by
\[f^{\flat}(\delta_{})=n\cdot e_{\eta}+n_{1}\cdot e_{s}+\log h'.\]

Next we consider the case $p$ is a generic node. With the notations in Remark \ref{rem:UnderGenNode}, we have
\[u_{p}=\pi^{c\cdot n}\cdot \sigma_{1}^{c}.\]
Thus, we define 
\[f^{\flat}(\delta)=n\cdot e_{\eta}+c_{1}\cdot n_{1}\cdot e_{s}+c\cdot \log \sigma_{1}.\]

Note that our local construction is obtained by specializing the section $u_{p}$ to the closed fiber. Since the underlying structure is fixed, such specialization is unique. Thus, the above construction can be glued together to obtain a global map $f^{\flat}$ as we want. 
\end{proof}

\subsection{Specializing the dual graph}\label{ss:SpecGraph}
Consider the dual graph $\underline{G}$ of the underlying curve $C_{\bar{s}}$ of the fixed extension $\underline{\xi}$. For each edge $l\in E(\underline{G})$, if $l$ corresponds to a special node, then we can associate to $l$ a non-negative integer $c$ given by Lemma \ref{lem:LocalUnderlying}(3); if $l$ corresponds to a generic node, then we associate to $l$ the contact order given by $\xi_{\eta}$. Denote by $V_{n}(\underline{G})$ the set of non-degenerate components of $C_{\bar{s}}$. Note that Remark \ref{rem:SpecialOrder} gives an orientation on $\underline{G}$, which is compatible with the contact orders defined on each edge and the subset $V_{n}(\underline{G})\subset V(\underline{G})$. Thus, we obtain a marked graph. We use $G$ to denote this graph with the discrete data.

\begin{prop}\label{prop:UniGraph}
Consider any minimal log map $\xi$ over $S$ with the fixed underlying map $\underline{\xi}$, which is an extension of $\xi_{\eta}$. Then the marked graph $G_{\xi_{\bar{s}}}$ is identical to the graph $G$ with the orientation and contact orders as above.
\end{prop}
\begin{proof}
First note that the underlying graph of $G_{\xi_{\bar{s}}}$ and $G$ are both given by the dual graph $\underline{G}$ of the underlying curve, and their sets of non-degenerate vertices are the same. It is enough to check that the two graphs have the same contact orders and orientations. We denote the underlying graph to be $\underline{G}$. 

Consider an edge $l\in E(\underline{G})$. If $l$ corresponds to a generic node, then by Lemma \ref{lem:NodeContactOpen}, the orientation and contact order of $l$ is unquely determined by the generic fiber $\xi_{\eta}$. Hence the two graphs $G_{\xi_{\bar{s}}}$ and $G$ have the same orientation and contact orders along $l$. 

Next, consider the case where $l$ corresponds to a special node $p$. Assume that the contact order of $\xi$ at $p$ is $c$. Note that the log structure is compatible with the underlying structure. Hence the two graphs have the same weight $c$ in Lemma \ref{lem:LocalUnderlying}(3) associated to $l$, which only depends on $f_{\eta}$. By Lemma \ref{lem:UnderTangency}, the orientation of $l$ in $G_{\xi_{\bar{s}}}$ is given by the one described in Remark \ref{rem:SpecialOrder}. This implies that the two graphs $G_{\xi_{\bar{s}}}$ and $G$ have the same orientation and contact order along $l$.

This finishes the proof of the statement. 
\end{proof}

\begin{cor}
The graph $G$ is admissible.
\end{cor}
\begin{proof}
This follows from the existence of the extension of $\xi_{\eta}$ and the above proposition. 
\end{proof}

Consider a minimal log map $\xi=(C\to S,\sM_{S}, f)$, which is an extension of $\xi_{\eta}$ over $S$ with underlying map $\underline{\xi}$. Consider the natural map $q^{gen}:\CharM(G)^{gp}\cong\CharM_{S,\bar{s}}^{gp}\to \CharM_{\eta}^{gp}$. This is a surjection. Denote by $K_{sp}$ the kernel of $q^{gen}$. Then we have the following exact sequence:
\begin{equation}\label{diag:DecompChar}
0\lra K_{sp}\lra \CharM(G)^{gp}\stackrel{q^{gen}}{\longrightarrow} \CharM_{\eta}^{gp}\lra 0.
\end{equation}
Note that all groups involved in the exact sequence (\ref{diag:DecompChar}) are free abelian groups. We fix a non-canonical decomposition, which is compatible with (\ref{diag:DecompChar}):
\begin{equation}\label{equ:DecompChar}
\CharM(G)^{gp}=K_{sp}\oplus \CharM_{\eta}^{gp}.
\end{equation}
Denote by $q^{sp}: \CharM(G)^{gp}\to K_{sp}$ the natural projection. Then for any element $e\in \CharM(G)^{gp}$, we write $e=q^{gen}(e)+q^{sp}(e)$.

Possibly after a base change, we fix a global chart $\beta:\CharM(G)\cong\CharM_{S,\bar{s}}\to \sM_{S}$. Thus, we have a map of groups $\beta^{gp}:\CharM_{S,\bar{s}}^{gp}\to \sM_{S}^{gp}$. By \cite[3.5(i)]{LogStack}, the group $K_{sp}$ is generated by elements in $\CharM(G)$, whose images in $R$ is not zero. Consider the composition
\begin{equation}\label{equ:SpeDegMap}
\bar{\beta}:=\nu_{\pi}\circ\exp\circ\beta^{gp}: K_{sp}\to \Z,
\end{equation}
where $\nu_{\pi}$ is the evaluation of of the fraction field $K$. 

\begin{lem}\label{lem:SpecialEle}
The map $\bar{\beta}$ only depends on the base $S$, and the fixed underlying extension $\underline{\xi}$.
\end{lem}
\begin{proof}
Note that all other irreducible elements in $\CharM(G)$ can be expressed as some non-negative rational linear combinations of the irreducible elements lying on some extremal rays of $C(\CharM(G))$. It is enough to consider an irreducible element $e\in \CharM(G)$, which lies on an extremal ray of the cone $C(\CharM(G))$ in $\CharM(G)\otimes_{\Z}\Q$, such that its image in $R$ is non-zero. Without loss of generality, we can assume that $a=\pi^{n}$. By Lemma \ref{lem:IrrEle}, there is a minimal positive integer $n'$, such that $n'\cdot e$ is the element associated to some vertex or some special node $l$ in $G$. In the first case, the the minimal vertex is specialized from a non-degenerate component over $\eta$. Hence, Lemma \ref{lem:LocalUnderlying}(1) implies that the degeneracy $n\cdot n'$ is uniquely determined by the generic fiber and the base $S$. If $n'\cdot e$ is the element associated to a special node, then this is also determined by the generic fiber and the base $S$. This proves the statement. 
\end{proof}

Consider the map $\beta_{\eta}'$ given by the composition
\[\CharM_{\eta}\lra \CharM^{gp}_{\eta}\lra \CharM_{S,\bar{s}}^{gp}\stackrel{\beta^{gp}}{\lra}\sM_{S}^{gp},\]
where the middle arrow is the natural inclusion given by (\ref{equ:DecompChar}). Note that for any $e\in \CharM_{\eta}$, the element $\beta_{\eta}'(e)$ generalize to a unique element in $\sM_{\eta}$. Thus we obtain a chart for $\sM_{\eta}$. 

\begin{defn}\label{defn:Specializable}
A chart $\beta_{\eta}':\CharM_{\eta}\to\sM_{\eta}$ is called {\em specializable}, if it comes from a global chart $\beta:\CharM_{S,\bar{s}}\to \sM_{S}$ as above.
\end{defn}

\begin{rem}
The specializable chart can be viewed as a restriction of the chart $\beta:\CharM(G)\to \sM_{S}$ to the generic point. However, it depends on the choice of the non-canonical splitting (\ref{equ:DecompChar}). 
\end{rem}

For any element $e\in \CharM_{S,\bar{s}}$, consider the decomposition given by (\ref{equ:DecompChar}):
\[e=q^{sp}(e)+q^{gen}(e)=e^{sp}+e^{gen}.\]
By the construction of $\beta_{\eta}'$, we have
\begin{equation}\label{equ:SpeSpeDeg}
\beta(e)_{\eta}=\beta^{gp}(e^{sp})_{\eta}+\beta_{\eta}'(e^{gen}), \mbox{\ \ in \ }\sM_{\eta}.
\end{equation}
Note that $\exp\circ\beta^{gp}(e^{sp})\in K$, hence is an invertible element in $\sM_{\eta}$. 

\begin{lem}\label{lem:SpeSpeDeg}
For any $v\in V(G)$, the special degeneracy of $v$ with respect to $\beta_{\eta}'$ as in Observation \ref{obs:SpecialDeg} only depends on $\bar{\beta}$ and $q^{sp}$.
\end{lem}
\begin{proof}
It follows from (\ref{equ:SpeSpeDeg}) that the special degeneracy of $v$ with respect to $\beta_{\eta}'$ is given by $\bar{\beta}\circ q^{sp}(e_{v})$, where $e_{v}\in \CharM(G)$ is the element associated to $v$. 
\end{proof}

\subsection{Uniqueness of the extension}\label{ss:UniqueLimit}

Assume that we have two minimal extensions $\xi_{1}=(C\to S,\sM_{1},f_{1})$ and $\xi_{2}=(C\to S,\sM_{2},f_{2})$ of $\xi_{\eta}$, with the same underlying $\underline{\xi}$. After a base change, we can assume that we have two global charts
\begin{equation}\label{equ:TwoChart}
\beta_{1}: \CharM(G)\to \sM_{1} \mbox{\ \ \ and\ \ \ } \beta_{2}: \CharM(G)\to \sM_{2}.
\end{equation}
for $\xi_{1}$ and $\xi_{2}$ respectively.

\begin{lem}\label{lem:ChartDif}
For any $e\in \CharM(G)$, we have a unique element $u\in R^{*}$ such that
\[\beta_{1}(e)_{\eta}=\log u + \beta_{2}(e)_{\eta} \mbox{\ \ \ in\ }\sM_{\eta}.\]
Thus, we have a canonical isomorphism of log structures $\sM_{1}\cong \sM_{2}$.
\end{lem}
\begin{proof}
We only need to consider the irreducible elements of $\CharM(G)$. Let $e$ be an irreducible element of $\CharM(G)$. By Proposition \ref{prop:UniGraph}, we have 
\[\overline{\beta_{1}(e)_{\eta}}=\overline{\beta_{2}(e)_{\eta}} \mbox{\ \ \ in\ }\CharM_{\eta}.\]
Hence we have a unique element $u\in K$ such that 
\[\beta_{1}(e)_{\eta}=u\cdot \beta_{2}(e)_{\eta} \mbox{\ \ \ in\ }\sM_{\eta},\]
where $K$ is the fraction field of $R$. It remains to prove that $u$ is an invertible element in $R$.

First assume that $e$ lies on an extremal ray of $C(\CharM(G))$. By Lemma \ref{lem:IrrEle}, we have a minimal positive integer $n$ such that $n\cdot e\in N(G)$ is either the element associated to some edge, or the element associated to some minimal vertex.

Consider the case where $n\cdot e$ is the element associated to some edge $l$. We identify the element $e_{l}\in \sM_{S}^{C/S}$ smoothing $l$ with its image in $\sM_{1}$ or $\sM_{2}$. Then we have 
\[n\cdot \beta_{1}(e) +\log u_{1} = e_{l} \mbox{\ \ in \ }\sM_{1}, \mbox{\ \ and\ \ } n\cdot \beta_{2}(e) + \log u_{2} = e_{l} \mbox{\ \ in \ }\sM_{2},\]
where $u_{1},u_{2}\in R^{*}$. By restricting to the generic point $\eta$, we have
\[e_{l,\eta}= n\cdot \beta_{1}(e)_{\eta} +\log u_{1} = n\cdot \beta_{2}(e)_{\eta} +\log u_{2} \mbox{\ \ in \ }\sM_{\eta}.\]
This implies that $u^{n}=u_{2}/u_{1}$, hence $u\in R^{*}$.

Next, we consider the case where $n\cdot e$ is the element associated to some minimal vertex $v'\in V(G)$, and assume that $v'$ is specialized from $v\in V(G_{\eta})$. Denote by $e^{sp}=q^{sp}(e)$ and $e^{gen}=q^{gen}(e)$. Let $\beta_{i,\eta}$ be the specializable chart as in Definition \ref{defn:Specializable} induced by $\beta_{i}$ for $i=1,2$. 
By (\ref{equ:SpeSpeDeg}), we may assume that 
\[n\cdot \beta_{1}(e)_{\eta} =n_{1}\cdot\log \pi + n\cdot\beta_{1,\eta}'(e^{gen})\]
and
\[n\cdot \beta_{2}(e)_{\eta} =n_{2}\cdot\log \pi + n\cdot\beta_{2,\eta}'(e^{gen}).\]
Note that the special degeneracy of $v'$ with respect to $\beta_{i,\eta}'$ is given by $n_{i}$ for $i=1,2$. By Lemma \ref{lem:SpecialEle} and \ref{lem:SpeSpeDeg}, the special degeneracy of $v'$ does not depend on the choice of $\beta_{i}$. Thus we have $n_{1}=n_{2}$. By Lemma \ref{lem:SpecialDeg}, we obtain a unique element $u\in R^{*}$, such that 
\[\beta_{2,\eta}'(e^{gen})= \log u + \beta_{1,\eta}'(e^{gen}).\] 

Finally assume that $e$ does not lie on any extremal ray. Then for some sufficiently divisible positive integer $n$, we have
\[n\cdot e = \sum_{i} n_{i}e_{i},\]
where $n_{i}$ is a positive integer, and $e_{i}$ is an irreducible element lying on some extremal ray for each $i$. Then the above argument imples that there exists a unique $u_{i}\in R^{*}$ such that $\beta_{1}(e_{i})_{\eta}=\beta_{2}(e_{i})_{\eta} + \log u_{i}$ for each $i$. Thus, we have
\[n\cdot \beta_{1}(e_{})_{\eta}=n\cdot \beta_{2}(e_{})_{\eta} + \log h,\]
where $h=\prod_{i}u_{i}^{n\cdot n_{i}}\in R^{*}$. This implies that $u^{n}=h$, hence $u\in R^{*}$. 
\end{proof}

\begin{prop}\label{prop:LimitUnique}
Possibly after a base change, the isomorphism $\xi_{1,\eta}\cong \xi_{2,\eta}$ can be extended uniquely to an isomorphism of $\xi_{1}$ and $\xi_{2}$.
\end{prop}
\begin{proof}
For simplicity, we assume that $\xi_{1,\eta}= \xi_{2,\eta}=\xi_{\eta}$. We fix two global chart $\beta_{1}$ and $\beta_{2}$ as in (\ref{equ:TwoChart}). Denote by $\beta_{i,\eta}:\CharM_{\eta}\to \sM_{\eta}$ the specializable chart induced by $\beta_{i}$ for $i=1,2$. By Lemma \ref{lem:ChartDif}, we can identify $\sM_{1}$ and $\sM_{2}$. Thus the two chart $\beta_{1,\eta}$ and $\beta_{2,\eta}$ are identical. 

We first show that the following diagram of log structures commutes:
\[
\xymatrix{
 &\sM_{S}^{C/S}  \ar[ld]_{\psi_{1}} \ar[rd]^{\psi_{2}}&\\
\sM_{1} \ar@{=}[rr] && \sM_{2},
}
\]
where $\psi_{i}$ is the structure map defining the corresponding log curve of $\xi_{i}$. Since we put the standard log structure along non-distinguished nodes, we only need to consider a special distinguished node $p$ over the closed point. Let $e_{p}\in \sM_{S}$ be a section smoothing $p$. Then we have 
\[\psi_{1}(e_{p})=\psi_{2}(e_{p}) + \log u.\]
where $u$ is a unit in $R$. Since $\xi_{1,\eta}=\xi_{2,\eta}=\xi_{\eta}$, by restricting the above equation to the generic point $\eta$, we obtian $u=1$. This proves the commutativity. Thus, we can identify the two log curves of $\xi_{1}$ and $\xi_{2}$.

It remains to show that the two morphisms of log structures $f^{\flat}_{\xi_{i}}$ for $i=1,2$ are identical. Pick a point $p\in C$ over $\bar{s}$, then we need to prove that locally at $p$ we have
\begin{equation}\label{equ:CompLogMap}
f^{\flat}_{\xi_{1}}(\delta)=f^{\flat}_{\xi_{2}}(\delta).
\end{equation}
Since the two log maps $\xi_{1}$ and $\xi_{2}$ are minimal, locally at $p$ we have 
\[\bar{f}^{\flat}_{\xi_{1}}(\delta)=\bar{f}^{\flat}_{\xi_{2}}(\delta) \mbox{\ \ in\ }\CharM_{S}.\]
Thus, locally at $p$, there exists an invertible function $u$ such that
\[f^{\flat}_{\xi_{1}}(\delta)=f^{\flat}_{\xi_{2}}(\delta)+\log u.\]
Since $\xi_{1,\eta}=\xi_{2,\eta}$, by restricting to the generic fiber, we obtain $u=1$. This proves (\ref{equ:CompLogMap}) at $p$. Therefore, the statement of the proposition holds. 
\end{proof}

\subsection{Proof of Theorem 1.2.1 and finiteness}
Now we can give the proof of the main Theorem \ref{thm:Main}:

\begin{proof}
The boundedness is proved in Section \ref{sec:Boundedness}, and the weak valuative criterion is proved in Section \ref{sec:Valuative}. Since the stack has finite diagonal, it was shown in \cite[Theorem 2.7]{EHKV} that $\sK_{\Gamma}(X^{log})$ admits a finite surjective morphism from a scheme. With this property and the weak valuative criterion, by \cite[Proposition 7.12]{LMB} the stack is proper. The Deligne-Mumford property follows from Proposition \ref{prop:FiniteAuto}. 

The representability and finiteness of the map $\sK_{\Gamma}(X^{log})\to \sK_{g,n}(X,\beta)$ follow from Proposition \ref{prop:remove-log-rep}, and \ref{prop:quasi-finite}.
\end{proof}

Denote by $K_{\Gamma}(X^{log})$ and $K_{g,n}(X,\beta)$ the coarse moduli spaces of $\sK_{\Gamma}(X^{log})$ and $\sK_{g,n}(X,\beta)$ respectively. It follows from \cite[1.3 Corollary]{Keel-Mori} that $K_{\Gamma}(X^{log})$ exists and is proper. By the universal property of coarse moduli spaces, we have a natural map 
\begin{equation}\label{equ:coarse-map}
K_{\Gamma}(X^{log}) \to K_{g,n}(X,\beta)
\end{equation}
Again, since this arrow is quasi-finite, we have

\begin{cor}
The natural map (\ref{equ:coarse-map}) is finite.
\end{cor}

\appendix

\section{Prerequisites on logarithmic geometry}

\subsection{Basic definitions and properties}\label{ss:BacisLog}
Following \cite{KKato} and \cite{Ogus}, we recall some basic terminology from logarithmic geometry.

\subsubsection*{Monoids}

A \textit{monoid} is a commutative semi-group with a unit. We usually use $``+"$ and $``0"$ to denote the binary operation and the unit of a monoid. A \textit{morphism between two monoids} is required to preserve the unit. 

Let $P$ be a monoid, we can associate a group
\[P^{gp}:=\{(a,b) |(a,b)\sim(c,d) \ \mbox{if} \ \exists s \in P \ \mbox{such that} \ s+a+d=s+b+c\}.\]
We have the following terminologies:
\begin{enumerate}
 \item $P$ is called \textit{integral} if the natural map $P\to P^{gp}$ is injective. 
 \item $P$ is called \textit{saturated} if it is integral and satisfies that for any $p\in P^{gp}$, if $n\cdot p\in P$ for some positive integer $n$ then $p\in P$. 
 \item $P$ is \textit{coherent} if it is finitely generated.
 \item $P$ is \textit{fine} if it is integral and coherent.
 \item $P$ is \textit{fs} if it is fine and saturated.
 \item $P$ is \textit{sharp} if there are no other units except $0$. A nonzero element $p$ in a sharp monoid $P$ is called \textit{irreducible} if $p=a+b$ implies either $a=0$ or $b=0$. Denote by $Irr(P)$ the set of irreducible elements in a sharp monoid $P$. 
 \item A fine monoid $P$ is called \textit{free} if $P\cong \N^{n}$ for some positive integer $n$. 
 \item A monoid $P$ is called torsion free if the associated group $P^{gp}$ is torsion free. 
 \item The monoid $P$ is called toric if $P$ is fine, saturated, and sharp. Note that in this case $P$ is automatically torsion free.
 \item A morphism $h: Q\to P$ of two integral monoids is called \textit{integral} if for any $a_{1}, a_{2}\in Q$, and $b_{1},b_{2}\in P$ which satisfy $h(a_{1})b_{1}=h(a_{2})b_{2}$, there exist $a_{2}, a_{4}\in Q$ and $b\in P$ such that $b_{1}=h(a_{3})b$ and $a_{1}a_{3}=a_{2}a_{4}$. 
\end{enumerate}

Denote by $Mon^{int}$ and $Mon^{sat}$ the categories of integral and saturated monoids respectively. Then there is a natural inclusion 
\[\iota: Mon^{sat}\to Mon^{int}.\]
On the other hand, given an integral monoid $M$, the elements $a\in M^{gp}$, such that $m\cdot a\in M$ for some positive integer $m$ form a saturated submonoid $M^{sat}\subset M^{gp}$. This induces another functor
\[\sS at: Mon^{int}\to Mon^{sat}.\]

\begin{prop}\cite[Chapter I, 1.2.3(3)]{Ogus}\label{prop:AdjSat}
The functor $\sS at$ is left adjoint to the functor $\iota$.
\end{prop}

\subsubsection*{Logarithmic structures}

Let $X$ be a scheme. A \textit{pre-log structure} on $X$ is a pair $(\sM,\exp)$, which consists of a sheaf of monoids $\sM$ on the \'etale site $X_{\acute{e}t}$ of $X$, and a morphism of sheaves of monoids $\exp: \sM\to \sO_{X}$, called the \textit{structure morphism of $\sM$}. Here we view $\sO_{X}$ as a sheaf of monoid under multiplication. 

A pre-log structure $\sM$ on $X$ is called a \textit{log structure} if $\exp^{-1}(\sO_{X}^{*})\cong \sO^{*}_{X}$ via $\exp$. We sometimes omit the morphism $\exp$, and only use $\sM$ to denote the log structure if no confusion could arise. We call the pair $(X,\sM)$ a \textit{log scheme}. 

Given two log structures $\sM$ and $\sN$ on $X$, a \textit{morphism of the log structures} $h:\sM\to\sN$ is a morphism of sheaves of monoids which is compatible with the structure morphisms of $\sM$ and $\sN$.

Given a pre-log strucutre $\sM$ on $X$, we can associate a log structure $\sM^{a}$ given by 
\[\sM^{a}:=\sM\oplus_{\exp^{-1}(\sO^{*}_{X})}\sO_{X}^{*}.\]
Consider a morphism of schemes $f:X\to Y$, and a log structure $\sM_{Y}$ on $Y$. We can define the \textit{pull-back log structure} $f^{*}(\sM_{Y})$ to be the log structure associated to the pre-log structure 
\[f^{-1}(\sM_{Y})\to f^{-1}(\sO_{Y})\to \sO_{X}.\]

Consider two log schemes $(X,\sM_{X})$ and $(Y,\sM_{Y})$. A morphism of log schemes $(X,\sM_{X})\to (Y,\sM_{Y})$ is a pair $(f,f^{\flat})$, where $f:X\to Y$ is a morphism of the underlying schemes, and $f^{\flat}:f^{*}(\sM_{Y})\to \sM_{X}$ is a morphism of log structures on $X$. The morphism $(f,f^{\flat})$ is called \textit{strict} if $f^{\flat}$ is an isomorphism of log structures. It is called \textit{vertical} if $\sM_{X}/f^{*}(\sM_{Y})$ is a sheaf of groups under the induced monoidal operation. A standard example of log structures is the following:

\begin{eg}\label{eg:StandardLog}
Consider a normal crossing divisor $D$ on a smooth scheme $X$. Then the following 
\[\sM_{X}=\{\ g\in \sO_{X} \ |\ \mbox{$g$ is invertible outside $D$}\}\]
with the natural injection $\sM_{X}\to \sO_{X}$ forms a log structure on $X$.
\end{eg}

\subsubsection{Charts of log structures}

Let $(X,\sM)$ be a log scheme, and $P$ a monoid. Denote by $P_{X}$ the constant sheaf of monoid $P$ on X. A chart of $\sM$ is a morphism $P_{X}\to \sM$ such that the associated log structure of the composition $P_{X}\to \sM\to\sO_{X}$ is $\sM$. The log structure $\sM$ is called a \textit{fine log structure} on $X$ if $P$ is fine. If the monoid $P$ is fs, then $\sM$ is called a \textit{fs log structure}. We denote by $\LogSch$ the category of fine log schemes, and $\LogSch^{fs}$ the category of fs log schemes.

Let $\CharM=\sM/\sO_{X}^{*}$ be the quotient sheaf. We call it the \textit{characteristic of the log structure $\sM$}. It is useful to notice that $\overline{f^{*}(\sM)}=f^{-1}(\CharM)$ for any morphism of schemes $f: Y\to X$. A fine log structure $\sM$ is called locally free if for any $\bar{x}\in X$, we have $\CharM_{\bar{x}}\cong \N^{n}$ for some positive integer $n$. Let $\CharM_{\bar{x}}^{gp,tor}$ be the torsion part of $\CharM_{\bar{x}}^{gp}$. The following result is very useful for creating charts.

\begin{prop}\label{prop:ChartLogStr}\cite[2.1]{LogStack}
Using the notation as above, there exist an fppf neighborhood $f:X'\to X$ of $x$, and a chart $\beta:P\to f^{*}(\sM)$ such that for some geometric point $\bar{x}'\to X'$ lying over $x$, the natural map $P\to f^{-1}\CharM_{\bar{x}'}$ is bijective. If $\CharM_{\bar{x}}^{gp,tor}\otimes k(x)=0$, then such a chart exists in an \'etale neighborhood of $x$.
\end{prop}

\begin{rem}
If $\sM$ is a fs log structure on $X$, then the above proposition implies that there exists a section $\CharM_{\bar{x}}\to\sM_{\bar{x}}$, which can be lifted to a chart \'etale locally near $x$.
\end{rem}

Consider a morphism $f:(X,\sM_{X})\to(Y,\sM_{Y})$ of fine log schemes. A \textit{chart of $f$} is a triple $(P_{X}\to\sM_{X},Q_{Y}\to\sM_{Y},Q\to P)$ where $P_{X}\to \sM_{X}$ and $Q_{Y}\to \sM_{Y}$ are charts of $\sM_{X}$ and $\sM_{Y}$ respectively, and $Q\to P$ is a morphism of monoids such that the following diagram is commutative:
\[\xymatrix{
Q_{X} \ar[r] \ar[d] & P_{X} \ar[d] \\
f^{*}(\sM_{Y}) \ar[r] & \sM_{X}.
}
\]
Note that the charts of morphism of fine log schemes exist \'etale locally.

Consider a morphism of log schemes $f:(X,\sM_{X})\to(Y,\sM_{Y})$. With the help of charts, we can describe the log smoothness properties of $f$ due to K. Kato \cite[Theorem 3.5]{KKato}. The log map $f$ is called \textit{log smooth} if \'etale locally, there is a chart $(P_{X}\to\sM_{X},Q_{Y}\to\sM_{Y},Q\to P)$ of $f$ such that:
\begin{enumerate}
 \item Ker$({Q^{gp}\to P^{gp}})$ and the torsion part of $Coker(Q^{gp}\to P^{gp})$ are finite groups;
 \item the induced map $X\to Y\times_{\Spec(\Z[Q])}\Spec\Z[P]$ is smooth in the usual sense.
\end{enumerate}

The map $f$ is called \textit{integral} if for every $p\in X$, the induced map $\CharM_{f(\bar{p})}\to\CharM_{\bar{p}}$ is integral. In general, the underlying structure map of a log smooth morphism need not be flat. However, it is shown in \cite[4.5]{KKato} that the underlying map of a log smooth and integral morphism is flat. 

\subsection{Deligne-Faltings log structures}\label{ss:DFLog}
\begin{defn}\label{defn:DF}
Consider a scheme $X$. A locally free log structure $\sM_{X}$ on $X$ is called a \textit{Deligne-Faltings (DF) log structures}, if there is a morphism of locally constant sheaves of monoids $\beta:\N^{k}\to \CharM_{X}$, which locally lifts to a chart. We call the map $\beta$ a \textit{global presentation of $\sM_{X}$}.
\end{defn}

\begin{rem}\label{rem:LineBundleDF}
Consider a DF log structure $\sM_{X}$ with a global presentation $\N^{k}\to \CharM_{X}$. Denote by $\{\delta_{i}\}_{i=1}^{k}$ the standard generators of $\N^{k}$. Then locally we have a lifting $\tilde{\beta}:\N^{k}\to\sM_{X}$. Note that the section $\beta(\delta_{i})$ with its inverse image under the canonical map $\pi:\sM_{X}\to\CharM_{X}$ is a $\sO^{*}_{X}$-torsor, which corresponds to a line bundle $L_{i}$. The composition 
\[\pi^{-1}\big(\beta(\delta_{i})\big)\subset \sM_{X} \to \sO_{X}\]
gives a morphism of line bundles $s_{i}:L_{i}\to\sO_{X}$. In fact, it was shown in \cite[Complement 1]{KKato} that a locally free DF log structure as above is equivalent to the data consisting of a $k$-tuple of line bundles $(L_{i})_{i=1}^{k}$ and morphisms of line bundles $s_{i}: L_{i}\to \sO_{X}$ for each $i$.

Note that $s_{i}\in H^{0}(L_{i}^{\vee})$. Denote by $D_{i}\subset X$ the vanishing locus of $s_{i}$. Note that $D_{i}$ consists of the points where the image of $\delta_{i}$ in $\CharM_{X}$ is non-trivial. If $s_{i}$ is not a zero section, then $D_{i}$ is a Cartier divisor in $X$. If $s_{i}$ is a zero section, then $D_{i}=X$. Consider the sub-log structure $\sM_{X}^{g} \subset \sM_{X}$, which is given by the set of zero sections and the corresponding line bundles. We call $\sM_{X}^{g}$ the \textit{generic part of $\sM_{X}$}. Note that if $D_{i}=\emptyset$, then the sub-log structure generated by $\delta_{i}$ is trivial.
\end{rem}

\begin{eg}\label{eg:SmoothDiv}
Consider a smooth Cartier divisor $D\subset X$, and the log structure $\sM_{X}$ associated to $D$ defined in example \ref{eg:StandardLog}. Then $\sM_{X}$ forms a DF log structure on $X$, which corresponds to the line bundle $\sO_{X}(-D)$ and the natural inclusion $\sO_{X}(-D)\hookrightarrow \sO_{X}$.
\end{eg}

\subsection{Olsson's Log Stacks}\label{s:LogStack}
We follow \cite{LogStack} to introduce the algebraic stack parameterizing log structures. Let us fix a base scheme $S$, and consider an algebraic stack $\sX$ in the sense of \cite{Artin}, which means that 
\begin{enumerate}
 \item the diagonal $\sX\to\sX\times_{S}\sX$ is representable and of finite type; 
 \item there exists a surjective smooth morphism $X\to \sX$ from a scheme. 
\end{enumerate}
Now we can define a fine log structure $\sM_{\sX}$ on $\sX$ by repeating the definition of log structure on schemes in \ref{ss:BacisLog}, but using lisse-\'etale site instead of the \'etale site. We refer to \cite[Section 5]{LogStack} for details of log structure on Artin stacks.

For any $S$-scheme $T$, and an arrow $g:T\to \sX$, we obtain a fine log structure $g^{*}(\sM_{\sX})$ on the lisse-\'etale site $T_{lis\mbox{-}\acute{e}t}$ of $T$. It is shown in \cite[5.3]{LogStack} that such $g^{*}(\sM_{\sX})$ is isomorphic to a unique fine log structure on the \'etale site $T_{\acute{e}t}$ of $T$. By abuse of notations, we denote by $g^{*}(\sM_{\sX})$ the corresponding log structure on $T$. Thus, we define a functor from $\sX$ to $\LogSch_{S}$, by pulling back the log structure $\sM_{\sX}$. The stack $\sX$ associated with this functor is called a {\em log stack} in \cite{FKato}. A fine log scheme $(X,\sM_{X})$ can be naturally viewed as a log algebraic stack.

Consider the fibered category $\sL og_{(\sX,\sM_{\sX})}$ over $\sX$. Its objects are pairs 
\[(g:X\to\sX,g^{*}(\sM_{\sX})\to \sM_{X}),\] 
where $g$ is a map from scheme $X$ to $\sX$, and $g^{*}(\sM_{\sX})\to \sM_{X}$ is a morphism of fine log structures on $X$. An arrow 
\[\big(g:X\to\sX,g^{*}(\sM_{\sX})\to \sM_{X}\big)\longrightarrow\big(h:Y\to\sX,h^{*}(\sM_{\sX})\to \sM_{Y}\big)\] 
is a strict morphism of log schemes $(X,\sM_{X})\to (Y,\sM_{Y})$, such that 
\begin{enumerate}
 \item the underlying map $X\to Y$ is a morphism over $\sX$;
 \item the following diagram of log schemes commutes:
\[
\xymatrix{
(X,\sM_{X}) \ar[r] \ar[d] & (Y,\sM_{Y}) \ar[d] \\
\big(X,g^{*}(\sM_{\sX})\big) \ar[r] & \big(Y,h^{*}(\sM_{\sX})\big).
}
\]
\end{enumerate}

\begin{rem}
In fact, the stack $\sL og_{(\sX,\sM_{\sX})}$ parametrizes log structures over $(\sX,\sM_{\sX})$. An object $\big(g:X\to\sX,g^{*}(\sM_{\sX})\to \sM_{X}\big)$ as above can be viewed as a morphism of log stacks $(X,\sM_{X})\to(\sX,\sM_{\sX})$.
\end{rem}

\begin{thm}\cite[5.9]{LogStack}\label{thm:LogStack}
The fibered category $\sL og_{(\sX,\sM_{\sX})}$ is an algebraic stack locally of finite presentation over $\sX$.
\end{thm}

\section{Logarithmic curves and their stacks}
In this section, we introduce the notion of log pre-stable curves. We will prove that the stack $\fM_{g,n}^{pre}$ parameterizing log pre-stable curves of genus $g$ and $n$ marked points is an open substack of some Olsson's log stack as above, hence is algebraic in the sense of \cite[5.1]{Artin}. We refer to \cite{FKato}, \cite{Mochizuki}, and \cite{LogCurve} for more details of log structures on curves.

\subsection{The canonical log structure on pre-stable curves}\label{ss:CanCurveLog}
Consider the stack $\fM_{g,n}$ parameterizing genus $g$ pre-stable curves with $n$ marked points, and let $\fC_{g,n}$ be the universal family over $\fM_{g,n}$. Denote by $\{\Sigma_{i}:\fM_{g,n}\rightarrow \fC_{g,n}\}_{i=1}^{n}$ the $n$ sections. The boundary $\fM^{sing}_{g,n}\subset\fM_{g,n}$ which parametrizes singular curves is a divisor with normal crossings on $\fM_{g,n}$. Hence the boundary divisor induces a canonical log structure $\sM_{\fM_{g,n}}$ on $\fM_{g,n}$, which is defined using the smooth topology as in \cite{LogStack}. 

Note that each section $\Sigma_{i}$ corresponds to a smooth divisor on $\fC_{g,n}$. By Example \ref{eg:StandardLog}, we have a log structure $\sM^{\Sigma_{i}}$ associated to this section $\Sigma_{i}$. The pre-image of $\fM_{g,n}^{sing}$ also gives a normal crossing divisor in $\fC_{g,n}$, hence a log structure $\sM_{\fC_{g,n}}^{\sharp}$ on $\fC_{g,n}$. Consider the log structure
\[\sM_{\fC_{g,n}}:=\sM_{\fC_{g,n}}^{\sharp}\oplus_{\sO_{\fC_{g,n}}^{*}}\sum_{i}\sM^{\Sigma_{i}}.\]
We call it the canonical log structure on $\fC_{g,n}$. There is a natural log smooth, integral, vertical map 
\begin{equation}\label{equ:CanSemiMap}
(\fC_{g,n},\sM_{\fC_{g,n}}^{\sharp})\rightarrow (\fM_{g,n},\sM_{\fM_{g,n}}),
\end{equation}
whose underlying map is given by the family $\fC_{g,n}\rightarrow \fM_{g,n}$. By adding the log structure from the marked points, we have an induced log smooth, integral map 
\begin{equation}\label{equ:CanPreMap}
(\fC_{g,n},\sM_{\fC_{g,n}})\rightarrow (\fM_{g,n},\sM_{\fM_{g,n}}).
\end{equation}

Given any family $C\rightarrow S$ of usual pre-stable curves of genus $g$, with $n$ marked points, we have the following cartesian diagram:
\[
\xymatrix{
C \ar[r] \ar[d]_{\pi} & \fC_{g,n} \ar[d] \\
S \ar[r] & \fM_{g,n}.
}
\]
Denote by $\sM_{C}^{\sharp C/S}$, $\sM_{C}^{C/S}$ and $\sM_{C}^{\Sigma_{i}}$ the log structures on $C$, obtained by pulling back $\sM_{\fC_{g,n}}^{\sharp}$, $\sM_{\fC_{g,n}}$ and $\sM^{\Sigma_{i}}$ respectively. Let $\sM_{S}^{C/S}$ be the log structure on $S$ obtained by pulling back $\sM_{\fM_{g,n}}$. Note that $\sM_{C}^{\Sigma_{i}}$ is the log structure given by the section $\Sigma_{i}$. Now we have two canonical log maps obtained by pulling back (\ref{equ:CanSemiMap}) and (\ref{equ:CanPreMap}) respectively:
\begin{equation}\label{equ:CanSemiMapLocal}
(C,\sM_{C}^{\sharp C/S})\rightarrow (S,\sM_{S}^{C/S}),
\end{equation}
and
\begin{equation}\label{equ:CanPreMapLocal}
(C,\sM_{C}^{C/S})\rightarrow (S,\sM_{S}^{C/S}).
\end{equation}

\begin{lem}\label{lem:UnivCanLog}
For any pair of fine log structures $(\sM_{C}',\sM_{S})$ over the family of prestable curves $C\rightarrow S$, such that the log map $(C,\sM_{C}')\rightarrow (S,\sM_{S})$ is log smooth, proper, integral and vertical, we have a unique pair of maps $\sM_{C}^{\sharp C/S}\rightarrow \sM_{C}'$ and $\sM_{S}^{C/S}\rightarrow \sM_{S}$ fitting in the following cartesian diagram of fine log schemes:
\[
\xymatrix{
(C,\sM_{C}') \ar[d] \ar[r] & (C,\sM_{C}^{\sharp C/S}) \ar[d]\\
(S,\sM_{S}) \ar[r] & (S,\sM_{S}^{C/S}),
}
\]
\end{lem}
\begin{proof}
See \cite{LogCurve}, and \cite[2.7]{LogSS} for a proof.
\end{proof}
 
\subsection{Log curves}
With the description above, we are able to introduce the log structure on curves that we are interested in.

\begin{defn}\label{DefLogC1}
A map of fine log schemes $(C,\sM_{C})\rightarrow (S,\sM_{S})$ with sections $\{\Sigma_{i}\}_{i=1}^{n}$ is called a genus $g$ log curve with $n$-markings if 
\begin{enumerate}
 \item the family $C\rightarrow S$ with $\{\Sigma_{i}\}$ is the usual prestable curve of genus $g$ and $n$-markings;
 \item the log structure $\sM_{C}$ is of the form $\sM_{C}=\sM_{C}'\oplus_{\sO_{C}^{*}}(\sum_{j}\sM^{\Sigma_{j}}_{C})$;
 \item the log map $(C,\sM_{C})\rightarrow (S,\sM_{S})$ comes from a log smooth, integral vertical map $(C,\sM_{C}')\rightarrow (S,\sM_{S})$ plus the log structure $\sM^{\Sigma_{i}}$ given by the markings.
\end{enumerate}
\end{defn}

By Lemma \ref{lem:UnivCanLog}, we have an equivalent definition of log curves using the canonical log structure.
\begin{defn}\label{DefLogC2}
A genus $g$, log curve with $n$-marked points over a scheme $S$ is given by the following data $(C\rightarrow S,\{\Sigma\}_{i=1}^{n},\sM_{S}^{C/S}\rightarrow \sM_{S})$, where
 \begin{enumerate}
  \item $(C\rightarrow S,\{\Sigma\}_{i=1}^{n})$ is a usual family of pre-stable curves of genus $g$ with $n$-markings;
  \item $\sM_{S}^{C/S}\rightarrow \sM_{S}$ is a morphism of fine log structures.
 \end{enumerate}
\end{defn}
If no confusion could arise, we will use $(C\rightarrow S,\sM_{S})$ for the log curves in the definition for short, and denote by $\sM_{C}$ for the log structure on the curves in the above Definition \ref{DefLogC1}.

\begin{defn}\label{LogPreCurve}
A log curve $(C\rightarrow S,\sM_{S})$ is called log pre-stable if the log structure $\sM_{S}$ is fine and saturated.
\end{defn}

\begin{rem}\label{CurveOpen}
By \cite[5.26]{LogStack}, the condition that the base log structure $\sM_{S}$ is fine and saturated is an open condition on $S$. 
\end{rem}

\subsection{The stack of log curves}\label{ss:LogCurveStack}

\begin{defn}\label{IsoLogCurve}
Given two log curves $(C\rightarrow S, \sM_{S})$ and $(C'\rightarrow S,\sM'_{S})$ over $S$. Denote by $\sM_{C}$ and $\sM_{C'}$ the log structure on $C$ and $C'$ associated to the two log curves respectively. An isomorphism between the above two log curves is a pair $(\rho,\theta)$ such that
\begin{enumerate}
 \item $\theta: (S,\sM_{S}) \rightarrow (S,\sM_{S}')$ and $\rho: (C, \sM_{C}) \rightarrow (C',\sM_{C'})$ are isomorphisms of log schemes; 
 \item the underlying map $\underline{\theta}: S \rightarrow S$ is the identity, and $\underline{\rho}: C \rightarrow C'$ is an isomorphism of usual prestable curves over $S$; 
 \item the pair $(\rho,\theta)$ fit in the following commutative diagram:
\[
\xymatrix{
(C,\sM_{C}) \ar[r]^{\rho} \ar[d] & (C',\sM_{C'}) \ar[d]\\
(S,\sM_{S}) \ar[r]^{\theta}      & (S,\sM_{S}').
}
\]
\end{enumerate}
\end{defn}

Denote by $\fM^{log}_{g,n}$ the fibered category over $\C$ parametrizing log curves with the arrow defined above. In fact, we have 
\[\fM_{g,n}^{log}\cong \sL og_{(\fM_{g,n},\sM_{\fM_{g,n}})}.\] 
Thus, the fibered category $\fM^{log}_{g,n}$ forms an algebraic stack in the sense of \cite{Artin}. Denote by $\fM^{pre}_{g,n}$ the substack of $\fM^{log}_{g,n}$ parameterizing log prestable curves. Then by Remark \ref{CurveOpen}, we have the following:
\begin{cor}\label{LogCurveStack}
The fibered category $\fM^{pre}_{g,n}$ is an open substack in $\fM^{log}_{g,n}$, hence is algebraic.
\end{cor}

\subsection{The canonical log structure at nodes}\label{CanLog}
Note that the log structure $\sM_{\fM_{g,n}}$ is locally free. Consider a usual prestable curve $C\to S$. Then the canonical log structure $\sM_{S}^{C/S}$ is also locally free. For any closed point $s \in S$, we have 
\[\overline{\sM}_{S,\bar{s}}^{C/S}\cong \N^{m},\] 
where $m$ is a non-negative integer. 

Shrinking $S$ if necessary, by Proposition \ref{prop:ChartLogStr} we can choose a global chart $\overline{\sM}_{S,\bar{s}}^{C/S}\cong\N^{m}\rightarrow \sM_{S}^{C/S}$. Denote by $\{e_{i}\}_{i=1}^{m}$ the standard generators of $\N^{m}$. 

Consider a node point $p \in C_{\bar{s}}$ in the fiber. Then there is an \'etale neighborhood $U$ of $p$, which contains no other nodes and marked points. We have a special element $e_{j} \in \{e_{i}\}_{i=1}^{m}$, with the following chart:
\[
\xymatrix{
\N^{m-1}\oplus\N^{2} \ar[r]  & \sM_{C}^{C/S}|_{U}  \\
\N^{m-1}\oplus\N     \ar[r] \ar[u]^{(id,\Delta)}& \pi^{*}(\sM_{S}^{C/S})|_{U} \ar[u]_{\pi^{\flat}}.
}
\]
Here on the bottom, the monoids $\N^{m-1}$ and $\N$ are generated by $\{e_{i}\}_{i\neq j}$ and $e_{j}$ respectively, and on the top we assume that $a$ and $b$ are the standard generators of the monoid $\N^{2}$. The map $(id,\Delta)$ is given by the identity on $\N^{n-1}$ and the diagnonal map $\Delta: e_{j} \mapsto a+b$. 

\begin{con}
Consider a log curve $(C\to S,\sM_{S})$. For convenience, we identify $e_{j}$ with its image in $\CharM_{S}$, and call it {\em the element in $\CharM_{S}$ smoothing the node $p$}, or simply {\em the element smoothing $p$}. 
\end{con}

For each node $p_{i}$ over $s$, we fix an element $e_{i}\in \CharM_{S,\bar{s}}^{C/S}$ smoothing it. Let $Irr(\overline{\sM}_{S,\bar{s}}^{C/S})$ be the set of irreducible elements in the monoid $\overline{\sM}_{S,\bar{s}}^{C/S}$. In fact we have $\{e_{i}\}_{i=1}^{m}= Irr(\overline{\sM}_{S,\bar{s}}^{C/S})$, and a natural map:
\[s_{C_{\bar{s}}}:\{\mbox{nodes in }C_{\bar{s}}\}\rightarrow Irr(\overline{\sM}_{S,\bar{s}}^{C/S})\] 
given by $p_{i}\mapsto$ (the element $e_{i}$ smoothing $p_{i}$). It was shown in \cite{FKato} that this map is a one-to-one correspondence. This means that all nodes in the fiber are smoothed independently.

\begin{rem}\label{rem:Special}
The bijection $s_{C_{\bar{s}}}$ implies that the canonical log structures $(\sM_{S}^{C/S},\sM_{C}^{C/S})$ is special in the sense of \cite[2.6]{LogSS}.
\end{rem}

We give a local description of the relation between canonical log structure and the underlying structure at the nodes as in \cite[Section 3]{FKato}. Let $A$ be a local neotherian henselian ring, and $s$ an element in the maximal ideal $m_{A}$ of $A$. Let $R$ be the henselization of $A[x,y]/(xy-s)$ at the ideal generated by $x,y$ and $m_{A}$. We still use $x,y$ to denote the corresponding elements in $R$.

\begin{lem}\label{lem:UnderDesCurve}\cite[2.1]{FKato}
Given $x',y'\in R$ such that $x'y'\in A$ and $(x',y',m_{A})=(x,y,m_{A})$ (equality of ideals in $R$). Then there exist units $u_{x},u_{y}\in R^{*}$ with $u_{x}u_{u}\in A$ such that $x'=u_{x}x$ and $y'=u_{y}y$ (or $y'=u_{x}x$ and $x'=u_{y}y$).
\end{lem} 

Consider the local family $\Spec R\to \Spec A$, the canonical log structure $(\sM_{R},\sM_{A})$ is given by the following commutative diagram of pre-log structures.
\[
\xymatrix{
\N^{2} \ar[rr]^{(e_{1},e_{2})\mapsto (x,y)} && R \\
\N \ar[rr]^{e\mapsto s} \ar[u]^{\Delta} && A \ar[u],
}
\]
where $e_{1},e_{2}$ (resp. $e$) are the standard generators of $\N^{2}$ (resp. $\N$), and $\Delta: e\mapsto e_{1}+e_{2}$ is the diagonal map. For convenience, we sometimes use $\log x,\log y$ and $\log s$ denote the image of $e_{1},e_{2}$ and $e$ in the corresponding log structures.



\end{document}